\documentclass[final]{lms-modified}
\usepackage{amsfonts,amsmath,hyperref,color}

 \newtheorem{theorem}{Theorem}[section]
 
 \newtheorem{corollary}[theorem]{Corollary}
 \newtheorem{lemma}[theorem]{Lemma}
 
  \newtheorem{definition}[theorem]{Definition}
  
 \newtheorem{question}[theorem]{Question}
 \numberwithin{equation}{section}
\title[Norm of a Fourier multiplier]
{When does the norm of a Fourier multiplier\\
dominate its $L^\infty$ norm?}
\author{Alexei Karlovich and Eugene Shargorodsky}
\classno{42B15, 46E30}
\extraline{%
This work was partially supported by the Funda\c{c}\~ao para a
Ci\^encia e a Tecnologia (Portuguese Foundation for Science and Technology)
through the project UID/MAT/00297/2013 (Centro de Matem\'atica e
Aplica\c{c}\~oes).}
\begin{document}
\maketitle
\begin{abstract}
One can define Fourier multipliers on a Banach function space by using the
direct and inverse Fourier transforms on $L^2(\mathbb{R}^n)$ or by using
the direct Fourier transform on $S(\mathbb{R}^n)$ and the inverse one on
$S'(\mathbb{R}^n)$. In the former case, one assumes that the Fourier 
multipliers belong to $L^\infty(\mathbb{R}^n)$, while in the latter one this 
requirement may or may not be included  in the definition. We provide 
sufficient conditions for those definitions to coincide as well as examples 
when they differ. In particular, we prove that if a Banach function space 
$X(\mathbb{R}^n)$ satisfies a certain weak doubling property, then the space
of all Fourier multipliers $\mathcal{M}_{X(\mathbb{R}^n)}$ is continuously 
embedded into $L^\infty(\mathbb{R}^n)$ with the best possible embedding 
constant one. For weighted Lebesgue spaces $L^p(\mathbb{R}^n,w)$, the weak 
doubling property is much weaker than the requirement that $w$ is a 
Muckenhoupt weight, and our result implies that 
$\|a\|_{L^\infty(\mathbb{R}^n)}\le\|a\|_{\mathcal{M}_{L^p(\mathbb{R}^n,w)}}$
for such weights. This inequality extends the inequality for $n=1$ from  
\cite[Theorem~2.3]{BG98}, where it is attributed to J.~Bourgain. We show 
that although the weak doubling property is not necessary, it is quite sharp. 
It allows the weight $w$ in $L^p(\mathbb{R}^n,w)$ to grow at any 
subexponential rate. On the other hand, the space $L^p(\mathbb{R},e^x)$ has 
plenty of unbounded Fourier multipliers.
\end{abstract}
\section{Introduction}
Let $S(\mathbb{R}^n)$ and $S'(\mathbb{R}^n)$ denote the Schwartz spaces 
of rapidly
decreasing functions and of tempered distributions on $\mathbb{R}^n$,
respectively. The action of a distribution $a\in S'(\mathbb{R}^n)$ on a
function $u\in S(\mathbb{R}^n)$ is denoted by  $\langle a,u\rangle:=a(u)$.
A Fourier multiplier on $\mathbb{R}^n$ with symbol $a\in S'(\mathbb{R}^n)$
is defined as the operator $u \mapsto F^{-1}aFu$, where
\[
(Fu)(\xi):=\widehat{u}(\xi):=\int_{\mathbb{R}^n}f(x)e^{-ix\xi}\,dx
\]
is the Fourier transform of $u\in S(\mathbb{R}^n)$, $F^{-1}$ denotes the
inverse Fourier transform, and $x\xi$ denotes the scalar product of
$x,\xi\in\mathbb{R}^n$. We observe that since $u\in S(\mathbb{R}^n)$ and
$a\in S'(\mathbb{R}^n)$, the function $Fu$ belongs to the space
$S(\mathbb{R}^n)$ and $aFu$ is a tempered distribution. Thus $F^{-1}aFu$ is
well defined and it belongs to $S'(\mathbb{R}^n)$. In fact, we have
$F^{-1}aFu=(F^{-1}a)*u$, and therefore,
$F^{-1}aFu \in C_{\rm poly}^\infty(\mathbb{R}^n)$ (see, e.g.,
\cite[Theorem~2.3.20]{G14-classical} or \cite[Theorem~7.19(b)]{R91}).
Here and in what
follows $C_{\rm poly}^\infty(\mathbb{R}^n)$  denotes the set of all smooth
polynomially bounded functions, i.e., the set of all infinitely differentiable
functions $f:\mathbb{R}^n\to\mathbb{C}$ such that for every
$\alpha\in\mathbb{Z}_+^n$ there exist
$m_\alpha\in\mathbb{Z}_+:=\{0,1,2,\dots\}$ and $C_\alpha>0$ 
satisfying $|\partial_x^\alpha f(x)|\le C_\alpha(1+|x|)^{m_\alpha}$
for all $x\in\mathbb{R}^n$.
Thus, if $u\in S(\mathbb{R}^n)$ and $a\in S'(\mathbb{R}^n)$, then
$F^{-1}aFu$ is a regular tempered distribution, whose action on
$v\in S(\mathbb{R}^n)$ is evaluated as follows:
\[
\langle F^{-1}aFu,v\rangle=\int_{\mathbb{R}^n} (F^{-1}aFu)(x)v(x)\,dx
\quad\mbox{for all}\quad v\in S(\mathbb{R}^n).
\]

Let $C_0^\infty(\mathbb{R}^n)$ denote the space of all infinitely
differentiable functions on $\mathbb{R}^n$ with compact supports
and let $\mathcal{D}'(\mathbb{R}^n)$ be the space of distributions, that is,
the dual space of $C_0^\infty(\mathbb{R}^n)$. Suppose $X(\mathbb{R}^n)$ is a
Banach space continuously embedded into the space of distributions 
$\mathcal{D}'(\mathbb{R}^n)$. 
We say that a distribution $a \in S'(\mathbb{R}^n)$ belongs to the set
$\mathcal{M}_{X(\mathbb{R}^n)}$ of Fourier multipliers on $X(\mathbb{R}^n)$ if
\[
\|a\|_{\mathcal{M}_{X(\mathbb{R}^n)}}
:=
\sup\left\{\frac{\|F^{-1} aFu\|_{X(\mathbb{R}^n)}}{\|u\|_{X(\mathbb{R}^n)}}:
u \in \left(S(\mathbb{R}^n)\cap X(\mathbb{R}^n)\right) \setminus\{0\}\right\}
<\infty.
\]
Many authors adopt the following alternative definition of Fourier multipliers
(see, e.g., \cite[p.~368]{BK97}, \cite[p.~323]{BKS02}, \cite[p.~28]{D79}, 
\cite[p.~7]{EG77}, \cite[p.~199]{RSS11}). 
A function $a\in L^\infty(\mathbb{R}^n)$ is said to belong to the set
$\mathcal{M}^0_{X(\mathbb{R}^n)}$ of Fourier multipliers on $X(\mathbb{R}^n)$
if
\[
\|a\|_{\mathcal{M}^0_{X(\mathbb{R}^n)}}
:=
\sup\left\{
\frac{\|F^{-1} aFu\|_{X(\mathbb{R}^n)}}{\|u\|_{X(\mathbb{R}^n)}}:
u \in \left(L^2(\mathbb{R}^n)\cap X(\mathbb{R}^n)\right) \setminus\{0\}
\right\}<\infty.
\]
Here $F^{\pm 1}$ are understood as mappings on $L^2(\mathbb{R}^n)$.
Since $S(\mathbb{R}^n)\subset L^2(\mathbb{R}^n)$, it is clear that
\begin{equation}\label{eq:multiplier-embedding}
\mathcal{M}^0_{X(\mathbb{R}^n)} 
\subseteq 
\mathcal{M}_{X(\mathbb{R}^n)}\cap L^\infty(\mathbb{R}^n) 
\subseteq 
\mathcal{M}_{X(\mathbb{R}^n)}
\end{equation}
and
\begin{equation}\label{eq:multiplier-est}
\quad
\|a\|_{\mathcal{M}_{X(\mathbb{R}^n)}}
\le
\|a\|_{\mathcal{M}^0_{X(\mathbb{R}^n)}}.
\end{equation}
We feel that insufficient attention has been paid so far to the relationship
between the above classes of Fourier multipliers. In this paper, we confine
ourselves to  the Fourier multipliers acting on so-called
Banach function spaces, which are defined below, and provide sufficient 
conditions for equalities to hold in \eqref{eq:multiplier-embedding} 
(see Theorem~\ref{th:M0}, 
Subsection~\ref{subsec:bounded-L2-approximation}, 
and Theorem~\ref{th:main})
as well as examples when they do not
hold (see Theorem~\ref{th:two-classes-of-multipliers-are-different}).
We pay particular attention to the question of existence of a constant $D_X$
such that
$\|a\|_{L^\infty(\mathbb{R}^n)}
\le
D_X \|a\|_{\mathcal{M}_{X(\mathbb{R}^n)}}$.

Our initial motivation came from the following result that appeared in 
the paper by E.~Berkson and T.~A.~Gillespie \cite{BG98}, where it was attributed
to J.~Bourgain. A measurable function
$w:\mathbb{R}^n\to[0,\infty]$ is referred to as a weight if $0<w(x)<\infty$
a.e. on $\mathbb{R}^n$. The weighted Lebesgue space $L^p(\mathbb{R}^n,w)$,
$1\le p\le \infty$, is the set of all measurable complex-valued functions $f$ 
on $\mathbb{R}^n$ satisfying
\[
\|f\|_{L^p(\mathbb{R}^n,w)}:=\|fw\|_{L^p(\mathbb{R}^n)}<\infty.
\]
Recall that a weight $w:\mathbb{R}^n\to[0,\infty]$ belongs to the Muckenhoupt
class $A_p(\mathbb{R}^n)$, $1<p<\infty$, if
\[
\sup_Q
\left(\frac1{|Q|} \int_Q w^p(x)\, dx\right)^{1/p}
\left(\frac1{|Q|} \int_Q w^{-p'}(x)\, dx\right)^{1/p'} < \infty,
\quad
\frac{1}{p} +\frac{1}{p'} = 1,
\]
where the supremum is taken over all cubes $Q \subset \mathbb{R}^n$ with
sides parallel to the coordinate axes.  
\begin{theorem}[{(\cite[Theorem~2.3]{BG98})}]
\label{th:BGB}
Suppose that $1<p<\infty$ and $w\in A_p(\mathbb{R})$. Then there
exists a constant $D_{p,w}>0$ depending on $p$ and $w$ such that for all 
$a\in\mathcal{M}_{L^p(\mathbb{R},w)}\cap L^\infty(\mathbb{R)}$,
\[
\|a\|_{L^\infty(\mathbb{R})}\le D_{p,w}\|a\|_{\mathcal{M}_{L^p(\mathbb{R},w)}}.
\]
\end{theorem}
The proof of Theorem~\ref{th:BGB} relies on the deep result on a.e. convergence 
of Fourier integrals, that is, the transplanted version of the celebrated 
Carleson theorem on the a.e. convergence of Fourier series (see, e.g.,
\cite[Theorem~6.1.1]{G14-modern} or \cite{L04}). Theorem~\ref{th:BGB} was 
extended by the first author \cite[Theorem~1]{K15} to the case of weighted 
Banach function spaces $X(\mathbb{R},w)$, in which the Cauchy singular integral
operator (the Hilbert transform) is bounded.

In this paper, we provide a more elementary proof that the estimate
\[
\|a\|_ {L^\infty(\mathbb{R}^n)}\le \|a\|_{\mathcal{M}_{X(\mathbb{R}^n)}}
\]
holds with the (optimal) constant equal to $1$ for a large class of Banach 
function spaces $X(\mathbb{R}^n)$ and arbitrary $n\ge 1$. In particular, it holds 
for all weighted Lebesgue spaces $L^p(\mathbb{R}^n,w)$ with $1<p<\infty$ and
Muckenhoupt weights $w\in A_p(\mathbb{R}^n)$.

We need several definitions to state our main result.
The set of all Lebesgue measurable complex-valued functions on $\mathbb{R}^n$
is denoted by $\mathfrak{M}(\mathbb{R}^n)$. Let $\mathfrak{M}^+(\mathbb{R}^n)$
be the subset of functions  in $\mathfrak{M}(\mathbb{R}^n)$ whose values lie
in $[0,\infty]$. The characteristic function of a measurable set
$E\subset\mathbb{R}^n$ is denoted by $\chi_E$ and the Lebesgue measure of $E$
is denoted by $|E|$.

Following \cite[Chap.~1, Definition~1.1]{BS88}, a mapping
$\rho:\mathfrak{M}^+(\mathbb{R}^n)\to [0,\infty]$
is called a Banach function norm if, for all functions $f,g,f_j$
($j\in\mathbb{N}$) in $\mathfrak{M}^+(\mathbb{R}^n)$, for all constants
$a\ge 0$, and for all measurable subsets $E$ of $\mathbb{R}^n$, the
following properties hold:
\begin{eqnarray*}
{\rm (A1)} &\quad &
\rho(f)=0  \Leftrightarrow  f=0\ \mbox{a.e.},
\quad
\rho(af)=a\rho(f),
\quad
\rho(f+g) \le \rho(f)+\rho(g),\\
{\rm (A2)} &\quad &
0\le g \le f \ \mbox{a.e.} \ \Rightarrow \ \rho(g)
\le \rho(f)
\quad\mbox{(the lattice property)},\\
{\rm (A3)} &\quad &
0\le f_j \uparrow f \ \mbox{a.e.} \ \Rightarrow \
       \rho(f_j) \uparrow \rho(f)\quad\mbox{(the Fatou property)},\\
{\rm (A4)} &\quad &
|E|<\infty \Rightarrow \rho(\chi_E) <\infty,\\
{\rm (A5)} &\quad &
|E|<\infty \Rightarrow \int_E f(x)\,dx \le C_E\rho(f)
\end{eqnarray*}
with $C_E \in (0,\infty)$ that may depend on $E$ and $\rho$ but is
independent of $f$.

When functions differing only on a set of measure zero are identified, the
set $X(\mathbb{R}^n)$ of all functions $f\in\mathfrak{M}(\mathbb{R}^n)$ for
which $\rho(|f|)<\infty$ becomes a Banach space under the norm
\[
\|f\|_{X(\mathbb{R}^n)} :=\rho(|f|)
\]
and  under the natural linear space operations
(see \cite[Chap.~1, Theorems~1.4 and~1.6]{BS88}). It is  called a 
{\it Banach function space}. 

If $\rho$ is a Banach function norm, its associate norm $\rho'$ is defined on
$\mathfrak{M}^+(\mathbb{R}^n)$ by
\[
\rho'(g):=\sup\left\{
\int_{\mathbb{R}^n} f(x)g(x)\,dx \ : \
f\in \mathfrak{M}^+(\mathbb{R}^n), \ \rho(f) \le 1
\right\}.
\]
It is a Banach function norm itself \cite[Chap.~1, Theorem~2.2]{BS88}.
The Banach function space $X'(\mathbb{R}^n)$ determined by the Banach
function norm $\rho'$ is called the associate space (K\"othe dual) of
$X(\mathbb{R}^n)$. The Lebesgue space $L^p(\mathbb{R}^n)$, $1\le p\le\infty$,
is the archetypical example of Banach function spaces. Other classical
examples of Banach function spaces are Orlicz spaces, rearrangement-invariant
spaces, and variable Lebesgue spaces $L^{p(\cdot)}(\mathbb{R}^n)$.

Let $X(\mathbb{R}^n)$ be a Banach function space. We say that
$f\in X_{\rm loc}(\mathbb{R}^n)$ if $f\chi_E\in X(\mathbb{R}^n)$ for every
measurable set $E\subset\mathbb{R}^n$ of finite measure. 
If $w:\mathbb{R}^n\to[0,\infty]$ is a weight satisfying
$w\in X_{\rm loc}(\mathbb{R}^n)$ and
$1/w\in X_{\rm loc}'(\mathbb{R}^n)$, then
\[
X(\mathbb{R}^n,w):=
\{f\in\mathfrak{M}(\mathbb{R}^n)\ : \ fw\in X(\mathbb{R}^n)\}
\]
becomes a Banach function space when it is equipped with the norm
\[
\|f\|_{X(\mathbb{R}^n,w)}:=\|fw\|_{X(\mathbb{R}^n)},
\]
and $[X(\mathbb{R}^n,w)]'=X'(\mathbb{R}^n,w^{-1})$ (see
\cite[Lemma~2.4]{KS14}). It is clear that
if $w\in A_p(\mathbb{R}^n)$, then $w\in L_{\rm loc}^p(\mathbb{R}^n)$ and
$1/w\in L_{\rm loc}^{p'}(\mathbb{R}^n)$, whence $L^p(\mathbb{R}^n,w)$ is a 
Banach function space.

For $y\in\mathbb{R}^n$ and $R>0$, let
$B(y,R) := \{x \in \mathbb{R}^n : |x-y|< R\}$ be the open ball of radius
$R$ centered at $y$. 
\begin{definition}
We say that a Banach function space $X(\mathbb{R}^n)$ satisfies the weak 
doubling property if there exists a number $\tau > 1$  such that
\[
\liminf_{R \to \infty}
\left(\inf_{y \in \mathbb{R}^n}
\frac{\|\chi_{B(y,\tau R)}\|_{X(\mathbb{R}^n)}}
{\|\chi_{B(y, R)}\|_{X(\mathbb{R}^n)}}
\right)<\infty.
\]
\end{definition}
\begin{theorem}[(Main result)]\label{th:main}
Let $n\ge 1$ and $X(\mathbb{R}^n)$ be a Banach function space  satisfying the 
weak doubling property. If 
$a \in \mathcal{M}_{X(\mathbb{R}^n)}\subset S'(\mathbb{R}^n)$, then 
$a \in L^\infty(\mathbb{R}^n)$ and
\begin{equation}\label{eq:main}
\|a\|_{L^\infty(\mathbb{R}^n)}
\le
\|a\|_{\mathcal{M}_{X(\mathbb{R}^n)}}.
\end{equation}
The constant $1$ on the right-hand side of \eqref{eq:main} is best possible.
\end{theorem}
The paper is organized as follows. 
Section~\ref{sec:auxiliary} contains auxiliary results. For a Banach 
function space, we introduce the bounded $L^2$-approximation property
and the norm fundamental property, study relations between them, and give 
examples of Banach function spaces, which do not satisfy these properties.
Further, we prove a variant of a well-known lemma on approximation at Lebesgue
points, which is an important ingredient in the proof of our main result.

The weak doubling property is discussed in  Section~\ref{sec:doubling}. In 
particular, we prove that a Muckenhoupt-type condition $A_X$ implies the weak 
doubling property and show that weighted Banach function spaces 
$X(\mathbb{R},w_j)$, built upon a translation-invariant 
Banach function space $X(\mathbb{R})$ and exponential weights $w_1(x)=e^{cx}$ 
and $w_2(x)=e^{c|x|}$ with $c>0$, fail to have the weak doubling property. 
On the other hand, we also show that $Y(\mathbb{R}^n,w)$ satisfies the weak 
doubling condition for weights $w$ that can grow at any subexponential rate.

Section~\ref{sec:proof} contains the proof of our main result. We divide it
into two parts. The first part of the proof is developed for Fourier 
multipliers belonging to a weighted Lebesgue space 
$L_{1,\sigma}(\mathbb{R}^n)$. Our arguments at this step are similar to those 
used in the proof of \cite[Theorem~2.3]{BG98} with the important difference 
that we substitute the application of the theorem on a.e. convergence of 
Fourier integrals by a simpler lemma on approximation at Lebesgue points 
proved in Section~\ref{sec:auxiliary}. Further, we approximate
an arbitrary Fourier multiplier $a\in\mathcal{M}_{X(\mathbb{R}^n)}$ by
$a\ast\psi_\varepsilon\in C_{\rm poly}^\infty(\mathbb{R}^n)$
with suitably chosen functions $\psi_\varepsilon\in C_0^\infty(\mathbb{R}^n)$.
Note that $C_{\rm poly}^\infty(\mathbb{R}^n)$ is contained in 
$L_{1,\sigma}(\mathbb{R}^n)$ for some $\sigma\in\mathbb{R}$, which allows us 
to complete the proof of Theorem~\ref{th:main}. We conclude this section 
the proof of a multi-dimensional analogue of Theorem~\ref{th:BGB}. 

In Section~\ref{sec:optimality-doubling}, we discuss the optimality of the 
requirement of the weak doubling property in Theorem~\ref{th:main}. 
In particular, we show that for an arbitrary translation-invariant Banach 
function space $Y(\mathbb{R})$ and the weight $w_1(x)=e^{cx}$  with any $c>0$,
the weighted Banach function space $Y(\mathbb{R},w_1)$ admits many 
unbounded Fourier  multipliers.

In Section~\ref{sec:alternative-definition}, we discuss the classes of Fourier
multipliers $\mathcal{M}^0_{X(\mathbb{R}^n)}$ and 
$\mathcal{M}_{X(\mathbb{R}^n)}\cap L^\infty(\mathbb{R}^n)$
and prove that they coincide if $X(\mathbb{R}^n)$ satisfies the bounded 
$L^2$-approximation property. We also construct 
an example showing that $\mathcal{M}^0_{X(\mathbb{R}^n)}$ and
$\mathcal{M}_{X(\mathbb{R}^n)} \cap L^\infty(\mathbb{R}^n)$ may differ.
We show that the classes $\mathcal{M}_{X(\mathbb{R}^n)}^0$ and
$\mathcal{M}_{X(\mathbb{R}^n)}\cap L^\infty(\mathbb{R}^n)$
are normed algebras and that the normed space $\mathcal{M}_{X(\mathbb{R}^n)}$
and the normed algebra $\mathcal{M}_{X(\mathbb{R}^n)}^0$ are not complete, 
in general.

The weak doubling property is of course by no means necessary for the 
conclusion of Theorem~\ref{th:main} to hold. Using duality and 
interpolation as in \cite{H60} (see also \cite[Lemma~6]{BT94}), one can prove 
the estimate \eqref{eq:main}
for arbitrary reflexive reflection-invariant Banach function spaces. This is 
done in Section~\ref{sec:reflection-invariant} with the help of the 
interpolation theorem for Calder\'on products  
$(X_0^{1-\theta}X_1^\theta)(\mathbb{R}^n)$ and Lozanovski{\u{\i}}'s formula
$(X^{1/2}(X')^{1/2})(\mathbb{R}^n)=L^2(\mathbb{R}^n)$.
Here we do not assume that the space $X(\mathbb{R}^n)$ satisfies the weak 
doubling property. We also show that the estimate 
\[
\|a\|_{L^\infty(\mathbb{R}^n)}\le\|a\|_{\mathcal{M}^0_{X(\mathbb{R}^n)}}
\quad\mbox{for all}\quad
a\in\mathcal{M}_{X(\mathbb{R}^n)}^0
\]
holds if $X(\mathbb{R}^n)$ is an arbitrary,  not necessarily reflexive,
reflection-invariant Banach function space.

In Section~\ref{subsec:Lofstrom}, we extend J. L\"ofstr\"om's result \cite{L83}
and show that there are no non-trivial Fourier multipliers on the weighted 
Banach function space $Y(\mathbb{R}^n, w)$ built upon a translation-invariant 
Banach function space $Y(\mathbb{R}^n)$ in the case of a weight $w$ growing 
superexponentially in all directions: 
$\mathcal{M}_{Y(\mathbb{R}^n, w)} = \mathbb{C}$. 
On the other hand, we show that there are non-trivial Fourier multipliers 
in $\mathcal{M}_{Y(\mathbb{R}^n,w)}$ in the case of (sub)exponentially 
growing weights like $w(x)=\exp(c|x|^\alpha)$ for $x\in\mathbb{R}^n$ with 
some constants $c>0$ and $\alpha\in(0,1]$.
\section{Auxiliary results}
\label{sec:auxiliary}
\subsection{Translation-invariant Banach function spaces}
We say that a Banach function space $X(\mathbb{R}^n)$ is 
{\it translation-invariant} if 
for all $y \in \mathbb{R}^n$ and 
for all functions $u \in X(\mathbb{R}^n)$, one has
\[
\|\tau_y u\|_{X(\mathbb{R}^n)} = \|u\|_{X(\mathbb{R}^n)}, 
\]
where the translation operator $\tau_y$ is defined by
$(\tau_y u)(x) := u(x - y)$ for all $x \in \mathbb{R}^n$.
\begin{lemma}\label{le:TI-X-and-Xprime}
Let $X(\mathbb{R}^n)$ be a Banach function space and $X'(\mathbb{R}^n)$ be
its associate space. Then $X(\mathbb{R}^n)$ is translation-invariant if and 
only  if $X'(\mathbb{R}^n)$ is translation-invariant.
\end{lemma}
\begin{proof}
Suppose $X(\mathbb{R}^n)$ is translation-invariant, $g\in X'(\mathbb{R}^n)$,
and $y\in\mathbb{R}^n$. Then for every $f\in X(\mathbb{R}^n)$ with
$\|f\|_{X(\mathbb{R}^n)}\le 1$, we have 
$\|\tau_{-y}f\|_{X(\mathbb{R}^n)}=\|f\|_{X(\mathbb{R}^n)}\le 1$. 
By H\"older's inequality (see \cite[Chap.~1, Theorem~2.4]{BS88}),
$g(\tau_{-y}f)\in L^1(\mathbb{R}^n)$. Changing variables, we get
\[
\int_{\mathbb{R}^n}g(x)(\tau_{-y}f)(x)\,dx
=
\int_{\mathbb{R}^n}(\tau_y g)(x)f(x)\,dx.
\]
Then, in view of \cite[Chap.~1, Lemma~2.8]{BS88},
\begin{align*}
\|\tau_y g\|_{X'(\mathbb{R}^n)}
&=
\sup\left\{
\left|\int_{\mathbb{R}^n}(\tau_y g)(x)f(x)\,dx\right|
\ :\ f\in X(\mathbb{R}^n),\
\|f\|_{X(\mathbb{R}^n)}\le 1
\right\}
\\
&=
\sup\left\{
\left|\int_{\mathbb{R}^n}g(x)f(x)\,dx\right|
\ :\ f\in X(\mathbb{R}^n),\
\|f\|_{X(\mathbb{R}^n)}\le 1
\right\}
=
\|g\|_{X'(\mathbb{R}^n)},
\end{align*}
that is, $X'(\mathbb{R}^n)$ is translation-invariant. The reverse implication
follows from what was proved above and the Lorentz-Luxemburg theorem
(see \cite[Chap.~1, Theorem~2.7]{BS88}).
\end{proof}
\subsection{The bounded $L^2$-approximation property}
\label{subsec:bounded-L2-approximation}
\begin{definition}
We will say that a Banach function space $X(\mathbb{R}^n)$ satisfies
the bounded $L^2$-approximation property if for every function
$u \in L^2(\mathbb{R}^n)\cap X(\mathbb{R}^n)$, there 
exists a sequence $\{u_j\}_{j\in\mathbb{N}} \subset C^\infty_0(\mathbb{R}^n)$
such that
\begin{equation}\label{eq:bounded-L2-approximation}
\lim_{j \to \infty}\|u - u_j\|_{L^2(\mathbb{R}^n)} = 0 , 
\quad
\limsup_{j \to \infty}\|u_j\|_{X(\mathbb{R}^n)} \le  \|u\|_{X(\mathbb{R}^n)}.
\end{equation}
\end{definition}
Following \cite[Chap.~1, Definition~3.1]{BS88}, a
function $f$ in a Banach function space $X(\mathbb{R}^n)$ is said to
have {\it absolutely continuous norm} in $X(\mathbb{R}^n)$ if 
$\|f\chi_{E_j}\|_{X(\mathbb{R}^n)}\to 0$ as $j\to\infty$ for every
sequence $\{E_j\}_{j\in\mathbb{N}}$ of measurable sets in $\mathbb{R}^n$
satisfying $\chi_{E_j}\to 0$ a.e. on $\mathbb{R}^n$ as $j\to\infty$.
The set of all functions of absolutely continuous norm in $X(\mathbb{R}^n)$
is denoted by $X_a(\mathbb{R}^n)$. If $X_a(\mathbb{R}^n)=X(\mathbb{R}^n)$,
then the space $X(\mathbb{R}^n)$ itself is said to have absolutely continuous
norm.
\begin{theorem}\label{th:AC-implies-bounded-L2-approximation}
Let $X(\mathbb{R}^n)$ be a Banach function space with absolutely continuous
norm. Then $X(\mathbb{R}^n)$ has the bounded $L^2$-approximation property.
\end{theorem}
\begin{proof}
It is clear that $L^2(\mathbb{R}^n)\cap X(\mathbb{R}^n)$ 
is a Banach function space when it is equipped with the norm 
\[
\|f\|_{L^2(\mathbb{R}^n)\cap X(\mathbb{R}^n)}
=
\max\left\{
\|f\|_{L^2(\mathbb{R}^n)},\|f\|_{X(\mathbb{R}^n)}
\right\}.
\]
It is easy to see that it has absolutely continuous norm. 
Then for every $u \in L^2(\mathbb{R}^n)\cap X(\mathbb{R}^n)$, there 
exists a sequence $\{u_j\}_{j\in\mathbb{N}} \subset C^\infty_0(\mathbb{R}^n)$
such that 
\[
\lim_{j \to \infty}\|u - u_j\|_{L^2(\mathbb{R}^n)\cap X(\mathbb{R}^n)} = 0 
\]
(a proof for the case $n=1$ can be found in \cite[Lemma~2.10(b)]{KS14}, it 
can be easily extended to $n\in\mathbb{N}$).  Hence 
\eqref{eq:bounded-L2-approximation} holds.
\end{proof}
Now let $\varrho(x)=e^{1/(|x|^2-1)}$ if $|x|<1$ and $\varrho(x)=0$
if $|x|\ge 1$. Consider the sequence 
\begin{equation}\label{eq:mollification}
\varrho_j(x)
:=
\frac{j^n\varrho(xj)}{\int_{\mathbb{R}^n}\varrho(x)\,dx},
\quad j\in\mathbb{N}.
\end{equation}
As usual, let $\operatorname{supp}u$ denote the support of a function
$u\in\mathfrak{M}(\mathbb{R}^n)$.
\begin{theorem}\label{th:TI-implies-bounded-L2-approximation}
Let $Y(\mathbb{R}^n)$ be a translation-invariant Banach function space 
and let $w$ be a continuous function such that $w(x)>0$ for all 
$x\in\mathbb{R}^n$. Then $Y(\mathbb{R}^n, w)$ has the bounded 
$L^2$-approximation property.
\end{theorem}
\begin{proof}
Take any function $u \in L^2(\mathbb{R}^n)\cap Y(\mathbb{R}^n, w)$ and any 
$\varepsilon > 0$. There exists $R > 0$ such that the function
$v := \chi_{B(0, R)} u$ satisfies 
\[
\|u - v\|_{L^2(\mathbb{R}^n)}  < \varepsilon/2 .
\]
It is clear that $\operatorname{supp} v \subseteq B(0, R)$ and 
$\|v\|_{Y(\mathbb{R}^n, w)} \le \|u\|_{Y(\mathbb{R}^n, w)}$
in view of axiom (A2).

Since $w > 0$ is continuous, there exists $j_0 \in \mathbb{N}$ such that
for all $x \in B(0, R)$ and $y \in B(0,1/j_0)$,
\[
\frac{(\tau_{-y} w)(x)}{w(x)} = \frac{w(x + y)}{w(x)} \le 
1 + \varepsilon.
\]
Then taking into account that $Y(\mathbb{R}^n)$ is translation-invariant, 
one gets for all $y \in B(0, 1/j_0)$,
\begin{align}
\left\|\tau_y v\right\|_{Y(\mathbb{R}^n, w)} 
&=
\left\|w\tau_y v\right\|_{Y(\mathbb{R}^n)} 
= 
\left\|\tau_y\Big((\tau_{-y} w) v\Big)\right\|_{Y(\mathbb{R}^n)} 
=
\|(\tau_{-y} w) v\|_{Y(\mathbb{R}^n)} 
\nonumber \\
&\le 
(1 + \varepsilon)\|wv\|_{Y(\mathbb{R}^n)} 
= 
(1 + \varepsilon)\|v\|_{Y(\mathbb{R}^n, w)}.
\label{eq:TI-implies-bounded-L2-approximation}
\end{align}
Let  $v_j := \varrho_j\ast v$, where $\varrho_j \in C_0^\infty(\mathbb{R}^n)$
are the functions defined by \eqref{eq:mollification}. Then
$v_j \in C_0^\infty(\mathbb{R}^n)$ and one can choose $j_1 \ge j_0$ such that 
for all $j \ge j_1$,
\[
\|v - v_j\|_{L^2(\mathbb{R}^n)}  < \varepsilon/2, 
\]
(see, e.g., \cite[Theorem~4.22]{B11}). Hence, for all $j \ge j_1$,
\[
\|u - v_j\|_{L^2(\mathbb{R}^n)}  < \varepsilon.
\]
Since $w \in Y_{\mathrm{loc}}(\mathbb{R}^n)$ and 
$1/w  \in Y'_{\mathrm{loc}}(\mathbb{R}^n)$, \
$Y(\mathbb{R}^n,w)$ is a Banach function space and $Y'(\mathbb{R}^n,w^{-1})$
is its associate space in view of \cite[Lemma~2.4]{KS14}.
Using  H\"older's inequality for Banach function spaces
(see \cite[Chap. 1, Theorem 2.4]{BS88}) and 
\eqref{th:TI-implies-bounded-L2-approximation}, one gets
for all  $g \in Y'(\mathbb{R}^n, w^{-1})$ and all $j \ge j_1$,
\begin{align*}
\left|\int_{\mathbb{R}^n} v_j(x) g(x)\, dx\right| 
&\le 
\int_{\mathbb{R}^n}
\left(\int_{\mathbb{R}^n} |\varrho_j(y)| |v(x - y)|\, dy\right) 
|g(x)|\, dx 
\\
&=
\int_{\mathbb{R}^n} |\varrho_j(y)|
\left(\int_{\mathbb{R}^n}  |v(x - y)| |g(x)|\, dx\right)\, dy 
\\
&\le 
\int_{B(0, 1/j)} |\varrho_j(y)|
\left\|\tau_y v\right\|_{Y(\mathbb{R}^n, w)} 
\left\|g\right\|_{Y'(\mathbb{R}^n, w^{-1})}\, dy
\\
&\le  
(1 + \varepsilon) \left\|v\right\|_{Y(\mathbb{R}^n, w)} 
\left\|g\right\|_{Y'(\mathbb{R}^n, w^{-1})}
\int_{\mathbb{R}^n} |\varrho_j(y)| \, dy \\
& = 
(1 + \varepsilon) \left\|v\right\|_{Y(\mathbb{R}^n, w)} 
\left\|g\right\|_{Y'(\mathbb{R}^n, w^{-1})}.
\end{align*}
By \cite[Chap.~1, Theorem~2.7 and Lemma~2.8]{BS88},
the above inequality implies that for all $j \ge j_1$,
\begin{align*}
\|v_j\|_{Y(\mathbb{R}^n,w)}
&=
\sup\left\{\left|\int_{\mathbb{R}^n} v_j(x) g(x)\, dx\right| : \
g\in Y'(\mathbb{R}^n,  w^{-1}), \
\|g\|_{Y'(\mathbb{R}^n,  w^{-1})}\le 1\right\}  
\\
&\le
(1 + \varepsilon) \left\|v\right\|_{Y(\mathbb{R}^n, w)},
\end{align*}
which completes the proof, since $\varepsilon > 0$ is arbitrary.
\end{proof}
Theorem~\ref{th:L-infinity-with-exotic-weight-fails-bounded-L2-approximation} 
below shows that one cannot drop the requirement of continuity of the weight 
$w$ in  Theorem~\ref{th:TI-implies-bounded-L2-approximation} 
(as well as in \cite[Lemma~2]{BT94}).
The construction of our counterexample is based on the fact that there exist
compact sets of positive Lebesgue measure with empty interior 
(see, e.g., \cite[Example~1.7.6]{B07} or 
\cite[Chap.~12, Exercise~9]{D12}).
\begin{lemma} \label{le:exotic-weight}
Let $G \subset \mathbb{R}^n$ be a compact set of positive measure with empty
interior and let
\begin{equation}\label{eq:exotic-weight}
w_G(x) := \left\{\begin{array}{cl}
    1,  & x \in G ,  \\
     2, &   x \in \mathbb{R}^n\setminus G .
\end{array}\right.
\end{equation}
Suppose $\psi \in C(\mathbb{R}^n)$ and 
$\|\psi\|_{L^\infty(\mathbb{R}^n, w_G)} \le 1$. Then 
for all $x \in \mathbb{R}^n$,
\begin{equation}\label{eq:exotic-weight-2}
|\psi(x)| \le 1/2.
\end{equation}
\end{lemma}
\begin{proof} 
For every point $x \in \mathbb{R}^n$ and $\varepsilon > 0$
there exists $\delta > 0$ such that $|\psi(y)| \ge |\psi(x)| - \varepsilon$
for each $y \in B(x, \delta)$.
Since $G$ is a closed set with empty interior, $B(x, \delta)\setminus G$ is a 
nonempty open set. It follows from \eqref{eq:exotic-weight} and the condition 
$\|\psi\|_{L^\infty(\mathbb{R}^n, w_G)} \le 1$ that 
$2|\psi(y)| \le 1$ for almost all $y \in B(x, \delta)\setminus G$.
Hence $|\psi(x)| - \varepsilon \le 1/2$  for all $\varepsilon > 0$. 
Passing in this inequality to the limit as $\varepsilon \to 0$, we arrive at 
\eqref{eq:exotic-weight-2}.
\end{proof}
\begin{theorem} 
\label{th:L-infinity-with-exotic-weight-fails-bounded-L2-approximation}
Let $G \subset \mathbb{R}^n$ be a compact set of positive measure with empty
interior and let the weight $w_G$ be defined by \eqref{eq:exotic-weight}. Then 
the Banach function space $L^\infty(\mathbb{R}^n,w_G)$ does not 
satisfy the bounded $L^2$-approximation property. 
\end{theorem}
\begin{proof} 
Let  $u := \frac34\, \chi_G$. Then 
$u \in L^2(\mathbb{R}^n)\cap L^\infty(\mathbb{R}^n, w_G)$
and $\|u\|_{L^\infty(\mathbb{R}^n, w_G)} \le \frac34$. It follows 
from Lemma~\ref{le:exotic-weight} that if $\psi \in C(\mathbb{R}^n)$ and 
$\|\psi\|_{L^\infty(\mathbb{R}^n, w_G)} \le 1$, then
\[
\|u - \psi\|_{L^2(\mathbb{R}^n)} 
\ge 
\left(\int_{G} |u(x) - \psi(x)|^2\, dx\right)^{1/2} 
\ge 
\left(\int_{G} \left(\frac34 - \frac12\right)^2\, dx\right)^{1/2}
\ge 
\frac14\, |G|^{1/2} .
\]
This inequality implies that there is no sequence 
$\{u_j\}_{j\in\mathbb{N}}\subset C_0^\infty(\mathbb{R}^n)$ such that
\eqref{eq:bounded-L2-approximation} is fulfilled for 
$X(\mathbb{R}^n)=L^\infty(\mathbb{R}^n,w_G)$.
\end{proof}
For a set $F\subset\mathbb{R}^n$, we denote by $F^\ast$ the closure of the set
$\{x+y\in\mathbb{R}^n\ :\ x\in F,\ y\in B(0,1)\}$.

The next result is well known. its proof is included for the reader's
convenience.
\begin{lemma}\label{le:approximation-L-infinity}
For every function $f\in L^\infty(\mathbb{R}^n)$ with compact support
there exists a sequence 
$\{v_j\}_{j\in\mathbb{N}}\subset C_0^\infty(\mathbb{R}^n)$ such that
$\operatorname{supp}v_j\subseteq(\operatorname{supp}f)^*$,
$\|v_j\|_{L^\infty(\mathbb{R}^n)}\le\|f\|_{L^\infty(\mathbb{R}^n)}$ for all
$j\in\mathbb{N}$, and
$v_j\to f$ a.e. on $\mathbb{R}^n$ as $j\to\infty$.
\end{lemma}
\begin{proof}
Let $\{\varrho_k\}_{k\in\mathbb{N}}$ be the sequence defined by 
\eqref{eq:mollification}. By \cite[Theorem~4.22]{B11}, the sequence
$\nu_k:=\varrho_k\ast f$ converges to $f$ in $L^1(\mathbb{R}^n)$
as $k\to\infty$. In view of \cite[Proposition~4.18]{B11}, we have
$\operatorname{supp}\nu_k\subseteq(\operatorname{supp}f)^*$.
By the Young inequality for convolutions (see, e.g., \cite[Theorem~4.15]{B11}),
one has 
$\|\nu_k\|_{L^\infty(\mathbb{R}^n)}
\le
\|\varrho_k\|_{L^1(\mathbb{R}^n)}
\|f\|_{L^\infty(\mathbb{R}^n)}
=
\|f\|_{L^\infty(\mathbb{R}^n)}$
for all $k\in\mathbb{N}$. Since $\|\nu_k-f\|_{L^1(\mathbb{R}^n)}\to 0$
as $k\to\infty$, there exists a subsequence $\{\nu_{k_j}\}_{j\in\mathbb{N}}$
of the sequence $\{\nu_k\}_{k\in\mathbb{N}}$ such that $\nu_{k_j}\to f$
a.e. on $\mathbb{R}^n$ as $j\to\infty$. Then the required sequence
$\{v_j\}_{j\in\mathbb{N}}$ is defined by $v_j:=\nu_{k_j}$ for $j\in\mathbb{N}$.
\end{proof}
We finish this subsection with a result that we will use in the next one.
\begin{lemma}\label{le:bounded-L2-approximation-necessity}
Suppose a Banach function space $X(\mathbb{R}^n)$ has the bounded 
$L^2$-approximation property. Then for every function
$f\in L^\infty(\mathbb{R}^n)$ with compact support there exists a sequence 
$\{v_j\}_{j\in\mathbb{N}} \subset C^\infty_0(\mathbb{R}^n)$ such that
\[
\operatorname{supp} v_j \subseteq \left(\operatorname{supp} f\right)^\ast , 
\quad
\|v_j\|_{L^\infty(\mathbb{R}^n)} \le  \|f\|_{L^\infty(\mathbb{R}^n)} + 1 ,
\quad 
\limsup_{j \to \infty}\|v_j\|_{X(\mathbb{R}^n)} \le  \|f\|_{X(\mathbb{R}^n)},
\]
and $v_j \to f$ a.e. on $\mathbb{R}^n$ as $j \to \infty$.
\end{lemma}
\begin{proof}
Let $\{u_k\}_{k\in\mathbb{N}} \subset C^\infty_0(\mathbb{R}^n)$ be such that
\[
\lim_{k \to \infty}\|f - u_k\|_{L^2(\mathbb{R}^n)} = 0 , 
\quad
\limsup_{k \to \infty}\|u_k\|_{X(\mathbb{R}^n)} \le  \|f\|_{X(\mathbb{R}^n)}.
\]
Consider a function $\zeta \in C^\infty_0(\mathbb{R}^n)$ 
such that $0 \le \zeta(x) \le 1$ for $x\in\mathbb{R}^n$, 
$\zeta(x) =1$ for  $x \in \operatorname{supp} f$, and 
$\operatorname{supp}\zeta\subseteq \left(\operatorname{supp} f\right)^\ast$. 
Then $\zeta f = f$. Further, let $\eta : \mathbb{C} \to \mathbb{C}$ be a 
function, which can be represented for $z=y_1+iy_2\in\mathbb{C}$ as
$\eta(z)=\eta(y_1+iy_2)=U(y_1,y_2)+iV(y_1,y_2)$ with real-valued functions
$U,V \in C^\infty_0(\mathbb{R}^2)$. Assume that
$\eta(z) = z$ for all $z \in \mathbb{C}$ with  
$|z| \le \|f\|_{L^\infty(\mathbb{R}^n)}$, $ |\eta(z)| \le |z|$ for all 
$z \in \mathbb{C}$, and 
\[
\max_{z \in \mathbb{C}} |\eta(z)| \le \|f\|_{L^\infty(\mathbb{R}^n)} + 1.
\]
Then $\eta\circ(\zeta f) = \eta\circ f = f$. 

Set $\nu_k := \eta\circ(\zeta u_k)$ for $k\in\mathbb{N}$. Then 
$\nu_k \in C^\infty_0(\mathbb{R}^n)$,
$\operatorname{supp} \nu_k 
\subseteq 
\operatorname{supp} \zeta 
\subseteq \left(\operatorname{supp} f\right)^\ast$,
\[
\|\nu_k\|_{L^\infty(\mathbb{R}^n)} 
\le 
\max_{z \in \mathbb{C}} |\eta(z)| 
\le 
\|f\|_{L^\infty(\mathbb{R}^n)} + 1,
\]
and 
\[
|\nu_k(x)| 
= 
|\eta(\zeta(x) u_k(x))| 
\le 
|\zeta(x) u_k(x)| 
\le 
|u_k(x)|
\quad\mbox{for all}\quad x\in\mathbb{R}^n.
\]
Hence
\[
\limsup_{k \to \infty}\|\nu_k\|_{X(\mathbb{R}^n)} 
\le  
\limsup_{k \to \infty}\|u_k\|_{X(\mathbb{R}^n)} 
\le 
\|f\|_{X(\mathbb{R}^n)}.
\]
Further, by the mean value theorem,
there exists a constant $C_\eta$ depending
on the maxima of the partial derivatives of 
the functions $U$ and $V$ such that
\begin{align*}
\|f - \nu_k\|_{L^2(\mathbb{R}^n)} 
&= 
\|\eta\circ(\zeta f) - \eta\circ(\zeta u_k)\|_{L^2(\mathbb{R}^n)} 
\le 
C_\eta
\|\zeta f - \zeta u_k\|_{L^2(\mathbb{R}^n)}
\\
&\le 
C_\eta
\|f - u_k\|_{L^2(\mathbb{R}^n)} \to 0 \ \mathrm{as} \ k \to \infty .
\end{align*}
Hence the sequence $\{\nu_k\}_{k \in \mathbb{N}}$ has a subsequence 
$\{\nu_{k_j}\}_{j\in\mathbb{N}}$ such that $\nu_{k_j} \to f$ 
a.e. on $\mathbb{R}^n$ as $j \to \infty$
(see, e.g., \cite[Chap.~1, Theorem~1.7(vi)]{BS88}). 
Then the required sequence $\{v_j\}_{j\in\mathbb{N}}$ is defined by
$v_j:=\nu_{k_j}$ for $j\in\mathbb{N}$.
\end{proof}
\subsection{The norm fundamental property}
\begin{definition}\label{def:NFP}
We say that a Banach function space $X(\mathbb{R}^n)$ satisfies
the norm fun\-da\-men\-tal property if for every $f\in X(\mathbb{R}^n)$,
\[
\|f\|_{X(\mathbb{R}^n)} =
\sup\left\{
\left|\int_{\mathbb{R}^n} f(x) \psi(x)\, dx\right|\ : \
\psi \in C^\infty_0(\mathbb{R}^n),  \
\|\psi\|_{X'(\mathbb{R}^n)} \le 1
\right\}.
\]
\end{definition}
Let $S_0(\mathbb{R}^n)$ denote the set of all simple compactly supported
functions.
\begin{lemma}\label{le:NFP-X}
Let $X(\mathbb{R}^n)$ be a Banach function space and $X'(\mathbb{R}^n)$
be its associate space.
For every $f\in X(\mathbb{R}^n)$,
\begin{equation}\label{eq:NFP-X-1}
\|f\|_{X(\mathbb{R}^n)}=\sup\left\{
\left|\int_{\mathbb{R}^n}f(x)s(x)\,dx\right| \ : \
s\in S_0(\mathbb{R}^n),\
\|s\|_{X'(\mathbb{R}^n)}\le 1
\right\}.
\end{equation}
\end{lemma}
\begin{proof}
By \cite[Theorem~2.7 and Lemma~2.8]{BS88}, for every 
$f\in X(\mathbb{R}^n)$,
\begin{equation}\label{eq:NFP-X-3}
\|f\|_{X(\mathbb{R}^n)}=\sup\left\{
\left|\int_{\mathbb{R}^n}f(x)g(x)\,dx\right| \ : \
g\in X'(\mathbb{R}^n),\
\|g\|_{X'(\mathbb{R}^n)}\le 1
\right\}.
\end{equation}
It follows from the inclusion $S_0(\mathbb{R}^n)\subset X'(\mathbb{R}^n)$
and equality \eqref{eq:NFP-X-3} that
\begin{equation}\label{eq:NFP-X-4}
\|f\|_{X(\mathbb{R}^n)}\ge \sup\left\{
\left|\int_{\mathbb{R}^n}f(x)s(x)\,dx\right| \ : \
s\in S_0(\mathbb{R}^n),\
\|s\|_{X'(\mathbb{R}^n)}\le 1
\right\}.
\end{equation}
Fix $g\in X'(\mathbb{R}^n)$ such that $\|g\|_{X'(\mathbb{R}^n)}\le 1$.
Then there exists a sequence 
$\{s_j\}_{j\in\mathbb{N}}\subset S_0(\mathbb{R}^n)$ such that
$0\le |s_1|\le|s_2|\le\dots\le|g|$ and $s_j\to g$ a.e. on $\mathbb{R}^n$
as $j\to\infty$. Therefore, $fs_j\to fg$ as $j\to\infty$ and
$|fs_j|\le|fg|$ for all $j\in\mathbb{N}$ a.e. on $\mathbb{R}^n$.
By H\"older's inequality (see \cite[Chap.~1, Theorem~2.4]{BS88}),
$fg\in L^1(\mathbb{R}^n)$. Hence, in view of the Lebesgue
dominated convergence theorem,
\[
\lim_{j\to\infty}\int_{\mathbb{R}^n}f(x)s_j(x)\,dx
=
\int_{\mathbb{R}^n}f(x)g(x)\,dx.
\]
On the other hand, inequality $|s_j|\le |g|$ implies that
$\|s_j\|_{X'(\mathbb{R}^n)}\le\|g\|_{X'(\mathbb{R}^n)}\le 1$ for all
$j\in\mathbb{N}$. Thus, for all
$g\in X'(\mathbb{R}^n)$ satisfying $\|g\|_{X'(\mathbb{R}^n)}\le 1$, we have
\begin{align*}
\left|\int_{\mathbb{R}^n}f(x)g(x)\,dx\right|
&=
\lim_{j\to\infty}\left|\int_{\mathbb{R}^n}f(x)s_j(x)\,dx\right|
\le 
\sup_{j\in\mathbb{N}}\left|\int_{\mathbb{R}^n}f(x)s_j(x)\,dx\right|
\\
&\le 
\sup\left\{
\left|\int_{\mathbb{R}^n}f(x)s(x)\,dx\right| \ : \
s \in S_ 0(\mathbb{R}^n),\
\|s\|_{X'(\mathbb{R}^n)}\le 1
\right\}.
\end{align*}
This inequality and equality \eqref{eq:NFP-X-3} imply that
\begin{equation}\label{eq:NFP-X-5}
\|f\|_{X(\mathbb{R}^n)}\le \sup\left\{
\left|\int_{\mathbb{R}^n}f(x)s(x)\,dx\right| \ : \
s\in S_0(\mathbb{R}^n),\
\|s\|_{X'(\mathbb{R}^n)}\le 1
\right\}.
\end{equation}
Combining inequalities \eqref{eq:NFP-X-4} and \eqref{eq:NFP-X-5}, we arrive
at equality \eqref{eq:NFP-X-1}.
\end{proof}
\begin{theorem}\label{th:bounded-L2-approximation-implies-NFP}
If $X'(\mathbb{R}^n)$ satisfies the bounded $L^2$-approximation property, 
then $X(\mathbb{R}^n)$ has the norm fundamental property.
\end{theorem}
\begin{proof} 
Lemma \ref{le:NFP-X} implies that it is sufficient to prove the inequality
\begin{equation}\label{eq:bounded-L2-approximation-implies-NFP}
\left|\int_{\mathbb{R}^n} \varphi(x) s(x)\, dx\right| 
\le   
\sup\left\{
\left|\int_{\mathbb{R}^n} \varphi(x)\psi(x)\, dx\right| 
: \ \psi \in C^\infty_0(\mathbb{R}^n), \ 
\|\psi\|_{X'(\mathbb{R}^n)} \le 1
\right\}
\end{equation}
for any $\varphi \in X(\mathbb{R}^n)$ and any $s \in S_0(\mathbb{R}^n)$ with
$\|s\|_{X'(\mathbb{R}^n)} \le 1$. According to 
Lemma~\ref{le:bounded-L2-approximation-necessity}, there exists a sequence 
$\{\psi_j\}_{j\in\mathbb{N}} \subset C^\infty_0(\mathbb{R}^n)$
such that
\[
\operatorname{supp} \psi_j \subseteq \left(\operatorname{supp} s\right)^\ast, 
\quad 
\|\psi_j\|_{L^\infty(\mathbb{R}^n)} \le  \|s\|_{L^\infty(\mathbb{R}^n)} + 1 ,
\quad
\limsup_{j \to \infty}\|\psi_j\|_{X'(\mathbb{R}^n)}\le\|s\|_{X'(\mathbb{R}^n)},
\]
and $\psi_j \to s$ a.e. on $\mathbb{R}^n$ as $j \to \infty$. Take an arbitrary 
$\varepsilon  \in (0, 1)$. Then 
$(1 - \varepsilon) \|\psi_j\|_{X'(\mathbb{R}^n)} \le 1$ for all sufficiently 
large $j \in \mathbb{N}$. Axiom (A5) and the Lebesgue dominated convergence 
theorem imply that
\begin{align*}
(1 - \varepsilon) \left|\int_{\mathbb{R}^n} \varphi(x) s(x)\, dx\right| 
&= 
(1 - \varepsilon) \lim_{j \to \infty} 
\left|\int_{\mathbb{R}^n} \varphi(x) \psi_j(x)\, dx\right| 
\\
&\le 
\sup\left\{
\left|\int_{\mathbb{R}^n} \varphi(x)\psi(x)\, dx\right|
\ : \ \psi \in C^\infty_0(\mathbb{R}^n), \ 
\|\psi\|_{X'(\mathbb{R}^n)} \le 1\right\}.
\end{align*}
Since  $\varepsilon  \in (0, 1)$ is arbitrary, 
\eqref{eq:bounded-L2-approximation-implies-NFP} follows.
\end{proof}
\begin{corollary}\label{co:NFP-associate-separable}
If $X(\mathbb{R}^n)$ is a Banach function space such that its associate
space $X'(\mathbb{R}^n)$ has absolutely continuous norm, then 
$X(\mathbb{R}^n)$ satisfies the norm fundamental property.
\end{corollary}
\begin{proof}
This follows from Theorems~\ref{th:bounded-L2-approximation-implies-NFP} 
and~\ref{th:AC-implies-bounded-L2-approximation}.
\end{proof}
Note that a Banach function space $X(\mathbb{R}^n)$ may satisfy the norm
fundamental property even if $(X')_a(\mathbb{R}^n)=\{0\}$. For instance,
if $X(\mathbb{R}^n)=L^1(\mathbb{R}^n)$, then 
$(X')_a(\mathbb{R}^n)=(L^\infty)_a(\mathbb{R}^n)=\{0\}$ in view of
\cite[Chap.~3, Theorem~5.5(b)]{BS88}. However, the following result is true.
\begin{corollary}\label{co:NFP-weighted-TI}
Let $Y(\mathbb{R}^n)$ be a translation-invariant Banach function space and let 
$w\in C(\mathbb{R}^n)$ be a function such that $w(x)>0$ for all 
$x\in\mathbb{R}^n$.  Then $X(\mathbb{R}^n)=Y(\mathbb{R}^n, w)$
has the norm fundamental property.
\end{corollary}
\begin{proof}
In view of Lemma~\ref{le:TI-X-and-Xprime}, the space $Y'(\mathbb{R}^n)$ is
translation-invariant. On the other hand, $w^{-1} \in C(\mathbb{R}^n)$ and 
$w^{-1} > 0$. It follows from \cite[Lemma~2.4(c)]{KS14} that
$X'(\mathbb{R}^n) = Y'(\mathbb{R}^n, w^{-1})$. By 
Theorem~\ref{th:TI-implies-bounded-L2-approximation}, the space
$Y'(\mathbb{R}^n,w^{-1})$ satisfies the bounded $L^2$-approximation
property. Then the space $Y(\mathbb{R}^n,w)$ has the norm fundamental
property due to Theorem~\ref{th:bounded-L2-approximation-implies-NFP}.
\end{proof}
Similarly to Theorem \ref{th:TI-implies-bounded-L2-approximation}, one cannot 
drop the requirement of continuity of the weight $w$ in 
Corollary~\ref{co:NFP-weighted-TI}.
\begin{lemma} 
Let $G \subset \mathbb{R}^n$ be a compact set of positive measure with empty
interior and let the weight $w_G$ be defined by \eqref{eq:exotic-weight}. 
Suppose $f \in L^1(\mathbb{R}^n, w_G^{-1})$, $f \ge 0$, and 
$\operatorname{supp} f \cap G$ has positive measure. Then
\[
\|f\|_{L^1(\mathbb{R}^n, w_G^{-1})} 
> 
\sup\left\{\left|\int_{\mathbb{R}^n} f(x)\psi(x)\, dx\right|
\ : \ \psi \in C(\mathbb{R}^n), \ 
\|\psi\|_{L^\infty(\mathbb{R}^n, w_G)} \le 1\right\} .
\]
\end{lemma}
\begin{proof} 
It follows from Lemma~\ref{le:exotic-weight} that
\[
\left|\int_{\mathbb{R}^n} f(x)\psi(x)\, dx\right| 
\le 
\frac12 \int_{\mathbb{R}^n} f(x)\, dx 
\]
for any function $\psi \in C(\mathbb{R}^n)$ with 
$\|\psi\|_{L^\infty(\mathbb{R}^n, w_G)} \le 1$. 
On the other hand,
\begin{align*}
\|f\|_{L^1(\mathbb{R}^n, w_G^{-1})} 
&= 
\int_{G} f(x) w_G^{-1}(x)\, dx 
+ 
\int_{\mathbb{R}^n\setminus G} f(x) w_G^{-1}(x)\, dx 
\\
&= 
\int_{G} f(x)\, dx + \frac12 \int_{\mathbb{R}^n\setminus G} f(x)\, dx 
> 
\frac12 \int_{\mathbb{R}^n} f(x)\, dx ,
\end{align*}
since $\operatorname{supp} f \cap G$ has positive measure.
\end{proof}
\begin{corollary}\label{co:L1-with-exotic-weight-fails-NFP}
Let $G \subset \mathbb{R}^n$ be a compact set of positive measure with empty
interior and let the weight $w_G$  be defined by \eqref{eq:exotic-weight}. 
Then the Banach function space $L^1(\mathbb{R}^n, w_G^{-1})$ does not have 
the norm fundamental property. 
\end{corollary}
Corollary~\ref{co:L1-with-exotic-weight-fails-NFP} and 
Theorem~\ref{th:bounded-L2-approximation-implies-NFP} provide 
an alternative proof of 
Theorem~\ref{th:L-infinity-with-exotic-weight-fails-bounded-L2-approximation}.
\subsection{Lemma on approximation at Lebesgue points}
Given $\delta>0$ and a function $\psi$ on $\mathbb{R}^n$, we define
the function $\psi_\delta$ by
\[
\psi_\delta(\xi) := \delta^{-n} \psi(\xi/\delta),
\quad
\xi\in\mathbb{R}^n.
\]
Recall that a point $x\in\mathbb{R}^n$ is said to be a Lebesgue point of
a function $f\in L_{\rm loc}^1(\mathbb{R}^n)$ if
\[
\lim_{R\to 0^+}\frac{1}{|B(x,R)|}\int_{B(x,R)}|f(y)-f(x)|\,dy=0.
\]
For $\sigma\in\mathbb{R}$, we will say that a measurable function $f$
belongs to the space $L_{1,\sigma}(\mathbb{R}^n)$ if
\[
\int_{\mathbb{R}^n} (1+|\xi|)^{-\sigma} |f(\xi)|\, d\xi < \infty.
\]
\begin{lemma}\label{le:Lebesgue}
Let $\sigma_1,\sigma_2\in\mathbb{R}$ be such that $\sigma_2\ge\sigma_1$ and
$\sigma_2>n$. Suppose $\psi$ is a measurable function on $\mathbb{R}^n$
satisfying
\begin{equation}\label{eq:Lebesgue-point-1}
|\psi(\xi)| \le C (1 + |\xi|)^{-\sigma_2}
\quad\mbox{for almost all}\quad
\xi \in \mathbb{R}^n
\end{equation}
with some constant $C\in(0,\infty)$. Then for every Lebesgue point
$\eta\in\mathbb{R}^n$ of a function $a$ belonging to the space
$L_{1,\sigma_1}(\mathbb{R}^n)$, one has
\[
\int_{\mathbb{R}^n} |a(\xi)-a(\eta)|\,
|\psi_\delta(\eta-\xi)|\,
d\xi \to 0
\quad \mbox{ as } \quad \delta \to 0.
\]
\end{lemma}
\begin{proof}
The proof is similar to the proof of \cite[Chap.~I, Theorem~1.25]{SW71}
(see also \cite[Chap.~II, Lemma~1]{CZ52}). We give it here for the
convenience of the reader.

Take an arbitrary $\varepsilon>0$. Since $\eta\in\mathbb{R}^n$ is a Lebesgue
point of $a$, there exists a $\rho>0$ such that
\begin{equation}\label{eq:Lebesgue-point-2}
r^{-n}\int_{|\zeta|<r}|a(\eta-\zeta)-a(\eta)|\,d\zeta \le \varepsilon
\quad
\mbox{for all}
\quad
r \in(0,\rho].
\end{equation}
Substituting $\eta-\xi$ with $\zeta$ and splitting the integral, we get for
any $\delta>0$,
\begin{align}
\int_{\mathbb{R}^n} |a(\xi)-a(\eta)|\,|\psi_\delta(\eta-\xi)|\,d\xi
=&
\int_{|\zeta|<\rho}|a(\eta-\zeta)-a(\eta)|\,|\psi_\delta(\zeta)|\,d\zeta
\nonumber \\
&+
\int_{|\zeta|\ge\rho}|a(\eta-\zeta)-a(\eta)|\,|\psi_\delta(\zeta)|\,d\zeta
\nonumber\\
=: &
I_1(\delta)+I_2(\delta).
\label{eq:Lebesgue-point-3}
\end{align}
Let $\mathbb{S}^{n-1}=\{\vartheta\in\mathbb{R}^n:|\vartheta|=1\}$ be the unit
sphere in $\mathbb{R}^n$ and let
\[
g(r):=\int_{\mathbb{S}^{n-1}}|a(\eta-r\vartheta)-a(\eta)|\,d\vartheta,
\]
where $d\vartheta$ is an element of the surface area on $\mathbb{S}^{n-1}$. Then
condition \eqref{eq:Lebesgue-point-2} is equivalent to
\begin{equation}\label{eq:Lebesgue-point-4}
G(r):=\int_0^r s^{n-1}g(s)\,ds\le \varepsilon r^n
\quad
\mbox{for all}\quad r
\in (0, \rho]
\end{equation}
(see, e.g., \cite[Theorem~2.49]{F99}). Let
\[
\phi(r) :=  C (1 + r)^{-\sigma_2},
\quad
\phi_{(\delta)}(r) := \delta^{-n} \phi(r/\delta) ,
\quad
r \ge 0.
\]
Then \eqref{eq:Lebesgue-point-1} implies that
\[
I_1(\delta)
\le
\int_{|\zeta|<\rho}|a(\eta-\zeta)-a(\eta)| \phi_{(\delta)}(|\zeta|) \,d\zeta
=
\int_0^\rho r^{n-1}g(r)\phi_{(\delta)}(r)\,dr.
\]
Integrating by parts twice and taking into account \eqref{eq:Lebesgue-point-4}
and the inequalities $\phi'_{(\delta)}<0$ and $\sigma_2>n$,
we obtain
\begin{align}
I_1(\delta)
& \le
\big[G(r)\phi_{(\delta)}(r)\big]_0^\rho
-
\int_0^\rho G(r)\phi'_{(\delta)}(r)\,dr
 \le
\varepsilon \rho^n \phi_{(\delta)}(\rho)
-
\int_0^\rho \varepsilon r^n \phi'_{(\delta)}(r)\,dr
\nonumber\\
&=
n \varepsilon \int_0^\rho  r^{n - 1} \phi_{(\delta)}(r)\,dr
 \le
n \varepsilon \int_0^\infty  r^{n - 1} \phi_{(\delta)}(r)\,dr
=
n \varepsilon \int_0^\infty  s^{n - 1} \phi(s)\,ds
\nonumber\\
&\le
C n \varepsilon \int_0^\infty  (1 + s)^{n - 1 - \sigma_2}\,ds
=
\frac{C n}{\sigma_2 - n}\, \varepsilon
=:
A \varepsilon .
\label{eq:Lebesgue-point-5}
\end{align}
Since $a\in L_{1,\sigma_1}(\mathbb{R}^n)$, we get for all $\delta>0$,
\begin{align}
I_2(\delta)
\le &
\int_{|\zeta|\ge\rho}|a(\eta-\zeta)-a(\eta)| \phi_{(\delta)}(|\zeta|) \,d\zeta
\nonumber\\
\le&
\left(
\int_{\mathbb{R}^n}|a(\eta-\zeta)|(1+|\eta-\zeta|)^{-\sigma_1}\,d\zeta
\right)
\sup_{|\zeta|\ge\rho}
\left((1+|\eta-\zeta|)^{\sigma_1}\phi_{(\delta)}(|\zeta|)\right)
\nonumber\\
&
+|a(\eta)|\int_{|\zeta|\ge\rho} \phi_{(\delta)}(|\zeta|) \,d\zeta
\nonumber\\
\le&
\|a\|_{1,\sigma_1}\sup_{|\zeta|\ge\rho}
\left((1+|\eta-\zeta|)^{\sigma_2}\phi_{(\delta)}(|\zeta|)\right)
+
|a(\eta)|\, \omega_n \int_\rho^\infty r^{n - 1}\phi_{(\delta)}(r) \,dr ,
\label{eq:Lebesgue-point-6}
\end{align}
where $\omega_n$ is the surface area of $\mathbb{S}^{n-1}$.
It is clear that for $|\zeta|\ge\rho$,
\begin{align}
(1+|\eta-\zeta|)^{\sigma_2}\phi_{(\delta)}(|\zeta|)
&\le
C(1+|\eta|+|\zeta|)^{\sigma_2} \delta^{-n}(1+|\zeta|/\delta)^{-\sigma_2}
\nonumber\\
&=
C\frac{(1+|\eta|+|\zeta|)^{\sigma_2}}{(\delta+|\zeta|)^{\sigma_2}}
\delta^{\sigma_2-n} <
C\left(\frac{1+|\eta|}{|\zeta|}+1\right)^{\sigma_2}\delta^{\sigma_2-n}
\nonumber\\
&\le
C\left(\frac{1+|\eta|}{\rho}+1\right)^{\sigma_2}\delta^{\sigma_2-n}.
\label{eq:Lebesgue-point-7}
\end{align}
Further,
\begin{align}
\int_\rho^\infty r^{n - 1}\phi_{(\delta)}(r) \,dr
&=
\int_{\rho/\delta}^\infty s^{n - 1} \phi(s)\,ds \le
C\int_{\rho/\delta}^\infty (1 + s)^{n - 1 - \sigma_2}\,ds
\nonumber\\
&=
\frac{C}{\sigma_2 - n}\, (1 + \rho/\delta)^{n - \sigma_2}
\le
\frac{C}{(\sigma_2 - n) \rho^{\sigma_2 - n}}\, \delta^{\sigma_2 - n}.
\label{eq:Lebesgue-point-8}
\end{align}
It follows from \eqref{eq:Lebesgue-point-6}--\eqref{eq:Lebesgue-point-8} that
\[
I_2(\delta) \le C\left(\|a\|_{1,\sigma_1} 
\left(\frac{1+|\eta|}{\rho}+1\right)^{\sigma_2} +
\frac{\omega_n |a(\eta)|}{(\sigma_2 - n) \rho^{\sigma_2 - n}}\right) 
\delta^{\sigma_2 - n}.
\]
Hence there exists a $\delta_0=\delta_0(\varepsilon)>0$ such that
\[
I_2(\delta) < \varepsilon
\quad\mbox{for all}\quad \delta \in (0, \delta_0),
\]
and inequality \eqref{eq:Lebesgue-point-5} implies that
\[
I_1(\delta)+I_2(\delta)<(A+1)\varepsilon
\quad\mbox{for all}\quad\delta \in (0, \delta_0).
\]
Combining this estimate with \eqref{eq:Lebesgue-point-3}, we arrive at the 
desired result.
\end{proof}
\section{Weak doubling property}
\label{sec:doubling}
\subsection{The infimum of the doubling constants}
For a Banach function space $X(\mathbb{R}^n)$ and $\tau>1$, 
consider the doubling constant
\begin{equation}\label{eq:def-D}
D_{X,\tau}:=\liminf_{R \to \infty}
\left(\inf_{y \in \mathbb{R}^n}
\frac{\|\chi_{B(y,\tau R)}\|_{X(\mathbb{R}^n)}}
{\|\chi_{B(y, R)}\|_{X(\mathbb{R}^n)}}
\right).
\end{equation}
We immediately deduce from the lattice property (Axiom (A2) in the definition 
of a Banach function space) that $1 \le D_{X, \tau_1} \le D_{X, \tau_2}$
for all $1<\tau_1 \le \tau_2$. Therefore,
\[
\inf_{\tau>1} D_{X,\tau}\ge 1.
\]
\begin{lemma}\label{le:weak-doubling-property}
If a Banach function space $X(\mathbb{R}^n)$ satisfies the weak doubling
property, then
\[
\inf_{\tau > 1} D_{X, \tau} = 1.
\]
\end{lemma}
\begin{proof}
Since $X(\mathbb{R}^n)$ satisfies the weak doubling property, there exists a
number $\varrho>1$ such that $D_{X,\varrho}<\infty$. Assume, contrary to
the hypothesis, that
\[
D := \inf_{\tau > 1} D_{X, \tau} > 1.
\]
Take an arbitrary $N \in \mathbb{N}$ and consider $\tau = \varrho^{1/N}$.
Since
\[
D_{X, \tau} \ge D > D_0 := \frac{D + 1}{2} > 1,
\]
it follows from the definition of $D_{X, \tau}$ that there exists
a number $R_0 > 0$ such that for all $R\ge R_0$,
\[
\inf_{y \in \mathbb{R}^n}
\frac{\|\chi_{B(y, \tau R)}\|_{X(\mathbb{R}^n)}}
{\|\chi_{B(y, R)}\|_{X(\mathbb{R}^n)}}
\ge D_0.
\]
Hence, for all $y \in \mathbb{R}^n$  and all $R \ge R_0$,
\[
\frac{\|\chi_{B(y, \varrho R)}\|_{X(\mathbb{R}^n)}}
{\|\chi_{B(y, R)}\|_{X(\mathbb{R}^n)}}
=
\prod_{j = 1}^N
\frac{\|\chi_{B(y, \tau^j R)}\|_{X(\mathbb{R}^n)}}
{\|\chi_{B(y, \tau^{j - 1} R)}\|_{X(\mathbb{R}^n)}}
\ge
D_0^N .
\]
Therefore, $D_{X, \varrho} \ge D_0^N$ for all $N \in \mathbb{N}$,
which is impossible since $D_0 > 1$ and $D_{X, \varrho} < \infty$.
The obtained contradiction completes the proof.
\end{proof}
\subsection{The doubling property and the $A_X$-condition}
\begin{definition}
We say that a Banach function space $X(\mathbb{R}^n)$ satisfies the (strong) 
doubling property if there exist  a number $\tau > 1$ and a constant 
$C_\tau > 0$ such that for all $R > 0$ and $y \in \mathbb{R}^n$,
\begin{equation}\label{eq:strong-doubling}
\frac{\|\chi_{B(y,\tau R)}\|_{X(\mathbb{R}^n)}}
{\|\chi_{B(y, R)}\|_{X(\mathbb{R}^n)}} 
\le 
C_\tau.
\end{equation}
\end{definition}
The doubling property is considerably stronger than the weak doubling property.
Indeed, it is easy to see that a Banach function space $X(\mathbb{R}^n)$ 
satisfies the weak doubling property if and only if there exist a number 
$\tau>1$, a constant $C_\tau >0$, a sequence 
$\{R_j\}_{j\in\mathbb{N}}\subset(0,\infty)$ satisfying $R_j\to \infty$ as 
$j\to\infty$, and a sequence $\{y_j\}_{j\in\mathbb{N}}$ in $\mathbb{R}^n$ 
such that
\begin{equation}\label{eq:weak-doubling}
\frac{\|\chi_{B(y_j, \tau R_j)}\|_{X(\mathbb{R}^n)}}
{\displaystyle
\|\chi_{B(y_j, R_j)}\|_{X(\mathbb{R}^n)}}
\le C_\tau
\quad\mbox{for all}\quad j\in\mathbb{N}.
\end{equation}
So, the difference between the doubling property and the weak doubling 
property is that the former requires estimate \eqref{eq:strong-doubling} 
to hold for all balls, while the latter requires it to hold only for some 
sequence of balls with radii going to infinity. We will return to this 
comparison in Subsection~\ref{subsec:examples-doubling}.
 
Now we give a sufficient condition guaranteeing that a Banach function
space $X(\mathbb{R}^n)$ satisfies the doubling property.
We say that a Banach function space $X(\mathbb{R}^n)$ satisfies the 
$A_X$-condition if
\begin{equation}\label{eq:AX}
\sup_Q\frac{1}{|Q|}
\|\chi_Q\|_{X(\mathbb{R}^n)}\|\chi_Q\|_{X'(\mathbb{R}^n)}<\infty,
\end{equation}
where the supremum is taken over all cubes with sides parallel to the 
coordinate axes. This condition goes back to E.~I.~Berezhnoi \cite{B99}.
\begin{lemma}\label{le:AX-implies-doubling}
If $X(\mathbb{R}^n)$ is a Banach function space satisfying the $A_X$-condition,
then $X(\mathbb{R}^n)$ satisfies the doubling property.
\end{lemma}
\begin{proof}
It is well-known that there exist constants $0<m_n<M_n<\infty$ such that for
every ball $B$ in $\mathbb{R}^n$ and the corresponding inscribed and
circumscribed cubes $Q$ and $P$ one has
\[
m_n|P|\le |B|\le M_n|Q|.
\]
Then it follows from Axiom (A2) in the definition of a Banach function 
function norm that condition \eqref{eq:AX} is equivalent to the condition
\[
C_X:=\sup_B\frac{1}{|B|}
\|\chi_B\|_{X(\mathbb{R}^n)}
\|\chi_B\|_{X'(\mathbb{R}^n)}<\infty,
\]
where the supremum is taken over all balls in $\mathbb{R}^n$. Then, for
every $y\in\mathbb{R}^n$, $\tau>1$ and $R>0$,
\[
\|\chi_{B(y,\tau R)}\|_{X(\mathbb{R}^n)}
\le
C_X\frac{|B(y,\tau R)|}{\|\chi_{B(y,\tau R)}\|_{X'(\mathbb{R}^n)}}
\le 
C_X\frac{|B(y,\tau R)|}{\|\chi_{B(y,R)}\|_{X'(\mathbb{R}^n)}}
=
C_X\tau^n\frac{|B(y,R)|}{\|\chi_{B(y,R)}\|_{X'(\mathbb{R}^n)}}.
\]
It follows from the above inequality and H\"older's inequality for 
Banach function spaces (see \cite[Chap.~1, Theorem~2.4]{BS88}) that
\[
\|\chi_{B(y,\tau R)}\|_{X(\mathbb{R}^n)}
\leq C_X\tau^n
\frac{
\|\chi_{B(y,R)}\|_{X(\mathbb{R}^n)}
\|\chi_{B(y,R)}\|_{X'(\mathbb{R}^n)}
}
{
\|\chi_{B(y,R)}\|_{X'(\mathbb{R}^n)}
}
=
C_X\tau^n\|\chi_{B(y,R)}\|_{X(\mathbb{R}^n)}.
\]
Applying this inequality, we immediately get \eqref{eq:strong-doubling} with 
$C_\tau = C_X\tau^n$ for every $\tau>1$, which completes the proof.
\end{proof}
\subsection{Translation-invariant Banach function spaces satisfy the 
doubling property}
\begin{lemma}\label{le:TI-balls}
Let $X(\mathbb{R}^n)$ be a translation-invariant Banach function space. 
\begin{enumerate}
\item[(a)]
There exist constants $C_1, C_2 > 0$ such that 
for all $R > 0$ and $y \in \mathbb{R}^n$,
\begin{equation}\label{eq:TI-balls-1}
C_1 \min\left\{1, R^n\right\} 
\le 
\|\chi_{B(y, R)}\|_{X(\mathbb{R}^n)} 
\le 
C_2 \max\left\{1, R^n\right\}.
\end{equation}

\item[(b)]
For all $\tau > 1$, $R > 0$, and $y \in \mathbb{R}^n $,
\begin{equation}\label{eq:TI-balls-2}
\frac{\|\chi_{B(y, \tau R)}\|_{X(\mathbb{R}^n)}}
{\|\chi_{B(y, R)}\|_{X(\mathbb{R}^n)}} 
\le 
\left(4\sqrt{n}\, \tau\right)^n .
\end{equation}
\end{enumerate}
\end{lemma}
\begin{proof}
(a) All cubes in this proof are assumed to be closed and to have sides 
parallel to the coordinate axes. 
Let $Q(x,a)$ denote the cube centered at $x$ of side length $a$. Since the 
space $X(\mathbb{R}^n)$ is translation-invariant, we have
$\|\chi_{B(x,R)}\|_{X(\mathbb{R}^n)}=\|\chi_{B(y,R)}\|_{X(\mathbb{R}^n)}$
and $\|\chi_{Q(x,a)}\|_{X(\mathbb{R}^n)}=\|\chi_{Q(y,a)}\|_{X(\mathbb{R}^n)}$
for all $x,y\in\mathbb{R}^n$ and $a,R>0$. Therefore, we may simply write
$B_R$ and $Q_a$ for arbitrary open balls of radius $R$ and arbitrary 
cubes of side length $a$, respectively.

Let $a>0$ and $\mathcal{F}$ be the family of $2^n$ cubes $Q_a$ with pairwise 
disjoint interiors obtained from a fixed cube $Q_{2a}$ by dividing each its 
side in two segments of equal length: $Q_{2a}=\cup_{Q_a\in\mathcal{F}}Q_a$.
Then
\begin{equation}\label{eq:TI-balls-3}
\|\chi_{Q_{2a}}\|_{X(\mathbb{R}^n)} 
= 
\left\|\sum_{Q_a\in\mathcal{F}} \chi_{Q_a}\right\|_{X(\mathbb{R}^n)} 
\le 
\sum_{Q_a\in\mathcal{F}} \left\|\chi_{Q_a}\right\|_{X(\mathbb{R}^n)} 
= 
2^n \left\|\chi_{Q_a}\right\|_{X(\mathbb{R}^n)} .
\end{equation}
Using inequality \eqref{eq:TI-balls-3} $m$ times, one gets for all 
$m\in\mathbb{N}$,
\begin{equation}\label{eq:TI-balls-4}
\|\chi_{Q_{2^m}}\|_{X(\mathbb{R}^n)} 
\le 
2^{mn} \left\|\chi_{Q_1}\right\|_{X(\mathbb{R}^n)} ,
\quad
\|\chi_{Q_1}\|_{X(\mathbb{R}^n)} 
\le 
2^{mn} \left\|\chi_{Q_{2^{-m}}}\right\|_{X(\mathbb{R}^n)} , 
\end{equation}
and hence
\begin{equation}\label{eq:TI-balls-5}
\left\|\chi_{Q_{2^{-m}}}\right\|_{X(\mathbb{R}^n)} 
\ge 
2^{-mn}  \|\chi_{Q_1}\|_{X(\mathbb{R}^n)}.
\end{equation}

If $R \ge 1$, there exists $m \in \mathbb{N}$ such that 
$2^{m - 2} < R \le 2^{m -1}$. Then $B_R$ is contained in a cube $Q_{2^m}$
of side length $2^m$ and it follows from the first inequality in
\eqref{eq:TI-balls-4} that
\begin{align}
\|\chi_{B_1}\|_{X(\mathbb{R}^n)} 
&\le
\|\chi_{B_R}\|_{X(\mathbb{R}^n)} 
\le 
\|\chi_{Q_{2^m}}\|_{X(\mathbb{R}^n)}
\nonumber \\
&\le
2^{mn} \left\|\chi_{Q_1}\right\|_{X(\mathbb{R}^n)} 
< 
\left(4^n \left\|\chi_{Q_1}\right\|_{X(\mathbb{R}^n)}\right) R^n .
\label{eq:TI-balls-6}
\end{align}
If $R \le 1$, there exists $m \in \mathbb{N}\cup\{0\}$ such that 
$2^{-m - 1} \le R/\sqrt{n} < 2^{-m}$. Then it is easy to see that 
$B_R$ contains a cube $Q_{2^{-m}}$ of side length $2^{-m}$ and it follows from 
\eqref{eq:TI-balls-5} that
\begin{align}
\|\chi_{B_1}\|_{X(\mathbb{R}^n)} 
&\ge 
\|\chi_{B_R}\|_{X(\mathbb{R}^n)} 
\ge 
\left\|\chi_{Q_{2^{-m}}}\right\|_{X(\mathbb{R}^n)} 
\nonumber \\
&\ge
2^{-mn}  \|\chi_{Q_1}\|_{X(\mathbb{R}^n)} 
> 
\left(n^{-n/2} \left\|\chi_{Q_1}\right\|_{X(\mathbb{R}^n)}\right) R^n .
\label{eq:TI-balls-7}
\end{align}
Estimates \eqref{eq:TI-balls-6} and \eqref{eq:TI-balls-7} imply 
\eqref{eq:TI-balls-1} with 
\begin{align*}
C_1 = \min\left\{
\|\chi_{B_1}\|_{X(\mathbb{R}^n)},\,
 n^{-n/2} \left\|\chi_{Q_1}\right\|_{X(\mathbb{R}^n)}
\right\} , 
\quad
C_2 = \max\left\{
\|\chi_{B_1}\|_{X(\mathbb{R}^n)},\, 
4^n \left\|\chi_{Q_1}\right\|_{X(\mathbb{R}^n)}
\right\} .
\end{align*}
Part (a) is proved.

(b) For any $R > 0$, there exists $m \in \mathbb{Z}$ such that 
$2^{m - 2} < \tau R \le 2^{m - 1}$. Then any ball $B_{\tau R}$
is contained in a cube $Q_{2^m}$ of side length $2^m$. Let 
$m_0 := \left[\log_2 \left(\tau \sqrt{n}\right)\right] + 1$. It is easy to 
see  that $2^{m - m_0 - 2} < R/\sqrt{n}$ and $B_R$
contains a cube  $Q_{2^{m- m_0 - 1}}$ of side length $2^{m- m_0 - 1}$. Hence
\[
\frac{
\|\chi_{B_{\tau R}}\|_{X(\mathbb{R}^n)}
}
{
\|\chi_{B_R}\|_{X(\mathbb{R}^n)}
} 
\le 
\frac{\|\chi_{Q_{2^m}}\|_{X(\mathbb{R}^n)}}
{\|\chi_{Q_{2^{m- m_0 - 1}}}\|_{X(\mathbb{R}^n)}} \le 2^{(m_0 + 1)n} 
\le
\left(4\sqrt{n}\, \tau\right)^n , 
\] 
where the second inequality is obtained by applying 
\eqref{eq:TI-balls-3} $m_0 + 1$ times.
\end{proof}
Lemma~\ref{le:TI-balls}(b) immediately yields the following.
\begin{corollary}
If $X(\mathbb{R}^n)$ is a translation-invariant Banach function space, then
it satisfies the doubling property.
\end{corollary}
\subsection{Translation-invariant spaces with exponential weights fail 
the weak doubling property}
\begin{theorem}\label{th:example-nondoubling-exponent}
Suppose that $X(\mathbb{R})$ is a translation-invariant Banach function space. 
If $w(x) := e^{c x}$ for $x\in\mathbb{R}$ with a constant $c > 0$, then
the weighted Banach function space $X(\mathbb{R}, w)$ does not satisfy the 
weak doubling property.  
\end{theorem}
\begin{proof}
Let $\tau > 1$. By the second inequality in \eqref{eq:TI-balls-1}, for
every $y \in \mathbb{R}$ and every $R \ge 1$, one has
\[
\|\chi_{B(y, R)}\|_{X(\mathbb{R}, w)} 
\le  
e^{c (y + R)} \|\chi_{B(y, R)}\|_{X(\mathbb{R})} 
\le 
C_2 e^{c (y + R)} R.
\]
Set
\[
r:= \frac{\tau - 1}{4}\, R , \ \ \ \
z := y + \left(\frac34\, \tau + \frac14\right) R .
\] 
It is easy to see that $B(z, r) \subset B(y, \tau R)$ and 
$x \ge y + \frac{\tau + 1}{2}\, R$ 
for all $x \in B(z, r)$. Then  these observations and the first inequality in 
\eqref{eq:TI-balls-1} imply that
\begin{align*}
\|\chi_{B(y, \tau R)}\|_{X(\mathbb{R}, w)} 
&\ge 
\|\chi_{B(z, r)}\|_{X(\mathbb{R}, w)} 
\ge 
e^{c \left(y + \frac{\tau + 1}{2}\, R\right)} 
\|\chi_{B(z, r)}\|_{X(\mathbb{R})} 
\\
& \ge 
e^{c \left(y + \frac{\tau + 1}{2}\, R\right)} C_1 \min\left\{1, r\right\} 
= 
C_1 e^{c \left(y + \frac{\tau + 1}{2}\, R\right)} 
\min\left\{1, \frac{\tau - 1}{4}\, R\right\} .
\end{align*}
Hence
\[
\inf_{y\in\mathbb{R}^n}
\frac{\|\chi_{B(y, \tau R)}\|_{X(\mathbb{R},w)}}
{\|\chi_{B(y, R)}\|_{X(\mathbb{R}, w)}} 
\ge 
\frac{C_1}{C_2}\, \frac{e^{c\,\frac{\tau - 1}{2}\, R}}{R}
\to \infty \ \mbox{ as } \ R \to \infty, 
\]
and $X(\mathbb{R}, w)$ 
does not satisfy the weak doubling property.
\end{proof}
\begin{theorem}\label{th:example-nondoubling-exponent-modulus}
Suppose that $X(\mathbb{R}^n)$ is a translation-invariant Banach function 
space. If $w(x) := e^{c |x|}$ for $x\in\mathbb{R}^n$ with a constant $c > 0$, 
then the weighted Banach function space $X(\mathbb{R}^n, w)$ does not satisfy 
the weak doubling property.  
\end{theorem}
\begin{proof}
The proof is similar to that of Theorem~\ref{th:example-nondoubling-exponent}.
Let $\tau > 1$. It follows from the second inequality in \eqref{eq:TI-balls-1} 
that for every $y \in \mathbb{R}^n$ and every $R \ge 1$, one has
\[
\|\chi_{B(y, R)}\|_{X(\mathbb{R}^n, w)} 
\le  
e^{c (|y| + R)} \|\chi_{B(y, R)}\|_{X(\mathbb{R}^n)} 
\le 
C_2 e^{c (|y| + R)} R^n.
\]
Set
\[
r:= \frac{\tau - 1}{4}\, R , 
\quad
z := \left\{\begin{array}{cl}
 y + \left(\frac34\, \tau + \frac14\right) R \frac{y}{|y|}\, ,   
&  y \not= 0 ,  \\[3mm]
\left(\frac34\, \tau + \frac14\right) R e_1,     
&   y = 0 , 
\end{array}\right.
\] 
where $e_1 := (1, 0, \dots, 0) \in \mathbb{R}^n$. It is not difficult to see that 
$B(z, r) \subset B(y, \tau R)$ and 
\[
|x| \ge |y| + \frac{\tau + 1}{2}\, R
\quad\mbox{for all}\quad x \in B(z, r). 
\]
Hence, taking into account the first inequality in \eqref{eq:TI-balls-1},
we obtain
\begin{align*}
\|\chi_{B(y, \tau R)}\|_{X(\mathbb{R}^n, w)} 
&\ge 
\|\chi_{B(z, r)}\|_{X(\mathbb{R}^n, w)} 
\ge 
e^{c \left(|y| + \frac{\tau + 1}{2}\, R\right)}  
\|\chi_{B(z, r)}\|_{X(\mathbb{R}^n)} 
\\
&\ge 
e^{c \left(|y| + \frac{\tau + 1}{2}\, R\right)} C_1 \min\left\{1, r^n\right\} 
= 
C_1 e^{c \left(|y| + \frac{\tau + 1}{2}\, R\right)} 
\min\left\{1, \left(\frac{\tau - 1}{4}\, R\right)^n\right\}.
\end{align*}
Thus,
\[
\inf_{y\in\mathbb{R}^n}
\frac{\|\chi_{B(y, \tau R)}\|_{X(\mathbb{R}^n, w)}}
{\|\chi_{B(y, R)}\|_{X(\mathbb{R}^n, w)}} 
\ge 
\frac{C_1}{C_2}\, \frac{e^{c\,\frac{\tau - 1}{2}\, R}}{R^n}
\to \infty \ \mbox{ as } \ R \to \infty ,
\]
and $X(\mathbb{R}^n, w)$ does not satisfy the weak doubling property.
\end{proof}
\subsection{Comparison of the doubling property and the weak doubling property}
\label{subsec:examples-doubling}
\begin{lemma}
If $X(\mathbb{R}^n)$ is a Banach function space satisfying the doubling 
property, then the function 
\begin{equation}\label{eq:fundamental-function}
f_X(R):=\|\chi_{B(0,R)}\|_{X(\mathbb{R}^n)},
\quad R\in(0,\infty),
\end{equation}
cannot grow faster than polynomially as $R\to+\infty$.
\end{lemma}
\begin{proof}
The proof is analogous to the proof of \cite[Lemma~5.2.4]{SC02}.
Suppose there exist $\tau>1$ and $C_\tau>0$ such that 
\eqref{eq:strong-doubling} holds for $y=0$ and any $R > 1$. Then
there exists $m \in \mathbb{N}$ such that $\tau^{m - 1} < R \le \tau^m$. 
Applying~\eqref{eq:strong-doubling} $m$ times, one gets
\begin{align*}
f_X(R)
\le
f_X(\tau^m) 
\le 
C_\tau^m 
f_X(1)
=
\tau^{m \log_\tau C_\tau} 
f_X(1)
<
R^{\log_\tau C_\tau} \tau^{\log_\tau C_\tau} 
f_X(1)
= 
C_\tau f_X(1) R^{\log_\tau C_\tau},
\end{align*}
which completes the proof.
\end{proof}
On the other hand, we will show that the weak doubling property of
a Banach function space $X(\mathbb{R}^n)$ allows the function $f_X$ given 
by \eqref{eq:fundamental-function} to grow at any subexponential rate as 
$R\to+\infty$. In fact, we will show that if a weight $w$ grows at a 
subexponential rate in an open cone and $Y(\mathbb{R}^n)$ is a 
translation-invariant Banach function space, then the weighted Banach 
function space $X(\mathbb{R}^n)=Y(\mathbb{R}^n,w)$ satisfies the weak 
doubling property.

Inequalities \eqref{eq:weak-doubling} and \eqref{eq:TI-balls-2} yield the 
following.
\begin{lemma}\label{le:subexponential-growth}
Let $Y(\mathbb{R}^n)$ be a translation-invariant Banach function space.
Suppose that $w:\mathbb{R}^n\to[0,\infty]$ is a weight satisfying
$w,1/w\in L_{\rm loc}^\infty(\mathbb{R}^n)$. If there exist
a number $\tau>1$, a constant $c_\tau>0$, a sequence 
$\{R_j\}_{j\in\mathbb{N}}\subset(0,\infty)$ satisfying $R_j\to\infty$ 
as $j\to\infty$, and a sequence $\{y_j\}_{j\in\mathbb{N}}\subset\mathbb{R}^n$
such that
\begin{equation}\label{eq:subexponential-growth}
\frac
{\displaystyle\operatornamewithlimits{ess\,sup}_{x \in B(y_j, \tau R_j)} w(x)}
{\displaystyle\operatornamewithlimits{ess\,inf}_{x \in B(y_j, R_j)} w(x)}
\le c_\tau\quad\mbox{for all}\quad j\in\mathbb{N},
\end{equation}
then the weighted Banach function space $X(\mathbb{R}^n)=Y(\mathbb{R}^n,w)$
satisfies the weak doubling property.
\end{lemma}
\begin{lemma}\label{le:subexponential-growth-2}
Let $Y(\mathbb{R}^n)$ be a translation-invariant Banach function space and 
$\varphi$ be a nonincreasing  function such that $\varphi(r) \to 0$ as 
$r \to +\infty$ and $r\varphi(r)$ is nondecreasing for $r \ge 1$. Suppose  
$w:\mathbb{R}^n\to[0,\infty]$ is a weight satisfying
$w,1/w\in L_{\rm loc}^\infty(\mathbb{R}^n)$ and there exist $C_0, C_1 > 0$, 
$\gamma > 0$,  $\theta \in \mathbb{R}^n$ with $|\theta| = 1$ such that
\begin{equation}\label{2sided}
C_0\, \exp(|x| \varphi(|x|)) 
\le 
w(x) 
\le 
C_1\, \exp(|x| \varphi(|x|)) 
\ \mbox{ when } \ 
\left|\frac{x}{|x|} - \theta\right| < \gamma , \
|x| \ge 1 .
\end{equation}
Then $X(\mathbb{R}^n)=Y(\mathbb{R}^n,w)$ satisfies the weak doubling property 
and there exists a constant $C > 0$ such that
\begin{equation}\label{eq:subexponential-growth-2}
f_X(R) \ge C \exp\big(R\varphi(R)\big)
\end{equation}
for all sufficiently large $R$, where the function 
$f_X:(0,\infty)\to(0,\infty)$ is given by \eqref{eq:fundamental-function}.
\end{lemma}
\begin{proof}
Let us show that  \eqref{eq:subexponential-growth} is satisfied for 
$R_j=\varphi^{-1/2}(j)$ and  $y_j = (j +m) \theta$ with a sufficiently large 
$m > 0$. Indeed, since 
\[
\frac{R_j}{j} 
= 
\frac{\varphi^{1/2}(j)}{j\varphi(j)} 
\le 
\frac{\varphi^{1/2}(j)}{\varphi(1)} \to 0 
\ \mbox{ as } \  j \to \infty ,
\]
the balls $B(y_j, \tau R_j)$ lie in the cone 
$\left|\frac{x}{|x|} - \theta\right| < \gamma$ 
provided $m$ is sufficiently large. Then for all $j\in\mathbb{N}$,
\begin{align*}
\frac
{\displaystyle\operatornamewithlimits{ess\,sup}_{x \in B(y_j, \tau R_j)} w(x)}
{\displaystyle\operatornamewithlimits{ess\,inf}_{x \in B(y_j, R_j)} w(x)} 
&\le
\frac{C_1 \exp\left[
\left(j + m + \tau\varphi^{-1/2}(j)\right) 
\varphi\left(j + m + \tau\varphi^{-1/2}(j)\right)
\right]}
{C_0 \exp\left[
\left(j + m - \varphi^{-1/2}(j)\right) 
\varphi\left(j + m - \varphi^{-1/2}(j)\right)
\right]} 
\\
&\le 
\frac
{C_1 \exp\left[
\left(j + m + \tau\varphi^{-1/2}(j)\right) 
\varphi(j + m)
\right]}
{C_0  \exp\left[
\left(j + m - \varphi^{-1/2}(j)\right) \varphi(j + m)
\right] }
\\
&= 
\frac{C_1}{C_0}\, 
\exp\left[(\tau + 1)\omega^{-1/2}(j) \varphi(j + m)\right] 
\\
&\le 
\frac{C_1}{C_0}\, \exp\left[(\tau + 1)\varphi^{1/2}(j)\right]
\le 
\frac{C_1}{C_0}\, \exp\left[(\tau + 1)\varphi^{1/2}(1)\right] .
\end{align*}
By Lemma~\ref{le:subexponential-growth}, $X(\mathbb{R}^n)=Y(\mathbb{R}^n,w)$ 
satisfies the weak doubling property. Since $Y(\mathbb{R}^n)$ is
translation-invariant and $r\varphi(r)$ is nonincreasing,
it follows from \eqref{2sided} that
\begin{align*}
f_X(R)
&
\ge 
\|\chi_{B((R - 1)\theta, 1)}\|_{Y(\mathbb{R}^n, w)} 
\ge 
C_0 \exp((R -2) \varphi(R - 2))
\|\chi_{B((R - 1)\theta, 1)}\|_{Y(\mathbb{R}^n)} 
\\
&= C_0 \exp((R -2) \varphi(R - 2))
\|\chi_{B(0, 1)}\|_{Y(\mathbb{R}^n)}
\ge C_0\exp(R\varphi(R))\|\chi_{B(0, 1)}\|_{Y(\mathbb{R}^n)}
\end{align*}
for all sufficiently large $R$, i.e. \eqref{eq:subexponential-growth-2} is 
satisfied with $C = C_0 \|\chi_{B(0, 1)}\|_{Y(\mathbb{R}^n)}$.
\end{proof}
Let $1 \le p \le \infty$. It follows from Lemma~\ref{le:subexponential-growth}
that $X(\mathbb{R})=L^p(\mathbb{R},w)$ has the weak doubling 
property if $w:\mathbb{R}\to[0,\infty]$ satisfies
$w,1/w\in L_{\rm loc}^\infty(\mathbb{R})$ and $w(x)$ is equal to, e.g., 
$(1 + x)^\alpha$, $\alpha > 0$; $\exp\left(x^\beta\right)$, $\beta \in (0, 1)$;  
or  $\exp\left(\frac{x}{\log\log(3 + x)}\right)$ for $x>0$.
On the other hand, Theorems \ref{th:example-nondoubling-exponent} and 
\ref{th:example-nondoubling-exponent-modulus} imply that $L^p(\mathbb{R},w)$
does not satisfy the weak doubling property 
for $w(x) = e^{cx}$  or $w(x) = e^{c|x|}$, $x \in \mathbb{R}$, 
with any $c>0$.
\section{Proof of the main result}
\label{sec:proof}
\subsection{The case of $a\in L_{1,\sigma}(\mathbb{R}^n)$ for some
$\sigma\in\mathbb{R}$}
\begin{theorem}\label{th:symbol-in-weighted-L1}
Let $X(\mathbb{R}^n)$ be a Banach function space satisfying the weak doubling
property. If $\sigma \in \mathbb{R}$ and
$a \in\mathcal{M}_{X(\mathbb{R}^n)}\cap L_{1,\sigma}(\mathbb{R}^n)$, then
$a \in L^\infty(\mathbb{R}^n)$ and
\begin{equation}\label{eq:symbol-in-weighted-L1-0}
\|a\|_{L^\infty(\mathbb{R}^n)}
\le
\|a\|_{\mathcal{M}_{X(\mathbb{R}^n)}}.
\end{equation}
\end{theorem}
\begin{proof}
Let $D_{X,\varrho}$ be defined for all $\varrho>1$ by \eqref{eq:def-D}.
If, for some $\varrho>1$, the quantity $D_{X,\varrho}$ is infinite, then it is
obvious that
\begin{equation}\label{eq:symbol-in-weighted-L1-1}
\|a\|_{L^\infty(\mathbb{R}^n)}
\le
D_{X, \varrho} \|a\|_{\mathcal{M}_{X(\mathbb{R}^n)}}.
\end{equation}
Since $X(\mathbb{R}^n)$ satisfies the weak doubling property, there 
exists $\varrho>1$ such that $D_{X,\varrho}<\infty$. Take an arbitrary
Lebesgue point $\eta \in \mathbb{R}^n$ of the function $a$. Let an even
function $\varphi \in C^\infty_0(\mathbb{R}^n)$ satisfy the following
conditions:
\[
0 \le \varphi \le 1,
\quad
\varphi(x) = 1
\mbox{ for }
|x| \le 1,
\quad
\varphi(x) = 0
\mbox{ for }
|x| \ge \varrho.
\]
Let
\[
f_{\delta, \eta}(x)
:=
e^{-i\eta x} \varphi(\delta x),
\quad
x \in \mathbb{R}^n,
\quad
\delta > 0,
\]
and
\[
f_{\delta, \eta, y}(x)
:=
f_{\delta, \eta}(x - y),
\quad
y \in \mathbb{R}^n.
\]
Then
\begin{align*}
(F f_{\delta, \eta, y})(\xi)
&=
e^{i\xi y} (F f_{\delta, \eta})(\xi)
=
e^{i\xi y} \delta^{-n} (F\varphi)\left(\frac{\xi - \eta}{\delta}\right)
=
e^{i\xi y} (F\varphi)_\delta(\xi - \eta)
=
e^{i\xi y} (F\varphi)_\delta(\eta - \xi)
\end{align*}
and
\begin{align*}
\left(F^{-1} aF f_{\delta, \eta, y}\right)(x)
&=
\frac{1}{(2\pi)^n} \int_{\mathbb{R}^n} e^{i(y - x)\xi} a(\xi)
(F\varphi)_\delta(\eta - \xi)\,d\xi ,
\\
a(\eta) f_{\delta, \eta, y}(x)
&=
\frac{1}{(2\pi)^n} \int_{\mathbb{R}^n}
e^{i(y - x)\xi} a(\eta) (F\varphi)_\delta(\eta - \xi)\,d\xi.
\end{align*}
Hence, for all $x,y\in\mathbb{R}^n$ and $\delta>0$,
\begin{align*}
\left|
\left(F^{-1}aFf_{\delta,\eta,y}\right)(x) - a(\eta) f_{\delta,\eta, y}(x)
\right|
& =
\frac{1}{(2\pi)^n}
\left|
\int_{\mathbb{R}^n}
e^{i(y - x)\xi} (a(\xi) -  a(\eta)) (F\varphi)_\delta(\eta - \xi)\, d\xi
\right|
\\
&\le
\frac{1}{(2\pi)^n} \int_{\mathbb{R}^n}
|a(\xi) -  a(\eta)|\,|(F\varphi)_\delta(\eta - \xi)|\,d\xi.
\end{align*}
Since $F\varphi \in S(\mathbb{R}^n)$ and $\eta$ is a Lebesgue point of $a$,
it follows from Lemma~\ref{le:Lebesgue} that for any $\varepsilon > 0$
there exists $\delta_\varepsilon > 0$ such that for all
$x, y \in \mathbb{R}^n$ and all $\delta \in (0, \delta_\varepsilon)$,
\[
\left|
\left(F^{-1} aF f_{\delta,\eta,y}\right)(x) - a(\eta) f_{\delta,\eta,y}(x)
\right| < \varepsilon.
\]
It is clear that
$|f_{\delta, \eta, y}|  \chi_{B(y, 1/\delta)} = \chi_{B(y, 1/\delta)}$.
Then the above inequality implies that for all $y \in \mathbb{R}^n$ and
$\delta \in (0, \delta_\varepsilon)$,
\[
|a(\eta)| \chi_{B(y, 1/\delta)}
\le
\left|F^{-1} aF f_{\delta,\eta,y}\right| + \varepsilon \chi_{B(y, 1/\delta)}.
\]
Hence
\begin{align}
|a(\eta)| \left\|\chi_{B(y, 1/\delta)} \right\|_{X(\mathbb{R}^n)}
&\le
\left\|F^{-1} aF f_{\delta, \eta, y}\right\|_{X(\mathbb{R}^n)}
+
\varepsilon \left\|\chi_{B(y, 1/\delta)}\right\|_{X(\mathbb{R}^n)}
\nonumber \\
&\le
\|a\|_{\mathcal{M}_{X(\mathbb{R}^n)}}
\left\|f_{\delta, \eta, y}\right\|_{X(\mathbb{R}^n)}
+
\varepsilon \left\|\chi_{B(y, 1/\delta)}\right\|_{X(\mathbb{R}^n)}
\nonumber \\
&\le
\|a\|_{\mathcal{M}_{X(\mathbb{R}^n)}}
\left\|\chi_{B(y, \varrho/\delta)}\right\|_{X(\mathbb{R}^n)}
+
\varepsilon \|\chi_{B(y, 1/\delta)}\|_{X(\mathbb{R}^n)}.
\label{eq:symbol-in-weighted-L1-2}
\end{align}
Since $D_{X,\varrho}<\infty$, the definition of $D_{X,\varrho}$
given in \eqref{eq:def-D} implies that there exist
$\delta \in (0, \delta_\varepsilon)$ and $y \in \mathbb{R}^n$ such that
\[
\frac{\left\|\chi_{B(y, \varrho/\delta)}\right\|_{X(\mathbb{R}^n)}}
{\left\|\chi_{B(y, 1/\delta)}\right\|_{X(\mathbb{R}^n)}}
\le
D_{X, \varrho} + \varepsilon.
\]
Choosing these $\delta$ and $y$, and dividing both sides of inequality
\eqref{eq:symbol-in-weighted-L1-2} by
$\left\|\chi_{B(y, 1/\delta)}\right\|_{X(\mathbb{R}^n)}$, we get
\[
|a(\eta)| \le  (D_{X, \varrho} + \varepsilon)
\|a\|_{\mathcal{M}_{X(\mathbb{R}^n)}} + \varepsilon
\quad\mbox{for all}\quad \varepsilon > 0.
\]
Hence, for all Lebesgue points $\eta \in \mathbb{R}^n$ of the function $a$,
we have
\[
|a(\eta)| \le  D_{X, \varrho} \|a\|_{\mathcal{M}_{X(\mathbb{R}^n)}}.
\]
Since $a\in L_{1,\sigma}(\mathbb{R}^n)\subset L_{\rm loc}^1(\mathbb{R}^n)$,
almost all points $\eta\in\mathbb{R}^n$ are Lebesgue points of the function
$a$ in view of the Lebesgue differentiation theorem (see, e.g.,
\cite[Corollary~2.1.16 and Exercise~2.1.10]{G14-classical}).
Therefore $a \in L^\infty(\mathbb{R}^n)$ and inequality
\eqref{eq:symbol-in-weighted-L1-1} is fulfilled for all $\varrho>1$.
It is now left to apply Lemma~\ref{le:weak-doubling-property}.
\end{proof}
\subsection{Proof of Theorem~\ref{th:main} for arbitrary
$a\in\mathcal{M}_{X(\mathbb{R}^n)}\subset S'(\mathbb{R}^n)$}
For a function $w\in S(\mathbb{R}^n)$, we will use the following notation
\[
\widehat{w} := Fw, 
\quad 
\check{w} := F^{-1} w, 
\quad
\widetilde{w}(\xi) := w(-\xi),
\quad 
(\tau_\zeta w)(\xi) := w(\xi - \zeta), 
\quad  e_\zeta(x) := e^{i \zeta x}.
\]
Let a nonnegative even function $\psi\in C^\infty_0(\mathbb{R}^n)$ satisfy 
the condition
\[
\int_{\mathbb{R}^n} \psi(\xi)\, d\xi = 1, 
\]
and let 
\[
\psi_\varepsilon(\xi) := \varepsilon^{-n} \psi(\xi/\varepsilon),
\quad
\varepsilon>0.
\] 
Fix $\varepsilon>0$ and take arbitrary functions 
$u \in S(\mathbb{R}^n)\cap X(\mathbb{R}^n)$ 
and $v \in C^\infty_0(\mathbb{R}^n)$. Then we have
$\widehat{u},\widehat{v}\in S(\mathbb{R}^n)$. In view of
\cite[Theorem~7.19(b)]{R91} or \cite[Theorem~2.3.20]{G14-classical},
we observe that 
\begin{equation}\label{eq:main-proof-1}
a\ast\psi_\varepsilon\in 
C_{\rm poly}^\infty(\mathbb{R}^n)\subset S'(\mathbb{R}^n)
\end{equation}
because $a\in S'(\mathbb{R}^n)$ and $\psi_\varepsilon\in S(\mathbb{R}^n)$.
Then $(a\ast\psi_\varepsilon)\widehat{u}\in S(\mathbb{R}^n)$.
By \cite[Theorem~2.2.14]{G14-classical},
\begin{align}
\int_{\mathbb{R}^n}\left(F^{-1}(a\ast\psi_\varepsilon)Fu\right)(x) v(x)\,dx
&=
\int_{\mathbb{R}^n}
\left(F^{-1}(a\ast\psi_\varepsilon)\widehat{u}\right)
\widehat{\hspace{2mm}}(\xi) \check{v}(\xi)\,d\xi
\nonumber\\
&=
\int_{\mathbb{R}^n}(a\ast\psi_\varepsilon)(\xi)\widehat{u}(\xi)\check{v}(\xi)\,d\xi .
\label{eq:main-proof-2}
\end{align}
Observe that $\widehat{u}\check{v}\in S(\mathbb{R}^n)$ in view of
\cite[Proposition~2.2.7]{G14-classical}. Then, taking into account
that $a\in S'(\mathbb{R}^n)$ and $\psi_\varepsilon\in S(\mathbb{R}^n)$,
it follows from \cite[Theorem~2.2.14]{G14-classical} 
or \cite[Theorem~7.19(d)]{R91}
that
\begin{align}
\int_{\mathbb{R}^n}(a\ast\psi_\varepsilon)(\xi)\widehat{u}(\xi)\check{v}(\xi)\,d\xi
&= 
\big((a\ast\psi_\varepsilon)\ast(\widehat{u} \check{v})^{\sim}\big) (0) 
= 
\big(
\left(a\ast(\widehat{u} \check{v})^{\sim}\right)\ast\psi_\varepsilon
\big)(0)
\nonumber\\
&= 
\int_{\mathbb{R}^n} 
\left(a\ast(\widehat{u} \check{v})^{\sim}\right)(\zeta)
\widetilde{\psi}_\varepsilon(\zeta)\,d\zeta
=
\int_{\mathbb{R}^n} 
\left\langle a, \tau_\zeta (\widehat{u} \check{v})\right\rangle 
\psi_\varepsilon(\zeta)\, 
d\zeta .
\label{eq:main-proof-3}
\end{align}
Since $\tau_\zeta\widehat{u}\in S(\mathbb{R}^n)
\subset C_{\rm poly}^\infty(\mathbb{R}^n)$, by the definition of 
multiplication of $a\in S'(\mathbb{R}^n)$ by a function in
$C_{\rm poly}^\infty(\mathbb{R}^n)$, we have
\begin{equation}\label{eq:main-proof-4}
\langle a,\tau_\zeta(\widehat{u}\check{v})\rangle
=
\langle a,\tau_\zeta\widehat{u}\cdot\tau_\zeta\check{v}\rangle
=
\langle a\tau_\zeta\widehat{u},\tau_\zeta\check{v}\rangle.
\end{equation}
It is easy to see that $\tau_\zeta\widehat{u}=F(e_{-\zeta}u)$ and
$\tau_\zeta\check{v}=F^{-1}(e_\zeta v)$. Then
\begin{equation}\label{eq:main-proof-5}
\langle a,\tau_\zeta(\widehat{u}\check{v})\rangle
=
\langle aF(e_{-\zeta}u),F^{-1}(e_\zeta v)\rangle.
\end{equation}
By the definition of the inverse Fourier transform of 
$aF(e_{-\zeta}u)\in S'(\mathbb{R}^n)$, we have
\begin{equation}\label{eq:main-proof-6}
\langle aF(e_{-\zeta}u),F^{-1}(e_\zeta v)\rangle
=
\langle F^{-1}aF(e_{-\zeta}u),e_\zeta v\rangle.
\end{equation}
Combining \eqref{eq:main-proof-2}--\eqref{eq:main-proof-6}, we arrive
at the following equality:
\begin{equation}\label{eq:main-proof-7}
\int_{\mathbb{R}^n}\big(F^{-1}(a\ast\psi_\varepsilon)Fu\big)(x)v(x)\,dx
=
\int_{\mathbb{R}^n} \langle F^{-1}aF(e_{-\zeta}u),e_\zeta v\rangle
\psi_\varepsilon(\zeta)\,d\zeta.
\end{equation}

Let $s \in S_0(\mathbb{R}^n)$ be such that $\|s\|_{X'(\mathbb{R}^n)}\le 1$.
Put $K := \operatorname{supp}s$ and consider a function
$\phi\in C_0^\infty(\mathbb{R}^n)$ such that  $0\le \phi\le 1$.
By Lemma~\ref{le:approximation-L-infinity}, there exists a sequence
$\{v_j\}_{j\in\mathbb{N}}\subset C_0^\infty(\mathbb{R}^n)$ such that
$\operatorname{supp}v_j\subseteq K^*$ and 
$\|v_j\|_{L^\infty(\mathbb{R}^n)}\le\|s\|_{L^\infty(\mathbb{R}^n)}$ for all
$j\in\mathbb{N}$ and $v_j\to s$ a.e. on $\mathbb{R}^n$ as $j\to\infty$.
Since $(a\ast\psi_\varepsilon)Fu$ belongs to $S(\mathbb{R}^n)$, we have
$F^{-1}(a\ast\psi_\varepsilon)Fu \in S(\mathbb{R}^n)\subset L^1(\mathbb{R}^n)$.
Therefore $\phi\cdot F^{-1}(a\ast\psi_\varepsilon)Fu \cdot s$
also belongs to $L^1(\mathbb{R}^n)$ and it follows from the Lebesgue dominated 
convergence theorem that
\begin{equation}\label{eq:main-proof-8}
\lim_{j \to \infty}\int_{\mathbb{R}^n}
\phi(x)\big(F^{-1}(a\ast\psi_\varepsilon)Fu\big)(x)v_j(x)\,dx 
= 
\int_{\mathbb{R}^n}
\phi(x)\big(F^{-1}(a\ast\psi_\varepsilon)Fu\big)(x)s(x)\,dx.
\end{equation}
Further, $|v_j| \le \|s\|_{L^\infty(\mathbb{R}^n)}\chi_{K^*}$ 
for all $j\in\mathbb{N}$. Since $a\in\mathcal{M}_{X(\mathbb{R}^n)}$, one has 
$F^{-1}aF(e_{-\zeta}u) \in X(\mathbb{R}^n)$, and it follows from axiom (A5) 
that $\left(F^{-1}aF(e_{-\zeta}u)\right) \chi_{K^*} \in L^1(\mathbb{R}^n)$. 
Hence, in view of the Lebesgue dominated convergence theorem,
for all $\zeta \in \mathbb{R}^n$,
\begin{equation}\label{eq:main-proof-9}
\lim_{j \to \infty}
\langle \phi F^{-1}aF(e_{-\zeta}u),e_\zeta v_j\rangle 
= 
\int_{\mathbb{R}^n}
\phi(x)\left(F^{-1}aF(e_{-\zeta}u)\right)(x) \left(e_\zeta s\right)(x)\,dx.
\end{equation}
H\"older's inequality for $X(\mathbb{R}^n)$ (see 
\cite[Chap.~1, Theorem~2.4]{BS88}) implies for all $\zeta\in\mathbb{R}^n$,
\begin{align}
\left|\langle \phi F^{-1} aF(e_{-\zeta}u),e_\zeta v_j\rangle\right|
&\le  
\|s\|_{L^\infty(\mathbb{R}^n)} 
\int_{\mathbb{R}^n}
\left|\left(F^{-1}aF(e_{-\zeta}u)\right)(x)\right| 
\chi_{K^*}(x)\,dx
\nonumber\\
&\le 
\|s\|_{L^\infty(\mathbb{R}^n)} 
\|F^{-1}aF(e_{-\zeta}u)\|_{X(\mathbb{R}^n)} \|\chi_{K^*}\|_{X'(\mathbb{R}^n)}
\nonumber\\
&\le 
\|s\|_{L^\infty(\mathbb{R}^n)} \|a\|_{\mathcal{M}_{X}(\mathbb{R}^n)} 
\|e_{-\zeta}u\|_{X(\mathbb{R}^n)} \|\chi_{K^*}\|_{X'(\mathbb{R}^n)}
\nonumber\\
&=
\|a\|_{\mathcal{M}_{X}(\mathbb{R}^n)} 
\|u\|_{X(\mathbb{R}^n)} \|s\|_{L^\infty(\mathbb{R}^n)} 
\|\chi_{K^*}\|_{X'(\mathbb{R}^n)}.
\label{eq:main-proof-10}
\end{align}
Taking into account \eqref{eq:main-proof-9}--\eqref{eq:main-proof-10} and
using the Lebesgue dominated convergence theorem again, one gets
\begin{align}
\lim_{j \to \infty} & \int_{\mathbb{R}^n} 
\langle \phi F^{-1}aF(e_{-\zeta}u),e_\zeta v_j\rangle
\psi_\varepsilon(\zeta)\,d\zeta 
\notag\\
&=
\int_{\mathbb{R}^n} \left( \int_{\mathbb{R}^n}
\phi(x)\left(F^{-1}aF(e_{-\zeta}u\right)(x) \left(e_\zeta s\right)(x)\,dx
\right) \psi_\varepsilon(\zeta)\,d\zeta.
\label{eq:main-proof-11}
\end{align}
It follows from  \eqref{eq:main-proof-8}, \eqref{eq:main-proof-11}, and 
\eqref{eq:main-proof-7} with $\phi v_j$ in place of $v$ that
\begin{align}
\int_{\mathbb{R}^n}\phi(x)
&
\big(F^{-1}(a\ast\psi_\varepsilon)Fu\big)(x)s(x)\,dx
\notag\\
&=
\int_{\mathbb{R}^n} \left(\int_{\mathbb{R}^n}
\phi(x)
\left(F^{-1}aF(e_{-\zeta}u\right)(x) \left(e_\zeta s\right)(x)\,dx\right)
\psi_\varepsilon(\zeta)\,d\zeta.
\label{eq:main-proof-12}
\end{align}
Similarly to \eqref{eq:main-proof-10} one gets
\begin{align}
\left|
\int_{\mathbb{R}^n}\phi(x)
\left(F^{-1}aF(e_{-\zeta}u\right)(x) \left(e_\zeta s\right)(x)\,dx
\right|
&\le
\|F^{-1}aF(e_{-\zeta}u)\|_{X(\mathbb{R}^n)} \|e_\zeta s\|_{X'(\mathbb{R}^n)}
\nonumber\\
&\le 
\|a\|_{\mathcal{M}_{X}(\mathbb{R}^n)} 
\|u\|_{X(\mathbb{R}^n)} \|s\|_{X'(\mathbb{R}^n)}.
\label{eq:main-proof-13}
\end{align}
Equality \eqref{eq:main-proof-12} and inequality \eqref{eq:main-proof-13}
immediately yield that for all $u\in S(\mathbb{R}^n)\cap X(\mathbb{R}^n)$,
all $s \in S_0(\mathbb{R}^n)$ satisfying $\|s\|_{X'(\mathbb{R}^n)}\le 1$,
and all $\phi\in C_0^\infty(\mathbb{R}^n)$ satisfying $0\le\phi\le 1$, one has
\begin{equation}\label{eq:main-proof-14}
\left|\int_{\mathbb{R}^n}\phi(x)
\big(F^{-1}(a\ast\psi_\varepsilon)Fu\big)(x)s(x)\,dx
\right|
\le 
\|a\|_{\mathcal{M}_{X}(\mathbb{R}^n)} 
\|u\|_{X(\mathbb{R}^n)} \|s\|_{X'(\mathbb{R}^n)}.
\end{equation}
Now, take functions $\phi_j \in C^\infty_0(\mathbb{R}^n)$ such that 
$0 \le \phi_j \le 1$ and $\phi_j(x) = 1$ for all $|x| \le j$
and all $j \in \mathbb{N}$. 
Since $F^{-1}(a*\psi_\varepsilon)Fu \in S(\mathbb{R}^n)$, we have
$\phi_j F^{-1} (a\ast\psi_\varepsilon)Fu \in C^\infty_0(\mathbb{R}^n)$
for all $j\in\mathbb{N}$. Inequality \eqref{eq:main-proof-14} implies that
for all $s \in S_0(\mathbb{R}^n)$ satisfying $\|s\|_{X'(\mathbb{R}^n)}\le 1$
and all $j\in\mathbb{N}$, one has
\[
\left|
\int_{\mathbb{R}^n}\phi_j(x) F^{-1} (a\ast\psi_\varepsilon)Fu(x) s(x)\, dx
\right|
\le 
\|a\|_{\mathcal{M}_{X(\mathbb{R}^n)}} 
\|u\|_{X(\mathbb{R}^n)} \|s\|_{X'(\mathbb{R}^n)}.
\]
Hence it follows from Lemma~\ref{le:NFP-X} that for all $j\in\mathbb{N}$,
\[
\|\phi_j F^{-1} (a\ast\psi_\varepsilon)Fu\|_{X(\mathbb{R}^n)}
\le
\|a\|_{\mathcal{M}_{X(\mathbb{R}^n)}} \|u\|_{X(\mathbb{R}^n)} .
\]
Since the functions $\phi_j F^{-1} (a\ast\psi_\varepsilon)Fu$ converge to
$F^{-1} (a\ast\psi_\varepsilon)Fu$ everywhere as $j \to \infty$,
Fatou's lemma (see \cite[Chap. 1, Lemma 1.5]{BS88}) implies that
$F^{-1} (a\ast\psi_\varepsilon)Fu \in X(\mathbb{R}^n)$ and
\[
\|F^{-1} (a\ast\psi_\varepsilon)Fu\|_{X(\mathbb{R}^n)}
\le
\|a\|_{\mathcal{M}_{X(\mathbb{R}^n)}} \|u\|_{X(\mathbb{R}^n)}
\]
for all $u \in S(\mathbb{R}^n)\cap X(\mathbb{R}^n)$. Thus
\[
\|a\ast\psi_\varepsilon\|_{\mathcal{M}_{X(\mathbb{R}^n)}}
\le
\|a\|_{\mathcal{M}_{X(\mathbb{R}^n)}}.
\]
Since
$C_{\rm poly}^\infty(\mathbb{R}^n)\subset 
L_{1, \sigma}(\mathbb{R}^n)$ for some
$\sigma \in \mathbb{R}$ and 
$a\ast\psi_\varepsilon\in C_{\rm poly}^\infty(\mathbb{R}^n)$,
it follows from Theorem~\ref{th:symbol-in-weighted-L1} that
$a\ast\psi_\varepsilon \in L^\infty(\mathbb{R}^n)$ and
\[
\|a\ast\psi_\varepsilon\|_{L^\infty(\mathbb{R}^n)}
\le
\|a\ast\psi_\varepsilon\|_{\mathcal{M}_{X(\mathbb{R}^n)}}
\le
\|a\|_{\mathcal{M}_{X(\mathbb{R}^n)}}
\quad\mbox{for all}\quad
\varepsilon>0.
\]
It is not difficult to see that $w\ast\psi_\varepsilon$ converges to $w$ in 
$S(\mathbb{R}^n)$ as $\varepsilon \to 0$ for every $w\in S(\mathbb{R}^n)$
(see, e.g., \cite[Exercise~2.3.2]{G14-classical}). Then,
for all $w \in S(\mathbb{R}^n)$,
\begin{align*}
|\langle a, w\rangle|
&=
\lim_{\varepsilon\to 0}|\langle a,w\ast\psi_\varepsilon\rangle|
=
\lim_{\varepsilon \to 0} |\langle a\ast\psi_\varepsilon, w\rangle| 
=
\lim_{\varepsilon\to 0}\left|
\int_{\mathbb{R}^n} (a\ast \psi_\varepsilon)(x)w(x)\,dx
\right|
\\
&\le
\limsup_{\varepsilon\to 0}
\|a\ast\psi_\varepsilon\|_{L^\infty(\mathbb{R}^n)} \|w\|_{L^1(\mathbb{R}^n)}
\le 
\|a\|_{\mathcal{M}_{X(\mathbb{R}^n)}} \|w\|_{L^1(\mathbb{R}^n)} .
\end{align*}
Hence $a$ can be extended to a bounded linear functional on 
$L^1(\mathbb{R}^n)$, i.e., it can be identified with a function in 
$L^\infty(\mathbb{R}^n)$ and 
$\|a\|_{L^\infty}\le\|a\|_{\mathcal{M}_{X(\mathbb{R}^n)}}$ holds. 

Suppose there exists a constant $D_X > 0$ such that
$\|a\|_{L^\infty(\mathbb{R}^n)} \le D_X \|a\|_{\mathcal{M}_{X(\mathbb{R}^n)}}$
for all $a \in\mathcal{M}_{X(\mathbb{R}^n)}$.
Then taking $a \equiv 1$, one gets $D_X \ge 1$. So, the constant 
$D_X = 1$ in the estimate 
$\|a\|_{L^\infty}\le\|a\|_{\mathcal{M}_{X(\mathbb{R}^n)}}$
is best possible, which completes the proof of Theorem~\ref{th:main}.
\subsection{Multidimensional analogue of Theorem~\ref{th:BGB}}
\begin{corollary}\label{co:main-reflexive-X-with-AX}
Suppose $n\ge 1$ and $X(\mathbb{R}^n)$ is a Banach function space
satisfying the $A_X$-condition. If
$a\in\mathcal{M}_{X(\mathbb{R}^n)} \subset S'(\mathbb{R}^n)$,
then $a\in L^\infty(\mathbb{R}^n)$ and
\[
\|a\|_{L^\infty(\mathbb{R}^n)}\le \|a\|_{\mathcal{M}_{X(\mathbb{R}^n)}}.
\]
The constant $1$ on the right-hand side in the above inequality is best
possible.
\end{corollary}
\begin{proof}
This statement follows from Theorem~\ref{th:main} and
Lemma~\ref{le:AX-implies-doubling}.
\end{proof}
We conclude this section with the proof of the following multidimensional
analogue of Theorem~\ref{th:BGB}.
\begin{corollary}
Let $n\ge 1$ and $1<p<\infty$. If $w\in A_p(\mathbb{R}^n)$ and
$a\in\mathcal{M}_{L^p(\mathbb{R}^n,w)} \subset S'(\mathbb{R}^n)$,
then $a\in L^\infty(\mathbb{R}^n)$ and
\[
\|a\|_{L^\infty(\mathbb{R}^n)}\le \|a\|_{\mathcal{M}_{L^p(\mathbb{R}^n,w)}}.
\]
The constant $1$ on the right-hand side in the above inequality is best
possible.
\end{corollary}
\begin{proof}
Since $w\in A_p(\mathbb{R}^n)$, we have
\begin{align*}
\|w\|_{A_p(\mathbb{R}^n)}
&=
\sup_{Q}\frac{1}{|Q|}
\|\chi_Q\|_{L^p(\mathbb{R}^n,w)}\|\chi_Q\|_{L^{p'}(\mathbb{R}^n,w^{-1})}
\\
&=
\sup_Q
\left(\frac1{|Q|} \int_Q w^p(x)\, dx\right)^{1/p}
\left(\frac1{|Q|} \int_Q w^{-p'}(x)\, dx\right)^{1/p'} 
< \infty,
\end{align*}
where $1/p+1/p'=1$ and the supremum is taken over all cubes with sides 
parallel to the axes. Thus, the Banach function space 
$X(\mathbb{R}^n)=L^p(\mathbb{R}^n,w)$  satisfies  the $A_X$-condition. 
It remains to apply Corollary~\ref{co:main-reflexive-X-with-AX}.
\end{proof}
We would like to stress again that the $A_p$ condition implies the doubling 
property (see Lemma \ref{le:AX-implies-doubling}), which places much stronger 
restrictions on the behaviour of the weight $w$ at infinity than the weak 
doubling property (see Subsection \ref{subsec:examples-doubling}). Note also 
that the latter puts no restrictions on the local behaviour of the weight 
$w$. When dealing with weighted function spaces $Y(\mathbb{R}^n, w)$, we 
usually assume that $w \in Y_{\mathrm{loc}}(\mathbb{R}^n)$ and 
$1/w  \in Y'_{\mathrm{loc}}(\mathbb{R}^n)$. It is instructive to compare the 
latter conditions with $w\in A_p(\mathbb{R}^n)$ in the case $n = 1$,
$Y(\mathbb{R}) = L^p(\mathbb{R})$ and
\[
w(x) =\left\{\begin{array}{cr}
|x|^\alpha ,   &  -1 \le x < 0 ,  
\\
 x^\beta ,   &   0 \le  x \le 1,
\\
1 ,   & |x|  > 1,
\end{array}\right. 
\quad\ \alpha, \beta \in \mathbb{R} .
\]
It is easy to see that $w \in L^p_{\mathrm{loc}}(\mathbb{R})$ and 
$1/w  \in L^{p'}_{\mathrm{loc}}(\mathbb{R})$ if and only if 
$-1/p < \alpha, \beta < 1 - 1/p$. On the other hand, 
$w \not\in A_p(\mathbb{R})$ if $\alpha \not= \beta$
(see \cite[Example~2.6]{BK97}). 
\section{On optimality of the requirement of the weak doubling property in 
Theorem~\ref{th:main}}
\label{sec:optimality-doubling}
\subsection{Estimates for convolutions}
\begin{theorem}[{(Cf. \cite[p. 91]{L83})}]\label{th:Lofstrom}
Let $w^*$ be a weight such that $w^*\in L^1_\mathrm{loc}(\mathbb{R}^n)$,
let $\Omega\subseteq\mathbb{R}^n$ be a set of positive measure, and let
\[
L^1_\Omega(\mathbb{R}, w^*) := 
\left\{
h \in L^1(\mathbb{R}, w^*)\ : \ \operatorname{supp} h \subseteq\Omega 
\right\}.
\]
Suppose $Y(\mathbb{R}^n)$ is a translation-invariant Banach function space and 
$w$ is a weight satisfying $w \in Y_{\mathrm{loc}}(\mathbb{R}^n)$ and
$1/w  \in Y'_{\mathrm{loc}}(\mathbb{R}^n)$. If
\begin{equation}\label{eq:Lofstrom-1}
w^*(y) w(x - y) w^{-1}(x) \ge 1 
\quad\mbox{for all}\quad
x \in \mathbb{R}^n,\ 
y \in \Omega,
\end{equation}
then 
\[
\|\kappa\ast f\|_{Y(\mathbb{R}^n, w)} 
\le 
\|\kappa\|_{L^1(\mathbb{R}^n, w^*)} 
\|f\|_{Y(\mathbb{R}^n, w)}
\quad\mbox{for all}\quad
f \in Y(\mathbb{R}^n, w),\
\kappa \in L^1_\Omega(\mathbb{R}, w^*).
\]
\end{theorem}
\begin{proof}
Since $w \in Y_{\mathrm{loc}}(\mathbb{R}^n)$ and 
$1/w  \in Y'_{\mathrm{loc}}(\mathbb{R}^n)$, we recall that
$Y(\mathbb{R}^n,w)$ is a Banach function space and $Y'(\mathbb{R}^n,w^{-1})$
is its associate space in view of \cite[Lemma~2.4]{KS14}.
Using \eqref{eq:Lofstrom-1} and H\"older's inequality for Banach function 
spaces (see \cite[Chap. 1, Theorem 2.4]{BS88})
and taking into account that $Y(\mathbb{R}^n)$ is translation-invariant, 
one gets for all  $g \in Y'(\mathbb{R}^n, w^{-1})$,
\begin{align*}
\left|\int_{\mathbb{R}^n} (\kappa\ast f)(x) g(x)\, dx\right| 
&\le 
\int_{\mathbb{R}^n} 
\left(\int_\Omega |\kappa(y)|\, |f(x - y)|\, dy\right) |g(x)|\, dx 
\\
&\le
\int_\Omega w^*(y) |\kappa(y)| 
\left(
\int_{\mathbb{R}^n} w(x - y) |f(x - y)|\cdot w^{-1}(x) |g(x)|\, dx
\right)\, dy 
\end{align*}
\begin{align*}
&\le 
\int_{\mathbb{R}^n} w^*(y) |\kappa(y)| 
\|\tau_y (w f)\|_{Y(\mathbb{R}^n)} 
\left\|w^{-1} g\right\|_{Y'(\mathbb{R}^n)}\, dy
\\
&=  
\|w f\|_{Y(\mathbb{R}^n)} 
\left\|w^{-1} g\right\|_{Y'(\mathbb{R}^n)}
\int_{\mathbb{R}^n}  w^*(y) |\kappa(y)|\, dy 
\\
&= 
\|\kappa\|_{L^1(\mathbb{R}^n, w^*)} 
\|f\|_{Y(\mathbb{R}^n, w)} \|g\|_{Y'(\mathbb{R}^n, w^{-1})}.
\end{align*}
By \cite[Chap.~1, Theorem~2.7 and Lemma~2.8]{BS88},
the above inequality implies that
\begin{align*}
\|\kappa\ast f\|_{Y(\mathbb{R}^n,w)}
&=
\sup\left\{\left|\int_{\mathbb{R}^n} (\kappa\ast f)(x) g(x)\, dx\right|\ :\
g\in Y'(\mathbb{R}^n,w^{-1}), \
\|g\|_{Y'(\mathbb{R}^n,w^{-1})}\le 1
\right\}
\\
&\le 
\|\kappa\|_{L^1(\mathbb{R}^n, w^*)} 
\|f\|_{Y(\mathbb{R}^n, w)},
\end{align*}
which completes the proof.
\end{proof}
\begin{corollary}\label{co:w1w2}
Suppose $Y(\mathbb{R})$ is a translation-invariant Banach function space. 
\begin{enumerate}
\item[{\rm(a)}] 
Let $w_1(x) = e^{cx}$ for $x \in \mathbb{R}$, with some constant $c > 0$.
Then for all $\kappa \in L^1(\mathbb{R}, w_1)$ and all 
$f \in Y(\mathbb{R}, w_1)$,
\[
\|\kappa\ast f\|_{Y(\mathbb{R}, w_1)} 
\le 
\|\kappa\|_{L^1(\mathbb{R}, w_1)} \|f\|_{Y(\mathbb{R}, w_1)}.
\]

\item[{\rm(b)}] Suppose $\varphi:\mathbb{R}\to\mathbb{R}$
is a function such that $\varphi(x) = x$ for $x \le 0$, and
\begin{equation}\label{eq:w1w2-1}
\frac{\varphi(z) - \varphi(x)}{z - x} \ge 1 
\quad\mbox{for all}\quad
z > x \ge 0 .
\end{equation}
Let $w_2(x) := e^{c\varphi(x)}$ for $x \in \mathbb{R}$, with some constant 
$c > 0$. Then for all $\kappa \in L^1_{(-\infty, 0]}(\mathbb{R}, w_2)$
and all $f \in Y(\mathbb{R}, w_2)$,
\[
\|\kappa\ast f\|_{Y(\mathbb{R}, w_2)} 
\le 
\|\kappa\|_{L^1(\mathbb{R}, w_2)} 
\|f\|_{Y(\mathbb{R}, w_2)}.
\]
\end{enumerate}
\end{corollary}
\begin{proof}
(a) Since
\[
w_1(y) w_1(x - y) w_1^{-1}(x) = e^{c(y + (x - y) -x)} = 1 
\quad\mbox{for all}\quad 
x, y \in \mathbb{R},
\]
the conditions of Theorem~\ref{th:Lofstrom} are satisfied for 
$w = w^* = w_1$ and $\Omega = \mathbb{R}$.

(b) Take any $y < 0$ and any $x \in \mathbb{R}$. If $x \ge 0$, then it 
follows from \eqref{eq:w1w2-1} that
\[
\frac{\varphi(x - y)-\varphi(x)}{(x - y) - x} \ge 1 
\ \Longrightarrow \  
y + \varphi(x - y) - \varphi(x) \ge 0 .
\]
If $y < x < 0$, then 
\[
\frac{\varphi(x - y) - \varphi(0)}{(x - y) - 0} \ge 1 
\ \Longrightarrow \  
y + \varphi(x - y) - x \ge 0 .
\]
Finally, if $x \le y < 0$, then
\[
\varphi(y) + \varphi(x - y) - \varphi(x) 
= 
y + (x - y) - x = 0 .
\]
Hence
\begin{equation}\label{eq:w1w2-2}
w_2(y) w_2(x - y) w_2^{-1}(x) 
= 
e^{c(
\varphi(y)+\varphi(x - y) - \varphi(x)
)} \ge 1 
\quad\mbox{for all}\quad x \in \mathbb{R} 
\end{equation}
and all $y < 0$. It is clear that \eqref{eq:w1w2-2} holds for $y = 0$ as well.

Since $w_2$ and $1/w_2$ are locally bounded, we have
$w_2 \in Y_{\mathrm{loc}}(\mathbb{R})\cap L_{\mathrm{loc}}^1(\mathbb{R})$ and 
$1/w_2\in Y'_{\mathrm{loc}}(\mathbb{R})$. Hence the conditions of 
Theorem~\ref{th:Lofstrom} are satisfied for $w = w^* = w_2$ and 
$\Omega = (-\infty, 0]$.
\end{proof}
\subsection{Banach function spaces  with unbounded Fourier multipliers}
\label{subsec:unbounded-Fourier-multipliers}
Let $Y(\mathbb{R})$ be a translation-invariant Banach function space and $w$ 
be a weight. We know that the weak doubling property of the space 
$X(\mathbb{R}) = Y(\mathbb{R}, w)$ allows the weight $w$ to grow at any 
subexponential rate (see Subsection~\ref{subsec:examples-doubling}). It is 
natural to ask whether Theorem~\ref{th:main} still holds for 
$X(\mathbb{R}) = Y(\mathbb{R}, w)$ with the exponential weight $w(x) = e^{cx}$ 
with $c>0$. We show in this subsection that the answer is negative and that 
$\mathcal{M}_{Y(\mathbb{R},w)}$ contains many unbounded functions in this 
case. This means that the weak doubling property is optimal in a sense.  
(This also provides an alternative  indirect proof of  
Theorem~\ref{th:example-nondoubling-exponent}).
\begin{theorem}\label{th:unbounded-multipliers}
Let $Y(\mathbb{R})$ be a translation-invariant Banach function
space and the weights $w_1$ and $w_2$ be the same as in 
Corollary~\ref{co:w1w2}. Then
\begin{equation}
F\left(S'(\mathbb{R})\cap L^1(\mathbb{R},w_1)\right)
\subseteq \mathcal{M}_{Y(\mathbb{R},w_1)} , 
\quad
F\big(S'(\mathbb{R})\cap L^1_{(-\infty, 0]}(\mathbb{R},w_2)\big)
\subseteq \mathcal{M}_{Y(\mathbb{R},w_2)}.
\label{eq:unbounded-multipliers-1}
\end{equation}
\end{theorem}
\begin{proof}
If $a\in F\left(S'(\mathbb{R})\cap L^1(\mathbb{R},w_1)\right)$, then
$F^{-1}a\in L^1(\mathbb{R},w_1)$. It follows from 
Corollary \ref{co:w1w2} that for every function
$u\in S(\mathbb{R})\cap Y(\mathbb{R},w_1)$,
\begin{equation}\label{eq:unbounded-multipliers-3}
\|F^{-1}aFu\|_{Y(\mathbb{R},w_1)}=\|(F^{-1}a)\ast u\|_{Y(\mathbb{R},w_1)}
\le \|F^{-1}a\|_{L^1(\mathbb{R},w_1)}\|u\|_{Y(\mathbb{R},w_1)}.
\end{equation}
Hence 
\begin{equation}\label{eq:unbounded-multipliers-4}
\|a\|_{\mathcal{M}_{Y(\mathbb{R},w_1)}}\le\|F^{-1}a\|_{L^1(\mathbb{R},w_1)}.
\end{equation}
The same argument allows one to show that if 
$a\in F\big(S'(\mathbb{R})\cap L^1_{(-\infty, 0]}(\mathbb{R},w_2)\big)$, then
\begin{equation}\label{eq:unbounded-multipliers-5}
\|a\|_{\mathcal{M}_{Y(\mathbb{R},w_2)}}\le\|F^{-1}a\|_{L^1(\mathbb{R},w_2)}.
\end{equation}
Inequalities 
\eqref{eq:unbounded-multipliers-4}--\eqref{eq:unbounded-multipliers-5}
imply embeddings \eqref{eq:unbounded-multipliers-1}.
\end{proof}
\begin{lemma}
Let $Y(\mathbb{R})$ be a translation-invariant Banach function
space and the weights $w_1$ and $w_2$ be the same as in 
Corollary~\ref{co:w1w2}.
\begin{enumerate}
\item[(a)]
If  $a\in F\left(L^1(\mathbb{R})\cap L^1(\mathbb{R},w_1)\right)$, then
\begin{equation}\label{eq:unbounded-multipliers-6}
\|a\|_{\mathcal{M}^0_{Y(\mathbb{R},w_1)}}\le\|F^{-1}a\|_{L^1(\mathbb{R},w_1)}. 
\end{equation}
\item[(b)]
If $a \in F\big(L^1_{(-\infty,0]}(\mathbb{R})\big)=
F\big(L^1(\mathbb{R})\cap L^1_{(-\infty, 0]}(\mathbb{R},w_2)\big)$, 
then
\begin{equation}\label{eq:unbounded-multipliers-7}
\|a\|_{\mathcal{M}^0_{Y(\mathbb{R},w_2)}}\le\|F^{-1}a\|_{L^1(\mathbb{R},w_2)}.
\end{equation}
\end{enumerate}
\end{lemma}
\begin{proof}
If $a\in F\left(L^1(\mathbb{R})\cap L^1(\mathbb{R},w_1)\right)$, then
\eqref{eq:unbounded-multipliers-3} holds for every 
$u\in L^2(\mathbb{R})\cap Y(\mathbb{R},w_1)$, which implies 
\eqref{eq:unbounded-multipliers-6} and completes the proof of part (a).
The proof of part (b) is analogous.
\end{proof}
In view of Theorem~\ref{th:unbounded-multipliers},
$\mathcal{M}_{Y(\mathbb{R},w_j)}$, $j = 1,2$ 
contain many unbounded functions. Let us give some concrete examples.

Let $\mathbb{C}_+:=\{z=x+iy\in\mathbb{C}:y>0\}$ and $H^2(\mathbb{C}_+)$
be the Hardy space of all functions $f$ analytic in $\mathbb{C}_+$ such that
$|f(x+iy)|^2$ is integrable for each $y>0$ and
\[
\sup_{y>0}\int_{\mathbb{R}}|f(x+iy)|^2\,dx<\infty.
\]
If $f\in H^2(\mathbb{C}_+)$, 
then the boundary function
\[
f(x)=\lim_{y\to 0+}f(x+iy)
\]
exists a.e. on $\mathbb{R}$ and belongs to $L^2(\mathbb{R})$ (see, e.g., 
the corollary of \cite[Theorem~1.1]{D70}). Let $H_+^2(\mathbb{R})$ be the 
Hardy space of the boundary functions of the functions 
$f\in H^2(\mathbb{C}_+)$.
\begin{corollary}\label{co:examples-of-unbounded-multipliers}
Let $Y(\mathbb{R})$ be a translation-invariant Banach function space and $w$ 
be one of the weights $w_1$ and $w_2$ in Corollary~\ref{co:w1w2}. 
\begin{enumerate}
\item[{\rm (a)}]
The Hardy space $H^2_+(\mathbb{R})$ corresponding to the upper complex 
half-plane is a subset of the space $\mathcal{M}_{Y(\mathbb{R}, w)}$.

\item[{\rm (b)}]
Let
\[
a_{-\alpha}(\xi) := \lim_{\varepsilon \to 0+} (\xi + i\varepsilon)^{-\alpha}, 
\quad
\xi \in \mathbb{R}\setminus\{0\} , \ 0 < \alpha < 1 ,
\]
where $z^{-\alpha}$ denotes a branch analytic in the complex plane cut along 
the negative half-line. Then $a_{-\alpha} \in \mathcal{M}_{Y(\mathbb{R}, w)}$.
\end{enumerate}
\end{corollary}
\begin{proof}
(a) It follows from H\"older's inequality that any function in 
$L^2(\mathbb{R})$, which vanishes on $(0, \infty)$, belongs to 
$L^1(\mathbb{R},w)$. By the Paley-Wiener theorem
(see, e.g., \cite[Corollary to Theorem~11.9]{D70}), if 
$u\in H_+^2(\mathbb{R})$, then $Fu(\xi)$ vanishes for almost all $\xi<0$.
Hence $F^{-1}u(\xi)$ vanishes for almost all $\xi>0$. Therefore 
$F^{-1}u\in L^1(\mathbb{R},w)$ and $u\in FL^1(\mathbb{R},w)$.
Now we can apply Theorem~\ref{th:unbounded-multipliers} to complete
the proof of part (a).

(b) Since $0<\alpha<1$, it follows from 
\cite[Example 2.3, formula $(2.41')$]{E81}
that
\[
a_{-\alpha}=k_\alpha F^{-1}f_{\alpha-1}^+=k_\alpha Ff_{\alpha-1}^-,
\]
where $k_\alpha\in\mathbb{C}$ is some constant depending on $\alpha$ and
\[
f_{\alpha-1}^+(x):=\left\{\begin{array}{ccc}
x^{\alpha-1} &\mbox{if} & x>0,
\\
0 &\mbox{if}& x<0,
\end{array}\right.
\quad
f_{\alpha-1}^-(x):=\left\{\begin{array}{ccc}
0 &\mbox{if} & x>0,
\\
|x|^{\alpha-1} &\mbox{if}& x<0,
\end{array}\right.
\]
define regular distributions in $S'(\mathbb{R})$ (see \cite[Example~1.6]{E81}).
It is clear that $f_{\alpha-1}^-\in L^1(\mathbb{R},w)$, whence
$a_{-\alpha}\in F\left(S'(\mathbb{R})\cap L^1(\mathbb{R},w)\right)$.
It remains to apply Theorem~\ref{th:unbounded-multipliers}.
\end{proof}
\section{Classes $\mathcal{M}^0_{X(\mathbb{R}^n)}$ and 
$\mathcal{M}_{X(\mathbb{R}^n)}\cap L^\infty(\mathbb{R}^n)$}
\label{sec:alternative-definition}
\subsection{Two classes of Fourier multipliers coincide in the case of a nice 
underlying space}
\begin{theorem}\label{th:M0}
If a Banach function space $X(\mathbb{R}^n)$ satisfies the bounded 
$L^2$-approximation property, then
\begin{equation}\label{eq:M0-1}
\mathcal{M}^0_{X(\mathbb{R}^n)}
=
\mathcal{M}_{X(\mathbb{R}^n)}
\cap L^\infty(\mathbb{R}^n)
\quad\mbox{ and }\quad
\|a\|_{\mathcal{M}_{X(\mathbb{R}^n)}}
=
\|a\|_{\mathcal{M}^0_{X(\mathbb{R}^n)}} .
\end{equation}
\end{theorem}
\begin{proof}
According to \eqref{eq:multiplier-embedding} and \eqref{eq:multiplier-est}, we 
only need to prove that
\[
\mathcal{M}^0_{X(\mathbb{R}^n)}
\supseteq
\mathcal{M}_{X(\mathbb{R}^n)} \cap L^\infty(\mathbb{R}^n)
\quad \mbox{ and } \quad
\|a\|_{\mathcal{M}_{X(\mathbb{R}^n)}}
\ge
\|a\|_{\mathcal{M}^0_{X(\mathbb{R}^n)}}.
\]
Choose any function 
$a \in \mathcal{M}_{X(\mathbb{R}^n)}\cap L^\infty(\mathbb{R}^n)$.
Take any function $u \in L^2(\mathbb{R}^n)\cap X(\mathbb{R}^n)$ and consider
a sequence $\{u_j\}_{j\in\mathbb{N}} \subset C^\infty_0(\mathbb{R}^n)$
satisfying \eqref{eq:bounded-L2-approximation}. Since 
$a \in L^\infty(\mathbb{R}^n)$ and 
$\mathcal{M}_{L^2(\mathbb{R}^n)}^0=L^\infty(\mathbb{R}^n)$
(see, e.g., \cite[Theorem~2.5.10]{G14-classical}), we have
\[
\lim_{j\to\infty}\|F^{-1} aFu - F^{-1} aFu_j\|_{L^2(\mathbb{R}^n)}=0. 
\]
Then there exists a subsequence $\{F^{-1} aFu_{j_k}\}_{k \in \mathbb{N}}$
of the sequence $\{F^{-1} aFu_j\}_{j \in \mathbb{N}}$ that
converges to $F^{-1} aFu$ almost everywhere. Since 
$a\in\mathcal{M}_{X(\mathbb{R}^n)}$, from the inequality in
\eqref{eq:bounded-L2-approximation} we get
\begin{align*}
\liminf_{k \to \infty} \|F^{-1} aFu_{j_k}\|_{X(\mathbb{R}^n)}
&\le 
\|a\|_{\mathcal{M}_{X(\mathbb{R}^n)}}
\liminf_{k \to \infty} \|u_{j_k}\|_{X(\mathbb{R}^n)}
\\
&\le 
\|a\|_{\mathcal{M}_{X(\mathbb{R}^n)}}
\limsup_{j \to \infty} \|u_j\|_{X(\mathbb{R}^n)}
\le
\|a\|_{\mathcal{M}_{X(\mathbb{R}^n)}} \|u\|_{X(\mathbb{R}^n)}.
\end{align*}
Fatou's lemma (see \cite[Chap. 1, Lemma 1.5]{BS88}) 
and the above inequality imply that the function
$F^{-1} aFu$  belongs to $X(\mathbb{R}^n)$ and
for $u \in L^2(\mathbb{R}^n)\cap X(\mathbb{R}^n)$,
\[
\|F^{-1} aFu\|_{X(\mathbb{R}^n)}
\le 
\liminf_{k \to \infty} \|F^{-1} aFu_{j_k}\|_{X(\mathbb{R}^n)}
\le
\|a\|_{\mathcal{M}_{X(\mathbb{R}^n)}} \|u\|_{X(\mathbb{R}^n)}.
\]
Hence
$a \in \mathcal{M}^0_{X}(\mathbb{R}^n)$  and
$\|a\|_{\mathcal{M}^0_{X(\mathbb{R}^n)}} \le 
\|a\|_{\mathcal{M}_{X(\mathbb{R}^n)}}$.
\end{proof}
\subsection{Two classes of Fourier multipliers are different in general}
The following theorem shows that equalities in \eqref{eq:M0-1} do not always 
hold. 
\begin{theorem}\label{th:two-classes-of-multipliers-are-different}
For a compact set $G \subset [0, 1]^n$ of positive measure with empty 
interior and a sequence $b=\{b_m\}_{m\in\mathbb{N}} \subset (0, 1)$
satisfying 
\begin{equation}\label{eq:two-classes-of-multipliers-are-different-0}
\lim_{m \to \infty} b_m = 0,
\end{equation}
consider the sequence
\begin{align}
G_m 
&:= 
(2m, 0, \dots, 0) + G 
\nonumber\\
&= 
\left\{(y_1 + 2m, y_2, \dots, y_n) \in \mathbb{R}^n\ : \ 
y = (y_1, y_2, \dots, y_n) \in G\right\} , \quad m \in \mathbb{N}, 
\label{eq:two-classes-of-multipliers-are-different-1}
\end{align}
and define the weight $w_{G,b}$ for $y=(y_1,\dots, y_n)\in\mathbb{R}^n$
satisfying $y_1\ge 0$ by
\begin{equation}\label{eq:two-classes-of-multipliers-are-different-2}
w_{G,b}(y) 
:= 
\left\{\begin{array}{cl}
    b_m,  & y \in G_m,\, m \in \mathbb{N}  ,  \\
    1, &   y \not\in \bigcup_{m \in \mathbb{N}} G_m ,
\end{array}\right.  
\end{equation}
and for $y=(y_1,\dots,y_n)\in\mathbb{R}^n$ satisfying $y_1<0$ by
\begin{equation}\label{eq:two-classes-of-multipliers-are-different-3}
w_{G,b}(y):=w_{G,b}(-y). 
\end{equation}
Then there
exists $a \in S(\mathbb{R}^n)$ such that 
$a \in \mathcal{M}_{L^\infty(\mathbb{R}^n,w_{G,b})}\setminus 
\mathcal{M}^0_{L^\infty(\mathbb{R}^n,w_{G,b})}$.
\end{theorem}
\begin{proof} 
Let $x^{(0)} \in G$ be a Lebesgue point of the function $\chi_G$ and let 
$\{\varrho_j\}_{j\in\mathbb{N}}$ be the sequence defined by 
\eqref{eq:mollification}. It follows from Lemma~\ref{le:Lebesgue} that there 
exists $j \in \mathbb{N}$ for which
\[
\int_{\mathbb{R}^n} |\chi_G(y) - \chi_G(x^{(0)})| \varrho_j(x^{(0)} - y)\, dy 
\le 
\frac12.
\]
Hence,
\[
\int_{\mathbb{R}^n} \varrho_j(x^{(0)} - y)\, dy 
- 
\int_{\mathbb{R}^n} \chi_G(y) \varrho_j(x^{(0)} - y)\, dy \le \frac12.
\]
Therefore, $1 - (\varrho_j \ast \chi_G)(x^{(0)}) \le 1/2$, whence
\begin{equation}\label{eq:two-classes-of-multipliers-are-different-4}
(\varrho_j \ast \chi_G)(x^{(0)}) \ge \frac12. 
\end{equation}
Let $a := F\varrho_j \in S(\mathbb{R}^n)$. Take any $u \in S(\mathbb{R}^n)$ 
with $\|u\|_{L^\infty(\mathbb{R}^n,w_{G,b})} \le 1$. 
The same argument as in the proof of Lemma~\ref{le:exotic-weight} shows that 
$|u(y)| \le 1$ for all $y \in \mathbb{R}^n$. Then by \cite[Theorem~7.8(b)]{R91},
\[
\left|(F^{-1}aF u)(x)\right| 
= 
|(\varrho_j \ast u)(x)| \le \int_{\mathbb{R}^n} \varrho_j(y)\, dy = 1 
\quad\mbox{for all}\quad
x  \in \mathbb{R}^n 
\]
and $\|F^{-1}aF u\|_{L^\infty(\mathbb{R}^n, w_{G,b})} \le 1$. 
Hence $a \in \mathcal{M}_{L^\infty(\mathbb{R}^n, w_{G,b})}$ and
$\|a\|_{\mathcal{M}_{L^\infty(\mathbb{R}^n, w_{G,b})}} \le 1$.

On the other hand, consider
\[
u_m :=  b_m^{-1}\,  \chi_{G_m} , \quad m \in \mathbb{N}.
\]
It is clear that $\|u_m\|_{L^\infty(\mathbb{R}^n,w_{G,b})} = 1$. 
Since $u_m\in L^1(\mathbb{R}^n)\subset S'(\mathbb{R}^n)$ and 
$a=F\varrho_j\in S(\mathbb{R}^n)$, it follows from \cite[Theorem~7.19(c)]{R91}
that 
$(\varrho_j\ast u_m)\widehat{\hspace{2mm}}
=\widehat{\varrho_j}\widehat{u}_m=aFu_m$. This equality and
\cite[Propositions~4.18, 4.20]{B11} imply that 
$F^{-1}aFu_m=\varrho_j\ast u_m\in C_0^\infty(\mathbb{R}^n)$ because
$\varrho_j\in C_0^\infty(\mathbb{R}^n)$ and $u_m\in L^\infty(\mathbb{R}^n)$
has compact support. Therefore,
\[
v_m
:= 
\|F^{-1}aF u_m\|_{L^\infty(\mathbb{R}^n, w_{G,b})}^{-1} 
F^{-1}aF u_m \in 
C_0^\infty(\mathbb{R}^n)\subset
S(\mathbb{R}^n) 
\]
and $\|v_m\|_{L^\infty(\mathbb{R}^n,w_{G,b})} = 1$.
Then, as above, $|v_m(x)| \le 1$ for all $x \in \mathbb{R}^n$
and $m\in\mathbb{N}$, i.e.
\begin{equation}\label{eq:two-classes-of-multipliers-are-different-5}
\left|F^{-1}aF u_m(x)\right| \le 
\|F^{-1}aF u_m\|_{L^\infty(\mathbb{R}^n, w_{G,b})}
\quad\mbox{for all}\quad
x  \in \mathbb{R}^n,
\
m\in\mathbb{N}.
\end{equation}
Let
\[
x^{(m)} :=  (2m, 0, \dots, 0) + x^{(0)},
\quad
m\in\mathbb{N}.
\]
Then, taking into account 
\eqref{eq:two-classes-of-multipliers-are-different-4}, 
we get for all $m\in\mathbb{N}$,
\begin{equation}\label{eq:two-classes-of-multipliers-are-different-6}
F^{-1}aF u_m\big(x^{(m)}\big) 
= 
b_m^{-1}  \left(\varrho_j \ast \chi_{G_m}\right) \big(x^{(m)}\big) 
= 
b_m^{-1}  \left(\varrho_j \ast \chi_{G}\right) \big(x^{(0)}\big) 
\ge 
\frac{1}{2b_m}.
\end{equation}
Now it follows from \eqref{eq:two-classes-of-multipliers-are-different-5},
\eqref{eq:two-classes-of-multipliers-are-different-6}
and \eqref{eq:two-classes-of-multipliers-are-different-0} that
\[
\|F^{-1}aF u_m\|_{L^\infty(\mathbb{R}^n,w_{G,b})} 
\ge 
\frac{1}{2b_m} \to \infty 
\quad \mbox{as} \quad m \to \infty ,
\]
while $\|u_m\|_{L^\infty(\mathbb{R}^n,w_{G,b})} = 1$. Hence 
$a \not\in \mathcal{M}^0_{L^\infty(\mathbb{R}^n, w_{G,b})}$.
\end{proof}
\subsection{Normed algebras of Fourier multipliers}
\begin{lemma}
Let $X(\mathbb{R}^n)$ be a Banach function space.
Then the set $\mathcal{M}^0_{X(\mathbb{R}^n)}$ is a normed algebra with
respect to the norm $\|\cdot\|_{\mathcal{M}^0_{X(\mathbb{R}^n)}}$ and
\[
\|ab\|_{\mathcal{M}^0_{X(\mathbb{R}^n)}} 
\le 
\|a\|_{\mathcal{M}^0_{X(\mathbb{R}^n)}}
\|b\|_{\mathcal{M}^0_{X(\mathbb{R}^n)}} 
\quad\mbox{for all}\quad a, b \in \mathcal{M}^0_{X(\mathbb{R}^n)}.
\]
\end{lemma}
\begin{proof}
The proof is straightforward.
\end{proof}
The proof of the following result requires a bit more effort.
\begin{theorem} \label{th:algebra-M-L-infty}
Let $X(\mathbb{R}^n)$ be a Banach function space. Then the set
$\mathcal{M}_{X(\mathbb{R}^n)}\cap L^\infty(\mathbb{R}^n)$ is a normed
algebra with respect to the norm $\|\cdot\|_{\mathcal{M}_{X(\mathbb{R}^n)}}$ 
and
\begin{equation}\label{eq:algebra-M-L-infty-1}
\|ab\|_{\mathcal{M}_{X(\mathbb{R}^n)}} 
\le 
\|a\|_{\mathcal{M}_{X(\mathbb{R}^n)}}
\|b\|_{\mathcal{M}_{X(\mathbb{R}^n)}} 
\quad\mbox{for all}\quad
a, b \in \mathcal{M}_{X(\mathbb{R}^n)}\cap L^\infty(\mathbb{R}^n).
\end{equation}
\end{theorem}
\begin{proof} 
It is clear that $\mathcal{M}_{X(\mathbb{R}^n)}\cap L^\infty(\mathbb{R}^n)$ 
is a normed space, so one only needs to prove that 
$ab \in \mathcal{M}_{X(\mathbb{R}^n)}\cap L^\infty(\mathbb{R}^n)$ for all 
$a, b \in \mathcal{M}_{X(\mathbb{R}^n)}\cap L^\infty(\mathbb{R}^n)$
and that \eqref{eq:algebra-M-L-infty-1} holds.

Take any function $u \in S(\mathbb{R}^n)\cap X(\mathbb{R}^n)$. 
Then  $F^{-1}bF u \in X(\mathbb{R}^n)$.  
Since $u \in S(\mathbb{R}^n)\subset L^2(\mathbb{R}^n)$ and 
$b \in L^\infty(\mathbb{R}^n)$, we have $F^{-1}bF u \in L^2(\mathbb{R}^n)$.
On the other hand, in view of \cite[Theorem 7.19(a)]{R91}, 
$u \in S(\mathbb{R}^n)$ and $b \in S'(\mathbb{R}^n)$ imply that
$F^{-1}bF u = \left(F^{-1}b\right)\ast u \in C^\infty(\mathbb{R}^n)$.
Hence $F^{-1}bF$ maps $S(\mathbb{R}^n)\cap X(\mathbb{R}^n)$ into
$X(\mathbb{R}^n)\cap L^2(\mathbb{R}^n)\cap C^\infty(\mathbb{R}^n)$.

Now take a function
$v \in X(\mathbb{R}^n)\cap L^2(\mathbb{R}^n)\cap C^\infty(\mathbb{R}^n)$ 
and consider the sequence $v_m := \phi_m v$, where
$\phi_m \in C_0^\infty(\mathbb{R}^n)$, 
$0 \le \phi_m \le 1$, and $\phi_{m}(x) = 1$ for 
$|x| \le {m}$ and $m \in \mathbb{N}$. Since 
$|v_m| \le |v|$, it follows from axiom (A2) that 
$v_m \in X(\mathbb{R}^n)$ and 
$\|v_m\|_{X(\mathbb{R}^n)} \le \|v\|_{X(\mathbb{R}^n)}$ 
for all $m \in \mathbb{N}$.
Further, $v_m \in C_0^\infty(\mathbb{R}^n)$ and 
$\|v_m - v\|_{L^2(\mathbb{R}^n)} \to 0$ as 
$m \to \infty$. Then it follows as in the proof of 
Theorem~\ref{th:M0} that $F^{-1}aF v \in X(\mathbb{R}^n)$ and
\[
\|F^{-1}aF v\|_{X(\mathbb{R}^n)} 
\le 
\|a\|_{\mathcal{M}_{X(\mathbb{R}^n)}}\|v\|_{X(\mathbb{R}^n)} .
\]
Taking $v = F^{-1}bF u$, one gets 
$F^{-1}a bF u = F^{-1}aF (F^{-1}bF u) \in X(\mathbb{R}^n)$ and
\[
\|F^{-1}a bF u\|_{X(\mathbb{R}^n)} 
\le 
\|a\|_{\mathcal{M}_{X(\mathbb{R}^n)}}\|F^{-1}bF u\|_{X(\mathbb{R}^n)} 
\le
\|a\|_{\mathcal{M}_{X(\mathbb{R}^n)}} 
\|b\|_{\mathcal{M}_{X(\mathbb{R}^n)}} 
\|u\|_{X(\mathbb{R}^n)}  
\]
for all $u \in S(\mathbb{R}^n)\cap X(\mathbb{R}^n)$,
which immediately implies \eqref{eq:algebra-M-L-infty-1}.
\end{proof}
Theorem~\ref{th:algebra-M-L-infty} allows one to prove that if there exists 
a constant $D_X > 0$ such that
\begin{equation}\label{eq:with-any-constant}
\|a\|_{L^\infty(\mathbb{R}^n)} 
\le   
D_X \|a\|_{\mathcal{M}_{X(\mathbb{R}^n)}}
\quad\mbox{for all}\quad 
a \in\mathcal{M}_{X(\mathbb{R}^n)}\cap L^\infty(\mathbb{R}^n),
\end{equation}
then in fact
\begin{equation}\label{eq:with-constant-one}
\|a\|_{L^\infty(\mathbb{R}^n)} \le  \|a\|_{\mathcal{M}_{X(\mathbb{R}^n)}} 
\quad\mbox{for all}\quad 
a \in\mathcal{M}_{X(\mathbb{R}^n)}\cap L^\infty(\mathbb{R}^n).
\end{equation}
Indeed,  one can apply \eqref{eq:with-any-constant} 
and Theorem~\ref{th:algebra-M-L-infty} 
to the function $a^m$ with $m \in \mathbb{N}$ to get
\[
\|a\|_{L^\infty(\mathbb{R}^n)}^m
=
\|a^m\|_{L^\infty(\mathbb{R}^n)}
\le
D_X \|a^m\|_{\mathcal{M}_{X(\mathbb{R}^n)}}
\le
D_X \|a\|_{\mathcal{M}_{X(\mathbb{R}^n)}}^m .
\]
Taking $m \to \infty$ in the inequality
\[
\|a\|_{L^\infty(\mathbb{R}^n)} 
\le  
D_X^{1/m} \|a\|_{\mathcal{M}_{X(\mathbb{R}^n)}},
\]
one gets \eqref{eq:with-constant-one}. One can use this observation instead of 
Lemma~\ref{le:weak-doubling-property} to derive 
\eqref{eq:symbol-in-weighted-L1-0} from  \eqref{eq:symbol-in-weighted-L1-1} 
under the assumption that a Banach function space $X(\mathbb{R}^n)$ 
satisfies the weak doubling property. In particular, this implies that
\cite[Theorem~2.3]{BG98} holds with the constant $K_{p,C}=1$.
It is also clear that the implication 
\eqref{eq:with-any-constant} $\Rightarrow$ \eqref{eq:with-constant-one}
holds with $\mathcal{M}^0_{X(\mathbb{R}^n)}$ in place of 
$\mathcal{M}_{X(\mathbb{R}^n)}$.
\subsection{Normed spaces of Fourier multipliers are not complete in general}
\begin{theorem}
Let $Y(\mathbb{R})$ be a translation-invariant Banach function space and $w$ 
be one of the weights $w_1$ and $w_2$ in Corollary~\ref{co:w1w2}.
\begin{enumerate}
\item[{\rm (a)}]
The normed space $\mathcal{M}_{Y(\mathbb{R}, w)}$ is not complete with respect 
to the norm $\|\cdot\|_{\mathcal{M}_{Y(\mathbb{R}, w)}}$.

\item[{\rm (b)}]
The normed algebra $\mathcal{M}^0_{Y(\mathbb{R}, w)}$ is not complete with 
respect to the norm $\|\cdot\|_{\mathcal{M}^0_{Y(\mathbb{R}, w)}}$.
\end{enumerate}
\end{theorem}
\begin{proof}
(a) Consider the function $g_0(x) := e^{-cx/2}\phi_0(x)$, 
where the constant $c>0$ is from the definition of the weights
$w_1$ and $w_2$, and a function $\phi_0 \in C^\infty(\mathbb{R})$ is such 
that $\phi_0(x) = 0$ for $x \ge 0$ and  $\phi_0(x) = 1$ for $x \le -1$. 
It is easy to see that $g_0\in L^1(\mathbb{R},w)$, whence
it may be identified with the distribution in $\mathcal{D}'(\mathbb{R})$.
On the other hand, $g_0\notin S'(\mathbb{R})$ (cf. 
\cite[Chap.~7, Exercise~3]{R91}).

Consider $g_k(x) := \phi_k(x) g_0(x)$, where 
$\phi_{k} \in C_0^\infty(\mathbb{R})$,  $0 \le \phi_{k} \le 1$, and 
$\phi_{k}(x) = 1$ for  $|x| \le {k}$ and all $k\in\mathbb{N}$.
The Lebesgue dominated convergence theorem implies that
$\|g_{k} - g_0\|_{L^1(\mathbb{R}, w)} \to 0$ as ${k} \to \infty$. 
Hence $\{g_k\}_{k\in\mathbb{N}}$ is a Cauchy sequence in 
$L^1(\mathbb{R},w)$. Let $a_k :=  F g_k$ for $k \in \mathbb{N}$. 

It follows from 
\eqref{eq:unbounded-multipliers-4}--\eqref{eq:unbounded-multipliers-5} that
\[
\|a_k-a_m\|_{\mathcal{M}_{Y(\mathbb{R},w)}}
\le
\|F^{-1}(a_k-a_m)\|_{L^1(\mathbb{R},w)}
=
\|g_k-g_m\|_{L^1(\mathbb{R},w)}
\quad\mbox{for all}\quad k,m\in\mathbb{N}.
\]
Therefore, $\{a_k\}_{k \in \mathbb{N}}$ is a Cauchy sequence in 
$\mathcal{M}_{Y(\mathbb{R}, w)}$. Suppose it converges to a 
limit $a_0$ in $\mathcal{M}_{Y(\mathbb{R}, w)}$. Then the sequence
$F^{-1}a_{k}Fu = g_{k}\ast u$ converges to $F^{-1}a_0Fu = (F^{-1}a_0)\ast u$ in 
$Y(\mathbb{R}, w)$ as ${k} \to \infty$ for any function
$u \in C_0^\infty(\mathbb{R}) \subset S(\mathbb{R})\cap Y(\mathbb{R}, w)$.

On the other hand, Corollary \ref{co:w1w2} implies that $g_{k}\ast u$ converges 
to $g_0\ast u$ in $Y(\mathbb{R}, w)$. Hence $(F^{-1}a_0)\ast u = g_0\ast u$ for 
any $u \in C_0^\infty(\mathbb{R})$. Since $(F^{-1}a_0)\ast u = g_0\ast u$ are 
continuous functions (see, e.g., \cite[Theorem 6.30(b)]{R91}), one gets
\[
\langle F^{-1}a_0, \widetilde{u}\rangle 
= 
\left((F^{-1}a_0)\ast u\right)(0) 
= 
\left(g_0\ast u\right)(0) = \langle g_0, \widetilde{u}\rangle 
\quad\mbox{for all}\quad
u \in C_0^\infty(\mathbb{R}) .
\]
Hence the distributions $F^{-1}a_0 \in S'(\mathbb{R})$ and 
$g_0 \in \mathcal{D}'(\mathbb{R})\setminus S'(\mathbb{R})$ are equal to each 
other. This contradiction shows that $\{a_k\}_k \in \mathbb{N}$ does not 
converge to a limit in $\mathcal{M}_{Y(\mathbb{R}, w)}$.

(b) Consider the functions $g_{m}(x) := e^{x/m}\, f^-_{\alpha - 1}(x)$, 
$m \in \mathbb{N}$, where $f^-_{\alpha - 1}$ is the same as in the proof of 
Corollary~\ref{co:examples-of-unbounded-multipliers}(b). The Lebesgue dominated 
convergence theorem implies that
\[
\|g_m - f^-_{\alpha - 1}\|_{L^1(\mathbb{R}, w)} \to 0 
\quad\mbox{as}\quad m \to \infty,
\quad\quad
\|g_m-g_k\|_{L^1(\mathbb{R},w)}\to 0
\quad\mbox{as}\quad m,k\to\infty.
\]
Then it follows from  
\eqref{eq:unbounded-multipliers-4}--\eqref{eq:unbounded-multipliers-7}
and Corollary~\ref{co:examples-of-unbounded-multipliers}(b) that 
\begin{equation}\label{eq:non-complete-1}
\|a_{m} - a_{-\alpha}\|_{\mathcal{M}_{Y(\mathbb{R}, w)}}
\le
k_\alpha\|g_m-f^-_{\alpha-1}\|_{L^1(\mathbb{R},w)}
\to 0 
\quad\mbox{as}\quad {m} \to \infty
\end{equation}
and
\begin{equation}\label{eq:non-complete-2}
\|a_m-a_k\|_{\mathcal{M}_{Y(\mathbb{R},w)}^0}
\le 
k_\alpha\|g_m-g_k\|_{L^1(\mathbb{R},w)}\to 0
\quad\mbox{as}\quad m,k\to\infty,
\end{equation}
where 
\[
a_{m}(\xi) 
:= 
k_\alpha F g_{m}(\xi) 
=
k_\alpha F f^-_{\alpha - 1}(\xi + i/{m}) 
= 
(\xi + i/{m})^{-\alpha},
\]
$k_\alpha$ is the constant from the proof of 
Corollary~\ref{co:examples-of-unbounded-multipliers}(b), and 
\[
a_{-\alpha}(\xi) 
:= 
k_\alpha F f^-_{\alpha - 1}(\xi) 
= 
\lim_{\varepsilon \to 0+} (\xi + i\varepsilon)^{-\alpha}. 
\]
Since $\{g_m\}_{m\in\mathbb{N}}$ is convergent in $L^1(\mathbb{R},w)$,
it follows from \eqref{eq:non-complete-2} that $\{a_m\}_{m\in\mathbb{N}}$
is a Cauchy sequence in  $\mathcal{M}^0_{Y(\mathbb{R}, w)}$. If it had a 
limit there, then inequality \eqref{eq:multiplier-est} would imply that it 
converges to the same limit in $\mathcal{M}_{Y(\mathbb{R},w)}$. On the other
hand, we know from \eqref{eq:non-complete-1} that $\{a_m\}_{m\in\mathbb{N}}$
converges to $a_{-\alpha}\notin L^\infty(\mathbb{R})$. Hence
$\{a_m\}_{m\in\mathbb{N}}$ cannot converge to a limit in
$\mathcal{M}^0_{Y(\mathbb{R}, w)}$.
\end{proof}
\section{Fourier multipliers on reflection-invariant Banach function spaces}
\label{sec:reflection-invariant}
\subsection{Interpolation in Calder\'on products of Banach function spaces}
Let $X_0(\mathbb{R}^n)$ and $X_1(\mathbb{R}^n)$ be Banach function spaces and
$0<\theta<1$. The Calder\'on product $(X_0^{1-\theta}X_1^\theta)(\mathbb{R}^n)$
(see \cite[p.~123]{C64}) consists of all measurable functions $f$ such that 
a.e. pointwise inequality $|f|\le\lambda|f_0|^{1-\theta}|f_1|^\theta$ holds for 
some $\lambda>0$ and elements $f_j$ in $X_j(\mathbb{R}^n)$ with 
$\|f_j\|_{X_j(\mathbb{R}^n)}\le 1$ for $j=0,1$. The norm of $f$ in 
$(X_0^{1-\theta}X_1^\theta)(\mathbb{R}^n)$ 
is defined to be the infimum of all values $\lambda$ appearing in the
above inequality. 
We will need an interpolation theorem, which follows immediately
from \cite[Theorem~1]{Z67} and \cite[Chap.~1, Theorem~2.7]{BS88}.
\begin{theorem}\label{th:Zabreiko}
Let $X_0(\mathbb{R}^n)$ and $X_1(\mathbb{R}^n)$ be Banach function spaces. 
Let $A$ be a linear operator bounded on 
$X_0(\mathbb{R}^n)$ and $X_1(\mathbb{R}^n)$. Then $A$ is bounded on
$(X_0^{1-\theta}X_1^\theta)(\mathbb{R}^n)$ and
\[
\|A\|_{\mathcal{B}\left((X_0^{1-\theta}X_1^\theta)(\mathbb{R}^n)\right)}
\le 
\|A\|_{\mathcal{B}(X_0(\mathbb{R}^n))}^{1-\theta}
\|A\|_{\mathcal{B}(X_1(\mathbb{R}^n))}^\theta.
\]
\end{theorem}
The following result is contained in 
\cite[Theorem~5]{L69} in a slightly different  form. 
\begin{lemma}\label{le:Lozanovskii}
If $X(\mathbb{R}^n)$ is a Banach function space and
$X'(\mathbb{R}^n)$ is its associate space, then
\[
(X^{1/2}(X')^{1/2})(\mathbb{R}^n)=L^2(\mathbb{R}^n)
\]
with equality of the norms.
\end{lemma}
The above lemma is a consequence of the more general Lozanovski{\u\i}'s formula 
\cite[Theorem~2]{L69} (see also \cite[Theorem~7.2]{CN03}):
\[
(X_0^{1-\theta}X_1^\theta)'(\mathbb{R}^n)=
((X_0')^{1-\theta}(X_1')^\theta)(\mathbb{R}^n), \quad
0<\theta<1,
\]
which is valid with equality of the norms. We refer to Maligranda's book 
\cite[p.~185]{M89} for the proof of Lozanovski{\u\i}'s 
Lemma~\ref{le:Lozanovskii}.
\begin{corollary}\label{co:interpolation}
Let $X(\mathbb{R}^n)$ be a Banach function space and 
$X'(\mathbb{R}^n)$ be its associate space. If $A$ is a linear operator 
bounded on $X(\mathbb{R}^n)$ and on $X'(\mathbb{R}^n)$, then $A$ is bounded
on $L^2(\mathbb{R}^n)$ and 
\[
\|A\|_{\mathcal{B}(L^2(\mathbb{R}^n))}
\le 
\|A\|_{\mathcal{B}(X(\mathbb{R}^n))}^{1/2}
\|A\|_{\mathcal{B}(X'(\mathbb{R}^n))}^{1/2}.
\]
\end{corollary}
This result follows immediately from Theorem~\ref{th:Zabreiko} and 
Lemma~\ref{le:Lozanovskii}.
\subsection{Fourier multipliers on reflexive reflection-invariant Banach
function spaces are bounded}
We say that a Banach function space $X(\mathbb{R}^n)$ is reflection-invariant
if $\|f\|_{X(\mathbb{R}^n)}=\|\widetilde{f}\|_{X(\mathbb{R}^n)}$
for every $f\in X(\mathbb{R}^n)$, where $\widetilde{f}$ denotes the reflection
of a function $f$ defined by $\widetilde{f}(x)=f(-x)$ for $x\in\mathbb{R}^n$.
\begin{lemma}\label{le:reflection-invariant-associate}
A Banach function space $X(\mathbb{R}^n)$ is reflection-invariant if and only
if its associate space $X'(\mathbb{R}^n)$ is reflection-invariant.
\end{lemma}
This statement follows immediately from 
\cite[Chap.~1, Theorem~2.7 and Lemma~2.8]{BS88}.
\begin{lemma}\label{le:multiplier-reflection}
If a Banach function space $X(\mathbb{R}^n)$ is reflection-invariant, then 
\begin{align*}
\|a\|_{\mathcal{M}_{X(\mathbb{R}^n)}}
=
\|\widetilde{a}\|_{\mathcal{M}_{X(\mathbb{R}^n)}} 
&
\quad\mbox{for all}\quad
a\in\mathcal{M}_{X(\mathbb{R}^n)} , 
\\
\|a\|_{\mathcal{M}^0_{X(\mathbb{R}^n)}}
=
\|\widetilde{a}\|_{\mathcal{M}^0_{X(\mathbb{R}^n)}}
&
\quad\mbox{for all}\quad
a\in\mathcal{M}^0_{X(\mathbb{R}^n)}.
\end{align*}
\end{lemma}
\begin{proof}
If $u\in (S(\mathbb{R}^n)\cap X(\mathbb{R}^n))\setminus\{0\}$, then
$\widetilde{u}\in (S(\mathbb{R}^n)\cap X(\mathbb{R}^n))\setminus\{0\}$ and
\[
F^{-1}\widetilde{a}Fu
=
F^{-1}(a(Fu){\widetilde{\hspace{2mm}}}){\widetilde{\hspace{2mm}}}
=
(F^{-1}aF\widetilde{u}){\widetilde{\hspace{2mm}}}.
\]
Therefore, taking into account that $X(\mathbb{R}^n)$ is reflection-invariant,
we see that
\begin{align*}
\|\widetilde{a}\|_{\mathcal{M}_{X(\mathbb{R}^n)}}
&=
\sup\left\{
\frac{\|F^{-1}\widetilde{a}Fu\|_{X(\mathbb{R}^n)}}
{\|u\|_{X(\mathbb{R}^n)}}
\ :\
u\in (S(\mathbb{R}^n)\cap X(\mathbb{R}^n))\setminus\{0\}
\right\}
\\
&=
\sup\left\{
\frac{\|(F^{-1}aF\widetilde{u}){\widetilde{\hspace{2mm}}}\|_{X(\mathbb{R}^n)}}
{\|u\|_{X(\mathbb{R}^n)}}
\ :\
u\in (S(\mathbb{R}^n)\cap X(\mathbb{R}^n))\setminus\{0\}
\right\}
\\
&=
\sup\left\{
\frac{\|F^{-1}aF\widetilde{u}\|_{X(\mathbb{R}^n)}}
{\|\widetilde{u}\|_{X(\mathbb{R}^n)}}
\ :\
u\in (S(\mathbb{R}^n)\cap X(\mathbb{R}^n))\setminus\{0\}
\right\} \\
&=
\sup\left\{
\frac{\|F^{-1}aFu\|_{X(\mathbb{R}^n)}}
{\|u\|_{X(\mathbb{R}^n)}}
\ :\
u\in (S(\mathbb{R}^n)\cap X(\mathbb{R}^n))\setminus\{0\}
\right\}
\\
&=
\|a\|_{\mathcal{M}_{X(\mathbb{R}^n)}}.
\end{align*}
Replacing $S(\mathbb{R}^n)$ with $L^2(\mathbb{R}^n)$ one gets a proof for 
$\mathcal{M}^0_{X(\mathbb{R}^n)}$.
\end{proof}
\begin{lemma}\label{le:duality-0-embeddings}
Let $X(\mathbb{R}^n)$ be a Banach function space
and $X'(\mathbb{R}^n)$ be its associate space.
\begin{enumerate}
\item[(a)]
If $\widetilde{a}\in\mathcal{M}^0_{X'(\mathbb{R}^n)}$,
then $a\in\mathcal{M}^0_{X(\mathbb{R}^n)}$ and 
\[
\|a\|_{\mathcal{M}^0_{X(\mathbb{R}^n)}}\le 
\|\widetilde{a}\|_{\mathcal{M}^0_{X'(\mathbb{R}^n)}}.
\]

\item[(b)]
If $a\in\mathcal{M}^0_{X(\mathbb{R}^n)}$,
then $\widetilde{a}\in\mathcal{M}^0_{X'(\mathbb{R}^n)}$ and 
\[
\|a\|_{\mathcal{M}^0_{X(\mathbb{R}^n)}}\ge 
\|\widetilde{a}\|^0_{\mathcal{M}_{X'(\mathbb{R}^n)}}.
\]
\end{enumerate}
\end{lemma}
\begin{proof}
If $u,v\in L^2(\mathbb{R}^n)$ and $a\in L^\infty(\mathbb{R}^n)$, then
\begin{align}
\int_{\mathbb{R}^n}(F^{-1}aFu)(x)v(x)\,dx
&=
\langle F^{-1}aFu,v\rangle
=
\langle aFu,F^{-1}v\rangle
\nonumber\\
&=
\langle a,Fu\cdot F^{-1}v\rangle
=
\langle a,F^{-1}\widetilde{u}\cdot F\widetilde{v}\rangle
\nonumber\\
&=
\langle 
a,(F^{-1}u){\widetilde{\hspace{2mm}}}\cdot(Fv){\widetilde{\hspace{2mm}}}
\rangle
=
\langle\widetilde{a},F^{-1}u\cdot Fv\rangle
\nonumber\\
&=
\langle\widetilde{a}Fv,F^{-1}u\rangle
=
\langle F^{-1}\widetilde{a}Fv,u\rangle
\nonumber\\
&=
\int_{\mathbb{R}^n}(F^{-1}\widetilde{a}Fv)(x)u(x)\,dx.
\label{eq:duality-0-embeddings}
\end{align}


(a) Take $u\in (L^2(\mathbb{R}^n)\cap X(\mathbb{R}^n))\setminus\{0\}$.
Then, in view of equality \eqref{eq:duality-0-embeddings},
Lemma~\ref{le:NFP-X},  H\"older's inequality (see 
\cite[Chap.~1, Theorem~2.4]{BS88}), and the embedding
$S_0(\mathbb{R}^n)\subset L^2(\mathbb{R}^n)\cap X'(\mathbb{R}^n)$, we obtain
\begin{align*}
\|F^{-1}aFu\|_{X(\mathbb{R}^n)}
&=
\sup\left\{
\frac{\displaystyle\left|\int_{\mathbb{R}^n}(F^{-1}aFu)(x)s(x)\,dx\right|}
{\|s\|_{X'(\mathbb{R}^n)}}\ :\
s\in S_0(\mathbb{R}^n)\setminus\{0\}
\right\}
\\
&=
\sup\left\{
\frac{\displaystyle\left|\int_{\mathbb{R}^n}
(F^{-1}\widetilde{a}Fs)(x)u(x)\,dx\right|}
{\|s\|_{X'(\mathbb{R}^n)}}\ :\ 
s\in S_0(\mathbb{R}^n)\setminus\{0\}
\right\}
\\
&\le 
\sup\left\{
\frac{
\|F^{-1}\widetilde{a}Fs\|_{X'(\mathbb{R}^n)}\|u\|_{X(\mathbb{R}^n)}}
{\|s\|_{X'(\mathbb{R}^n)}}\ :\
s\in S_0(\mathbb{R}^n)\setminus\{0\}
\right\}
\\
&\le 
\|u\|_{X(\mathbb{R}^n)}
\sup\left\{
\frac{\|F^{-1}\widetilde{a}Fv\|_{X'(\mathbb{R}^n)}}
{\|v\|_{X'(\mathbb{R}^n)}}\ :\
v\in (L^2(\mathbb{R}^n)\cap X'(\mathbb{R}^n))\setminus\{0\}
\right\}
\\
&= 
\|\widetilde{a}\|_{\mathcal{M}^0_{X(\mathbb{R}^n)}}
\|u\|_{X(\mathbb{R}^n)},
\end{align*}
whence
\[
\|a\|_{\mathcal{M}^0_{X(\mathbb{R}^n)}}
=
\sup\left\{
\frac{\|F^{-1}aFu\|_{X(\mathbb{R}^n)}}{\|u\|_{X(\mathbb{R}^n)}}
\ : \
u\in (L^2(\mathbb{R}^n)\cap X(\mathbb{R}^n))\setminus\{0\}
\right\}
\le
\|\widetilde{a}\|_{\mathcal{M}^0_{X(\mathbb{R}^n)}},
\]
which completes the proof of part (a).

(b) By the Lorentz-Luxemburg theorem (see \cite[Chap.~1, Theorem~2.7]{BS88}),
$X(\mathbb{R}^n)=X''(\mathbb{R}^n)$ with the equality of the norms. Then 
$(\widetilde{a}){\widetilde{\hspace{2mm}}}=a\in
\mathcal{M}^0_{X(\mathbb{R}^n)}=\mathcal{M}^0_{X''(\mathbb{R}^n)}$.
Hence, part~(b) follows from part~(a).
\end{proof}
\begin{lemma}\label{le:duality-embeddings}
Let $X(\mathbb{R}^n)$ be a Banach function space 
and $X'(\mathbb{R}^n)$ be its associate space.
\begin{enumerate}
\item[(a)]
Suppose the space $X(\mathbb{R}^n)$ satisfies the norm fundamental property.
If $\ \widetilde{a}\in\mathcal{M}_{X'(\mathbb{R}^n)}$,
then $a\in\mathcal{M}_{X(\mathbb{R}^n)}$ and 
\[
\|a\|_{\mathcal{M}_{X(\mathbb{R}^n)}}\le 
\|\widetilde{a}\|_{\mathcal{M}_{X'(\mathbb{R}^n)}}.
\]

\item[(b)]
Suppose the space $X'(\mathbb{R}^n)$ satisfies the norm fundamental property.
If $a\in\mathcal{M}_{X(\mathbb{R}^n)}$, then 
$\widetilde{a}\in\mathcal{M}_{X'(\mathbb{R}^n)}$ and 
\[
\|a\|_{\mathcal{M}_{X(\mathbb{R}^n)}}\ge 
\|\widetilde{a}\|_{\mathcal{M}_{X'(\mathbb{R}^n)}}.
\]
\end{enumerate}
\end{lemma}
\begin{proof}
The proof is similar to that of Lemma~\ref{le:duality-0-embeddings}.
Interpreting  $\langle\cdot, \cdot\rangle$ in \eqref{eq:duality-0-embeddings} 
as the $(S'(\mathbb{R}^n), S(\mathbb{R}^n))$ rather than
$(L^2(\mathbb{R}^n), L^2(\mathbb{R}^n))$ or 
$(L^\infty(\mathbb{R}^n), L^1(\mathbb{R}^n))$ duality, one gets for any
$u,v\in S(\mathbb{R}^n)$ and $a\in S'(\mathbb{R}^n)$,
\begin{equation}\label{eq:duality-embeddings}
\int_{\mathbb{R}^n}(F^{-1}aFu)(x)v(x)\,dx
=
\int_{\mathbb{R}^n}(F^{-1}\widetilde{a}Fv)(x)u(x)\,dx.
\end{equation}

(a) Take $u\in (S(\mathbb{R}^n)\cap X(\mathbb{R}^n))\setminus\{0\}$.
Then, in view of equality \eqref{eq:duality-embeddings},
Definition~\ref{def:NFP}, and H\"older's inequality (see 
\cite[Chap.~1, Theorem~2.4]{BS88}), we obtain
\begin{align*}
\|F^{-1}aFu\|_{X(\mathbb{R}^n)}
&=
\sup\left\{
\frac{\displaystyle\left|\int_{\mathbb{R}^n}(F^{-1}aFu)(x)\psi(x)\,dx\right|}
{\|\psi\|_{X'(\mathbb{R}^n)}}\ :\
\psi\in C_0^\infty(\mathbb{R}^n)\setminus\{0\}
\right\}
\\
&=
\sup\left\{
\frac{\displaystyle\left|\int_{\mathbb{R}^n}
(F^{-1}\widetilde{a}F\psi)(x)u(x)\,dx\right|}
{\|\psi\|_{X'(\mathbb{R}^n)}}\ :\ 
\psi\in C_0^\infty(\mathbb{R}^n)\setminus\{0\}
\right\}
\\
&\le 
\sup\left\{
\frac{
\|F^{-1}\widetilde{a}F\psi\|_{X'(\mathbb{R}^n)}\|u\|_{X(\mathbb{R}^n)}}
{\|\psi\|_{X'(\mathbb{R}^n)}}\ :\
\psi\in C_0^\infty(\mathbb{R}^n)\setminus\{0\}
\right\}
\\
&\le 
\|u\|_{X(\mathbb{R}^n)}
\sup\left\{
\frac{\|F^{-1}\widetilde{a}Fv\|_{X'(\mathbb{R}^n)}}
{\|v\|_{X'(\mathbb{R}^n)}}\ :\
v\in (S(\mathbb{R}^n)\cap X'(\mathbb{R}^n))\setminus\{0\}
\right\}
\\
&= 
\|\widetilde{a}\|_{\mathcal{M}_{X(\mathbb{R}^n)}}
\|u\|_{X(\mathbb{R}^n)},
\end{align*}
whence
\[
\|a\|_{\mathcal{M}_{X(\mathbb{R}^n)}}
=
\sup\left\{
\frac{\|F^{-1}aFu\|_{X(\mathbb{R}^n)}}{\|u\|_{X(\mathbb{R}^n)}}
\ : \
u\in (S(\mathbb{R}^n)\cap X(\mathbb{R}^n))\setminus\{0\}
\right\}
\le
\|\widetilde{a}\|_{\mathcal{M}_{X(\mathbb{R}^n)}},
\]
which completes the proof of part (a).

(b) By the Lorentz-Luxemburg theorem (see \cite[Chap.~1, Theorem~2.7]{BS88}),
$X(\mathbb{R}^n)=X''(\mathbb{R}^n)$ with the equality of the norms. Then 
$(\widetilde{a}){\widetilde{\hspace{2mm}}}=a\in
\mathcal{M}_{X(\mathbb{R}^n)}=\mathcal{M}_{X''(\mathbb{R}^n)}$.
Hence, part~(b) follows from part~(a).
\end{proof}
\begin{theorem}\label{th:duality-equality}
Let $X(\mathbb{R}^n)$ be a reflection-invariant Banach function space
and $X'(\mathbb{R}^n)$ be its associate space.
\begin{enumerate}
\item[(a)]
We have $\mathcal{M}^0_{X(\mathbb{R}^n)}=\mathcal{M}^0_{X'(\mathbb{R}^n)}$ and 
$\|a\|_{\mathcal{M}^0_{X(\mathbb{R}^n)}}=
\|a\|_{\mathcal{M}^0_{X'(\mathbb{R}^n)}}$ 
for all $a\in\mathcal{M}^0_{X(\mathbb{R}^n)}$.
\item[(b)]

If both $X(\mathbb{R}^n)$ and $X'(\mathbb{R}^n)$ satisfy
the norm fundamental property, then we have
$\mathcal{M}_{X(\mathbb{R}^n)}=\mathcal{M}_{X'(\mathbb{R}^n)}$ and 
$\|a\|_{\mathcal{M}_{X(\mathbb{R}^n)}}=\|a\|_{\mathcal{M}_{X'(\mathbb{R}^n)}}$ 
for all $a\in\mathcal{M}_{X(\mathbb{R}^n)}$.
\end{enumerate}
\end{theorem}
\begin{proof}
We prove part (b). The proof of part (a) is almost exactly the same. By 
Lemma~\ref{le:reflection-invariant-associate}, both $X(\mathbb{R}^n)$
and $X'(\mathbb{R}^n)$ are reflection-invariant Banach function spaces.
If $a\in\mathcal{M}_{X'(\mathbb{R}^n)}$, then 
$\widetilde{a}\in\mathcal{M}_{X'(\mathbb{R}^n)}$ and
\begin{equation}\label{eq:duality-equality-1}
\|a\|_{\mathcal{M}_{X'(\mathbb{R}^n)}}
=
\|\widetilde{a}\|_{\mathcal{M}_{X'(\mathbb{R}^n)}}
\ge 
\|a\|_{\mathcal{M}_{X(\mathbb{R}^n)}}
\end{equation}
in view of Lemmas~\ref{le:multiplier-reflection} 
and~\ref{le:duality-embeddings}(a). On the other hand, if
$a\in\mathcal{M}_{X(\mathbb{R}^n)}$, then 
$\widetilde{a}\in\mathcal{M}_{X'(\mathbb{R}^n)}$ and
\begin{equation}\label{eq:duality-equality-2}
\|a\|_{\mathcal{M}_{X(\mathbb{R}^n)}}
\geq
\|\widetilde{a}\|_{\mathcal{M}_{X'(\mathbb{R}^n)}}
=
\|(\widetilde{a}){\widetilde{\hspace{2mm}}}
\|_{\mathcal{M}_{X'(\mathbb{R}^n)}}
=
\|a\|_{\mathcal{M}_{X'(\mathbb{R}^n)}}
\end{equation}
in view of Lemmas~\ref{le:duality-embeddings}(b) 
and~\ref{le:multiplier-reflection}. Combining inequalities 
\eqref{eq:duality-equality-1}--\eqref{eq:duality-equality-2},
we arrive at the desired result.
\end{proof}
Now we will show that one cannot drop the norm fundamental property in 
Theorem~\ref{th:duality-equality}(b).
\begin{theorem} \label{notL2a}
Suppose $G\subset [0,1]^n$ is a compact set of positive measure with empty 
interior, $b=\{b_m\}_{m\in\mathbb{N}}\subset(0,1)$ is a sequence satisfying
\eqref{eq:two-classes-of-multipliers-are-different-0}, and $w_{G,b}$ is
the weight given by 
\eqref{eq:two-classes-of-multipliers-are-different-1}--%
\eqref{eq:two-classes-of-multipliers-are-different-3}.
Then the reflection-invariant space $L^1(\mathbb{R}^n,w_{G,b}^{-1})$ 
does not satisfy the norm-fundamental property and there 
exists $a\in S(\mathbb{R}^n)$ such that 
$a\in \mathcal{M}_{L^\infty(\mathbb{R}^n,w_{G,b})}\setminus
\mathcal{M}_{L^1(\mathbb{R}^n,w_{G,b}^{-1})}$.
\end{theorem}
\begin{proof} 
It follows from \eqref{eq:two-classes-of-multipliers-are-different-3} that
$w_{G,b}(y)=w_{G,b}(-y)$ for all $y\in\mathbb{R}^n$. Therefore,
$L^1(\mathbb{R}^n, w_{G,b}^{-1})$ is a reflection-invariant Banach function 
space. One can prove, similarly to 
Corollary~\ref{co:L1-with-exotic-weight-fails-NFP}, that
$L^1(\mathbb{R}^n,w_{G,b}^{-1})$ does not satisfy the norm-fundamental 
property.

Let $\rho \in C_0^\infty(\mathbb{R}^n)$ be an even function such that 
$\rho \ge 0$ and $\rho(y) = 1$ when $|y| \le \sqrt{n + 3}$, and let
$a := F\rho$. Then $a \in S(\mathbb{R}^n)$ and $\widetilde{a} = a$. 

Take any $u \in S(\mathbb{R}^n)$ with 
$\|u\|_{L^\infty(\mathbb{R}^n,w_{G,b})} \le 1$. 
The same argument as in the proof of Lemma~\ref{le:exotic-weight} shows 
that $|u(y)| \le 1$ for all $y \in \mathbb{R}^n$. Then
\[
\left|(F^{-1}aF u)(x)\right| 
= 
|(\rho \ast u)(x)| \le \int_{\mathbb{R}^n} \rho(y)\, dy 
= 
\|\rho\|_{L^1(\mathbb{R}^n)} < \infty 
\quad\mbox{for all}\quad
x  \in \mathbb{R}^n 
\]
and 
$\|F^{-1}aF u\|_{L^\infty(\mathbb{R}^n, w_{G,b})} 
\le \|\rho\|_{L^1(\mathbb{R}^n)}$. Hence 
$a \in \mathcal{M}_{L^\infty(\mathbb{R}^n,w_{G,b})}$ and
\[
\|a\|_{\mathcal{M}_{L^\infty(\mathbb{R}^n,w_{G,b})}} 
\le \|\rho\|_{L^1(\mathbb{R}^n)}.
\]

On the other hand, consider $v_m \in C_0^\infty(\mathbb{R}^n)$ such that 
$v_m \ge 0$,  $v_m(y) = 1$ for $|y - y^{(m)}| \le 1/4$ and  
$v_m(y) = 0$ for $|y - y^{(m)}| \ge 1/2$, where 
$y^{(m)} := \left(2m - \frac12, \frac12, \dots, \frac12\right)$ is the 
centre of the cube $Q_m := (2m -1, 0, \dots, 0) + [0, 1]^n$, $m \in \mathbb{N}$.
Then $\operatorname{supp} v_m \subset Q_m$ and 
$w_{G,b}(x)=1$ for $x\in Q_m$, whence
\[
\|v_m\|_{L^1(\mathbb{R}^n, w_{G,b}^{-1}}) = \int_{Q_m} v_m(y)\, dy \le 1 .
\] 
Since the distance from any point of $G_m$ to any point of $Q_m$ is less than 
or equal to $\sqrt{2^2 + 1^2 + \cdots + 1^2} = \sqrt{n + 3}$, 
it follows from the definition of $\rho$ that for all $x \in G_m$,
\begin{align*}
|F^{-1}aF v_m(x)| 
&= 
|(\rho \ast v_m)(x)| = \int_{\mathbb{R}^n} v_m(y) \rho(x - y)\, dy 
\\
&= 
\int_{Q_m} v_m(y) \rho(x - y)\, dy 
=
\int_{Q_m} v_m(y)\, dy \ge 4^{-n} \Omega_n , 
\end{align*}
where $\Omega_n$ is the volume of the unit ball in $\mathbb{R}^n$. Hence
\[
\|F^{-1}aF v_m\|_{L^1(\mathbb{R}^n, w_{G,b}^{-1})} 
\ge 
b_m^{-1} \int_{G_m} |F^{-1}aF v_m(x)| \, dx 
\ge  
\frac{4^{-n} \Omega_n |G|}{b_m} \to \infty 
\quad \mbox{as} \quad m \to \infty ,
\]
while $\|v_m\|_{L^1(\mathbb{R}^n, w_{G,b}^{-1})} \le 1$. 
Hence $a \not\in \mathcal{M}_{L^1(\mathbb{R}^n, w_{G,b}^{-1})}$.
\end{proof}
For Banach spaces $E_0,E_1$ and a number $\theta\in(0,1)$, let 
$[E_0,E_1]_\theta$ denote the space obtained by the (lower) complex method of 
interpolation (see, e.g., \cite{C64} or \cite[Chap.~IV, \S 1.4]{KPS82}).
\begin{theorem}\label{th:XXprime-to-L-infinity}
Let $X(\mathbb{R}^n)$ be a Banach function space and $X'(\mathbb{R}^n)$ be its 
associate space. If 
$a \in \mathcal{M}^0_{X(\mathbb{R}^n)}\cap \mathcal{M}^0_{X'(\mathbb{R}^n)}$, 
then
\begin{equation}\label{eq:XXprime-to-L-infinity-1}
\|a\|_{L^\infty(\mathbb{R}^n)} 
\le 
\|a\|_{\mathcal{M}^0_{X(\mathbb{R}^n)}}^{1/2} 
\|a\|_{\mathcal{M}^0_{X'(\mathbb{R}^n)}}^{1/2}.
\end{equation}
\end{theorem}
\begin{proof} 
Since $L^2(\mathbb{R}^n) = (X^{1/2}(X')^{1/2})(\mathbb{R}^n)$ (see 
Lemma~\ref{le:Lozanovskii}) and $L^2(\mathbb{R}^n)$ has absolutely continuous 
norm, by \cite[Chap.~IV, Theorem~1.14]{KPS82}, we have
$L^2(\mathbb{R}^n) = [X(\mathbb{R}^n), X'(\mathbb{R}^n)]_{1/2}$
with equality of the norms. 

Let $\overline{X}(\mathbb{R}^n)$ and $\overline{X'}(\mathbb{R}^n)$ denote the 
closures of $X(\mathbb{R}^n)\cap X'(\mathbb{R}^n)$ in the spaces 
$X(\mathbb{R}^n)$ and $X'(\mathbb{R}^n)$ respectively. Then 
$L^2(\mathbb{R}^n) = [\overline{X}(\mathbb{R}^n), 
\overline{X'}(\mathbb{R}^n)]_{1/2}$ with equality of the norms
(see the discussion after the proof of  \cite[Chap.~IV, Theorem~1.3]{KPS82}). 

By \cite[Chap.~IV, Theorem~1.3]{KPS82},  
$X(\mathbb{R}^n)\cap X'(\mathbb{R}^n) \subset 
[X(\mathbb{R}^n), X'(\mathbb{R}^n)]_{1/2} = L^2(\mathbb{R}^n)$. 
Since
$a \in \mathcal{M}^0_{X(\mathbb{R}^n)}\cap \mathcal{M}^0_{X'(\mathbb{R}^n)}$, 
the operator $W_a:f\mapsto F^{-1}aFf$, defined initially on 
$X(\mathbb{R}^n)\cap X'(\mathbb{R}^n)$,
can be extended to bounded linear operators
\[
W_a : \overline{X}(\mathbb{R}^n) \to X(\mathbb{R}^n),
\quad
W_a : \overline{X'}(\mathbb{R}^n) \to X'(\mathbb{R}^n). 
\]
Then, by the interpolation theorem for the complex method of interpolation
(see, e.g., \cite[Chap.~IV, Theorem~1.2]{KPS82}),
\begin{align}
\|W_a\|_{\mathcal{B}(L^2(\mathbb{R}^n))} 
&=  
\|W_a\|_{\mathcal{B}
\left([\overline{X}(\mathbb{R}^n), \overline{X'}(\mathbb{R}^n)]_{1/2}, 
[X(\mathbb{R}^n), X'(\mathbb{R}^n)]_{1/2}\right)} 
\notag\\
&\le 
\|W_a\|_{\mathcal{B}
\left(\overline{X}(\mathbb{R}^n), X(\mathbb{R}^n)\right)}^{1/2} 
\|W_a\|_{\mathcal{B}
\left(\overline{X'}(\mathbb{R}^n), X'(\mathbb{R}^n)\right)}^{1/2}.
\label{eq:XXprime-to-L-infinity-2}
\end{align}
Since $X(\mathbb{R}^n)\cap X'(\mathbb{R}^n)\subset 
L^2(\mathbb{R}^n)\cap \overline{X}(\mathbb{R}^n)$
and 
$X(\mathbb{R}^n)\cap X'(\mathbb{R}^n)\subset 
L^2(\mathbb{R}^n)\cap \overline{X'}(\mathbb{R}^n)$,
we conclude that 
$L^2(\mathbb{R}^n)\cap \overline{X}(\mathbb{R}^n)$
is dense in $\overline{X}(\mathbb{R}^n)$ and 
$L^2(\mathbb{R}^n)\cap \overline{X'}(\mathbb{R}^n)$
is dense in $\overline{X'}(\mathbb{R}^n)$. Hence
\begin{align}
\|W_a\|_{\mathcal{B}
\left(\overline{X}(\mathbb{R}^n), X(\mathbb{R}^n)\right)}
&=
\sup\left\{
\frac{\|F^{-1}aFu\|_{X(\mathbb{R}^n)}}{\|u\|_{X(\mathbb{R}^n)}}
\ :\ u\in (L^2(\mathbb{R}^n) \cap \overline{X}(\mathbb{R}^n))\setminus\{0\}
\right\}
\notag\\
&\le 
\|a\|_{\mathcal{M}^0_{X(\mathbb{R}^n)}}
\label{eq:XXprime-to-L-infinity-3}
\end{align}
and
\begin{align}
\|W_a\|_{\mathcal{B}
\left(\overline{X'}(\mathbb{R}^n), X'(\mathbb{R}^n)\right)}
&=
\sup\left\{
\frac{\|F^{-1}aFu\|_{X'(\mathbb{R}^n)}}{\|u\|_{X'(\mathbb{R}^n)}}
\ :\ u\in (L^2(\mathbb{R}^n) \cap \overline{X'}(\mathbb{R}^n))\setminus\{0\}
\right\}
\notag\\
&\le 
\|a\|_{\mathcal{M}^0_{X'(\mathbb{R}^n)}}.
\label{eq:XXprime-to-L-infinity-4}
\end{align}
It is well known (see, e.g., \cite[Theorem~2.5.10]{G14-classical}) that
$\mathcal{M}_{L^2(\mathbb{R}^n)}=L^\infty(\mathbb{R}^n)$ and
\begin{equation}\label{eq:XXprime-to-L-infinity-5}
\|W_a\|_{\mathcal{B}(L^2(\mathbb{R}^n))}
=
\|a\|_{\mathcal{M}_{L^2(\mathbb{R}^n)}}
=
\|a\|_{L^\infty(\mathbb{R}^n)}.
\end{equation}
Combining 
\eqref{eq:XXprime-to-L-infinity-2}--\eqref{eq:XXprime-to-L-infinity-5}, 
we arrive at \eqref{eq:XXprime-to-L-infinity-1}.
\end{proof}
We are now we are in a position to prove the main result of this section.
\begin{theorem}\label{th:bounded-multipliers-reflection-invariant}
Let $X(\mathbb{R}^n)$ be a  reflection-invariant Banach function
space. 
\begin{enumerate}
\item[(a)]
If $a\in\mathcal{M}^0_{X(\mathbb{R}^n)}$, then 
\begin{equation}\label{eq:bounded-multipliers-reflection-invariant-1-0}
\|a\|_{L^\infty(\mathbb{R}^n)} 
\le 
\|a\|_{\mathcal{M}^0_{X(\mathbb{R}^n)}}.
\end{equation}
\item[(b)]
If $X(\mathbb{R}^n)$ is reflexive and 
$a\in\mathcal{M}_{X(\mathbb{R}^n)}\subset S'(\mathbb{R}^n)$, then
$a\in L^\infty(\mathbb{R}^n)$ and
\begin{equation}\label{eq:bounded-multipliers-reflection-invariant-1}
\|a\|_{L^\infty(\mathbb{R}^n)}\le\|a\|_{\mathcal{M}_{X(\mathbb{R}^n)}}.
\end{equation}
\end{enumerate}
The constant $1$ in the right-hand sides of 
\eqref{eq:bounded-multipliers-reflection-invariant-1-0} and
\eqref{eq:bounded-multipliers-reflection-invariant-1} is best 
possible.
\end{theorem}
\begin{proof}
Part (a) follows from Theorems \ref{th:duality-equality}(a) 
and~\ref{th:XXprime-to-L-infinity}.

(b) By \cite[Chap.~1, Corollary~4.4]{BS88}, a Banach function space 
$X(\mathbb{R}^n)$ is reflexive if and only if both $X(\mathbb{R}^n)$ and
$X'(\mathbb{R}^n)$ have absolutely continuous norm. Then both $X(\mathbb{R}^n)$ 
and $X'(\mathbb{R}^n)$ satisfy the norm fundamental property (see 
Corollary~\ref{co:NFP-associate-separable}).
If $a\in\mathcal{M}_{X(\mathbb{R}^n)}$, then 
$a\in\mathcal{M}_{X'(\mathbb{R}^n)}$ and 
\begin{equation}\label{eq:bounded-multipliers-reflection-invariant-2}
\|a\|_{\mathcal{M}_{X(\mathbb{R}^n)}}=\|a\|_{\mathcal{M}_{X'(\mathbb{R}^n)}}
\end{equation}
in view of Theorem \ref{th:duality-equality}(b). It follows from 
\cite[Lemma~2.12(b)]{KS14}, that the set $C_0^\infty(\mathbb{R}^n)$
is dense in the spaces $X(\mathbb{R}^n)$ and $X'(\mathbb{R}^n)$.
Hence the convolution operator $W_a:u\mapsto F^{-1}aFu$ defined initially
on $C_0^\infty(\mathbb{R}^n)$ extends to a bounded linear operator on both
$X(\mathbb{R}^n)$ and $X'(\mathbb{R}^n)$. Moreover,
\begin{equation}\label{eq:bounded-multipliers-reflection-invariant-3}
\|W_a\|_{\mathcal{B}(X(\mathbb{R}^n))}
=
\|a\|_{\mathcal{M}_{X(\mathbb{R}^n)}},
\quad
\|W_a\|_{\mathcal{B}(X'(\mathbb{R}^n))}
=
\|a\|_{\mathcal{M}_{X'(\mathbb{R}^n)}}.
\end{equation}
By Corollary~\ref{co:interpolation}, $W_a$ is bounded on $L^2(\mathbb{R}^n)$
and
\begin{equation}\label{eq:bounded-multipliers-reflection-invariant-4}
\|W_a\|_{\mathcal{B}(L^2(\mathbb{R}^n))}
\le
\|W_a\|_{\mathcal{B}(X(\mathbb{R}^n))}^{1/2}
\|W_a\|_{\mathcal{B}(X'(\mathbb{R}^n))}^{1/2}.
\end{equation}
It is well known (see, e.g., \cite[Theorem~2.5.10]{G14-classical}) that
$\mathcal{M}_{L^2(\mathbb{R}^n)}=L^\infty(\mathbb{R}^n)$ and
\begin{equation}\label{eq:bounded-multipliers-reflection-invariant-5}
\|W_a\|_{\mathcal{B}(L^2(\mathbb{R}^n))}
=
\|a\|_{\mathcal{M}_{L^2(\mathbb{R}^n)}}
=
\|a\|_{L^\infty(\mathbb{R}^n)}.
\end{equation}
Combining 
\eqref{eq:bounded-multipliers-reflection-invariant-2}--%
\eqref{eq:bounded-multipliers-reflection-invariant-5},
we see that
\[
\|a\|_{L^\infty(\mathbb{R}^n)}
\le 
\|a\|_{\mathcal{M}_{X(\mathbb{R}^n)}}^{1/2}
\|a\|_{\mathcal{M}_{X'(\mathbb{R}^n)}}^{1/2}
=
\|a\|_{\mathcal{M}_{X(\mathbb{R}^n)}}^{1/2}
\|a\|_{\mathcal{M}_{X(\mathbb{R}^n)}}^{1/2}
=
\|a\|_{\mathcal{M}_{X(\mathbb{R}^n)}},
\]
which completes the proof of 
\eqref{eq:bounded-multipliers-reflection-invariant-1}.

Suppose now that there exists a constant $D_X > 0$ such that
$\|a\|_{L^\infty(\mathbb{R}^n)} \le D_X \|a\|_{\mathcal{M}_{X(\mathbb{R}^n)}}$
for all $a \in\mathcal{M}_{X(\mathbb{R}^n)}$.
Then taking $a \equiv 1$, one gets $D_X \ge 1$. So, the constant 
$D_X = 1$ in \eqref{eq:bounded-multipliers-reflection-invariant-1}
is best possible. The same can be proved similarly for 
\eqref{eq:bounded-multipliers-reflection-invariant-1-0}.
\end{proof}
Unfortunately, we have not been able to answer the following question.
\begin{question} 
Can one drop the reflexivity requirement in 
Theorem~\ref{th:bounded-multipliers-reflection-invariant}(b)?
\end{question}
\section{Concluding remarks}\label{subsec:Lofstrom}
Let $Y(\mathbb{R})$ be a translation-invariant Banach function space. We have 
seen that for subexponentially growing weights $w$ like \eqref{2sided}, 
$Y(\mathbb{R}, w)$ has the weak doubling property in view of 
Lemma~\ref{le:subexponential-growth-2}, and hence 
$\mathcal{M}_{Y(\mathbb{R}, w)} \subseteq L^\infty(\mathbb{R})$
according to Theorem~\ref{th:main}. This inclusion holds also for 
reflexive transaltion-invariant Banach function spaces $Y(\mathbb{R})$ and
symmetric weights $w = \widetilde{w}$ that may grow arbitrarily fast 
(see Theorem \ref{th:bounded-multipliers-reflection-invariant}).
On the other hand, $\mathcal{M}_{Y(\mathbb{R}, w)}$ may contain unbounded 
functions if $w$ grows at least exponentially as $x \to +\infty$ 
and decays to $0$ exponentially as $x \to -\infty$ 
(see  Corollaries~\ref{co:w1w2} and~\ref{co:examples-of-unbounded-multipliers}). 
It is natural to ask whether there are any unbounded Fourier multipliers in the 
case of  weights like
\begin{equation}\label{eq:superexponential-example}
w(x) 
= 
\left\{\begin{array}{cl}
\exp{\left( |x|^{\alpha_1}\right)} ,   &  x < 0 ,  
\\
\exp{\left(x^{\alpha_2}\right)} ,   &  x \ge 0 ,
\end{array}\right. \quad   \alpha_1, \alpha_2 > 1 .
\end{equation}
It turns out that there are no non-trivial Fourier multipliers in the case of 
weights like \eqref{eq:superexponential-example} and, more generally, of 
weights on $\mathbb{R}^n$ that grow superexponentially in all directions: 
$\mathcal{M}_{Y(\mathbb{R}^n,w)}=\mathbb{C}$. This fact was observed first 
by L\"ofstr\"om \cite{L83} in the case $Y(\mathbb{R}^n)=L^p(\mathbb{R}^n)$ 
with $1\le p\le\infty$. For the convenience of the reader, we present here 
a slightly modified argument from \cite{L83} in the case of arbitrary
translation-invariant Banach function spaces.
\begin{theorem}[{(Cf. \cite[p.~93]{L83})}]
\label{th:Lofstrom2}
Suppose $X(\mathbb{R}^n)$ is a Banach function space such that for every 
$x_0 \in \mathbb{R}^n\setminus\{0\}$ there exist $\varepsilon > 0$ and a 
sequence $\{x_k\}_{k\in\mathbb{N}} \subset \mathbb{R}^n$ satisfying the 
condition
\begin{equation}\label{eq:Lofstrom2-1}
\frac
{\|\chi_{B(x_k, \rho)}\|_{X(\mathbb{R}^n)}}
{\|\chi_{B(x_k - x_0, \varepsilon)}\|_{X(\mathbb{R}^n)}} 
\to \infty \ \ \mbox{\rm as } \ k \to \infty
\end{equation}
for every $\rho \in (0, \varepsilon]$. If $\kappa$ is a distribution such that
\begin{equation}\label{eq:Lofstrom2-2}
\|\kappa\ast f\|_{X(\mathbb{R}^n)} \le C \|f\|_{X(\mathbb{R}^n)} 
\quad \mbox{for all}\quad 
f \in C_0^\infty(\mathbb{R}^n),
\end{equation}
with some constant $C > 0$, then $\kappa = c \delta$ with some constant 
$c \in \mathbb{C}$, where $\delta$ is the Dirac measure.
\end{theorem}
\begin{proof}
Take any $x_0 \in \mathbb{R}^n\setminus\{0\}$. Suppose 
$x_0 \in \operatorname{supp} \kappa$. Then there exists
$f \in C_0^\infty(\mathbb{R}^n)$ such that 
$\operatorname{supp} \widetilde{f} \subseteq B(x_0, \varepsilon)$ and
\[
\kappa\ast f(0) = \big\langle \kappa, \widetilde{f}\big\rangle = 1 .
\]
Since $\kappa\ast f$ is continuous (see, e.g., \cite[Theorem 6.30(b)]{R91}), 
there exists $\rho  \in (0, \varepsilon]$ such that $|\kappa\ast f(x)| > 1/2$ 
for all $x \in B(0, \rho)$. Then
\[
\left|
\big(\kappa\ast \left(\tau_{x_k}f\right)\big)(x + x_k)
\right| 
= 
\left|
\big(\tau_{-x_k}\left(\kappa\ast \left(\tau_{x_k}f\right)\right)\big)(x)
\right| 
=
|\kappa\ast f(x)| > \frac12
\]
for all $x \in B(0, \rho)$. Hence it follows from \eqref{eq:Lofstrom2-2} that
\begin{align*}
\frac12\, \|\chi_{B(x_k, \rho)}\|_{X(\mathbb{R}^n)} 
&\le 
\left\|\chi_{B(x_k, \rho)}
\big(\kappa\ast \left(\tau_{x_k}f\right)\big)
\right\|_{X(\mathbb{R}^n)}
\le 
\left\|\kappa\ast \left(\tau_{x_k}f\right)\right\|_{X(\mathbb{R}^n)} 
\\
&\le 
C \|\tau_{x_k}f\|_{X(\mathbb{R}^n)} \le C\|f\|_{L^\infty(\mathbb{R}^n)} 
\|\chi_{B(x_k - x_0, \varepsilon)}\|_{X(\mathbb{R}^n)} ,
\end{align*}
since 
$\operatorname{supp} \left(\tau_{x_k}f\right) 
= 
x_k + \operatorname{supp} f \subseteq x_k + B(-x_0, \varepsilon) 
= 
B(x_k - x_0, \varepsilon)$. 
So,
\[
\frac{\|\chi_{B(x_k, \rho)}\|_{X(\mathbb{R}^n)}}
{\|\chi_{B(x_k - x_0, \varepsilon)}\|_{X(\mathbb{R}^n)}} 
\le 
2 C\|f\|_{L^\infty(\mathbb{R}^n)} 
\quad\mbox{for all}\quad
k \in \mathbb{N} ,
\]
which contradicts \eqref{eq:Lofstrom2-1}. This means that 
$x_0 \in \mathbb{R}^n\setminus\{0\}$ cannot belong to 
the support of $\kappa$, i.e.
$\operatorname{supp} \kappa = \{0\}$. Hence $\kappa$ is a linear combination 
of $\delta$ and its partial derivatives 
(see, e.g., \cite[Theorems 6.24(d) and 6.25]{R91}).
It is easy to see that then \eqref{eq:Lofstrom2-2} implies the equality 
$\kappa = c \delta$ with some constant $c \in \mathbb{C}$.
\end{proof}
\begin{theorem}[{(Cf. \cite[p.~91-93]{L83})}]
\label{th:Lofstrom2-for-weighted-TI-spaces}
Let $Y(\mathbb{R}^n)$ be a translation-invariant Banach function space
and $w:\mathbb{R}^n\to[0,\infty]$ be a weight such that 
$w \in Y_{\mathrm{loc}}(\mathbb{R}^n)$ and 
$1/w  \in Y'_{\mathrm{loc}}(\mathbb{R}^n)$. Suppose for every 
$x_0 \in \mathbb{R}^n\setminus\{0\}$ there exist $\varepsilon > 0$ and a 
sequence $\{x_k\}_{k\in\mathbb{N}}\subset\mathbb{R}^n$ satisfying the 
condition
\begin{equation}\label{eq:Lofstrom2-for-weighted-TI-spaces}
\inf_{|x| \le \varepsilon, \ |y| \le \varepsilon}
\ 
\frac{w(x_k + x)}{w(x_k - x_0 + y)} \to \infty 
\quad \mbox{as} \quad k \to \infty .
\end{equation}
Then $\mathcal{M}_{Y(\mathbb{R}^n, w)}=\mathbb{C}$.
\end{theorem}
\begin{proof}
Since $Y(\mathbb{R}^n)$ is translation-invariant, it follows from
Lemma~\ref{le:TI-balls}(a) that there exist constants $C_1,C_2>0$ such that
for all $k\in\mathbb{N}$ and all $\rho\in(0,\varepsilon]$ one has
\[
\frac{\|\chi_{B(x_k, \rho)}\|_{Y(\mathbb{R}^n)}}
{\|\chi_{B(x_k - x_0, \varepsilon)}\|_{Y(\mathbb{R}^n)}}
\ge 
\frac{C_1\min\{1,\rho^n\}}{C_2\max\{1,\varepsilon^n\}}=:C(\rho,\varepsilon).
\]
Hence
\begin{align*}
\frac{\|\chi_{B(x_k, \rho)}\|_{X(\mathbb{R}^n)}}
{\|\chi_{B(x_k - x_0, \varepsilon)}\|_{X(\mathbb{R}^n)}} 
&= 
\frac{\|w\chi_{B(x_k, \rho)}\|_{Y(\mathbb{R}^n)}}
{\|w\chi_{B(x_k - x_0, \varepsilon)}\|_{Y(\mathbb{R}^n)}} 
\ge 
\frac{\displaystyle\inf_{|x| \le \varepsilon} w(x_k + x)}
{\displaystyle\sup_{|y| \le \varepsilon} w(x_k - x_0 + y)} 
\frac{\|\chi_{B(x_k, \rho)}\|_{Y(\mathbb{R}^n)}}
{\|\chi_{B(x_k - x_0, \varepsilon)}\|_{Y(\mathbb{R}^n)}} 
\\
&\ge
C(\rho,\varepsilon) 
\inf_{|x| \le \varepsilon, \ |y| \le \varepsilon}\ 
\frac{w(x_k + x)}{w(x_k - x_0 + y)}
\to \infty \quad\mbox{as} \quad k \to \infty.
\end{align*}
Thus, the conditions of Theorem~\ref{th:Lofstrom2} are satisfied
for $X(\mathbb{R}^n)=Y(\mathbb{R}^n,w)$. 

If $a\in\mathcal{M}_{Y(\mathbb{R},w)}$, then 
\[
\|F^{-1}a\ast u\|_{Y(\mathbb{R}^n,w)}
=
\|F^{-1}aFu\|_{Y(\mathbb{R}^n,w)}
\le
\|a\|_{\mathcal{M}_{Y(\mathbb{R}^n,w)}}\|u\|_{Y(\mathbb{R}^n,w)}
\]
for all $u\in C_0^\infty(\mathbb{R}^n)$. By Theorem~\ref{th:Lofstrom2},
$F^{-1}a=c\delta$. Therefore $a\in\mathbb{C}$.
\end{proof}
\begin{corollary}
Let $Y(\mathbb{R})$ be a translation-invariant Banach function space and
$w:\mathbb{R}\to(0,\infty)$ be the weight given by 
\eqref{eq:superexponential-example}. Then 
$\mathcal{M}_{Y(\mathbb{R},w)}=\mathbb{C}$.
\end{corollary}
\begin{proof}
Let $\alpha > 1$ and $\epsilon > 0$. Using the mean value theorem, one gets
\[
(\rho + \epsilon)^\alpha - \rho^\alpha 
\ge 
\alpha \rho^{\alpha - 1}\epsilon \to +\infty \ \mbox{ as } \ 
\rho \to +\infty .
\]
Let $w$ be the weight defined by \eqref{eq:superexponential-example}. Take any 
$x_0 \in \mathbb{R}\setminus\{0\}$. Let $\varepsilon := |x_0|/3$, 
$x_k := (k + 1)x_0$, $k \in \mathbb{N}$. If $x_0 < 0$, then it follows from 
the above that
\begin{align*}
\inf_{|x| \le \varepsilon, \ |y| \le \varepsilon}
\ 
\frac{w(x_k + x)}{w(x_k - x_0 + y)} 
&\ge 
\frac{\exp{\left(\big((k + 1 - 1/3) |x_0|\big)^{\alpha_1}\right)}}
{\exp{\left(\big((k + 1/3) |x_0|\big)^{\alpha_1}\right)}} 
\\
& = 
\exp\left(
\big((k + 2/3)^{\alpha_1} - (k + 1/3)^{\alpha_1}\big) |x_0|^{\alpha_1}
\right)
\to \infty 
\quad \mbox{as} \quad k \to \infty,
\end{align*}
that is, condition \eqref{eq:Lofstrom2-for-weighted-TI-spaces} is satisfied.
Similarly, it can be shown that it is also satisfied if $x_0>0$.
Since $w\in L_{\rm loc}^\infty(\mathbb{R})\subset Y_{\rm loc}(\mathbb{R})$
and $1/w\in L_{\rm loc}^\infty(\mathbb{R})\subset Y_{\rm loc}'(\mathbb{R})$,
it remains to apply Theorem~\ref{th:Lofstrom2-for-weighted-TI-spaces}.
\end{proof}
It might be instructive to contrast the above nonexistence results of 
non-trivial Fourier multipliers with the following statements.
\begin{theorem}\label{th:nontrivial-multipliers}
Let $Y(\mathbb{R}^n)$ be a translation-invariant Banach function space
and $w$ be a weight such that $w \in Y_{\mathrm{loc}}(\mathbb{R}^n)$ and
$1/w  \in Y'_{\mathrm{loc}}(\mathbb{R}^n)$. Suppose there exist $R > 0$, 
$\varepsilon > 0$ and $C_\varepsilon >0$ such that
\begin{equation}\label{eq:nontrivial-multipliers-1}
\frac{w(x + y)}{w(x)} \le C_\varepsilon 
\quad\mbox{for all}\quad 
|x| \ge R, \ |y| \le \varepsilon.
\end{equation} 
Then there exists a constant
$C > 0$ such that for any $\kappa \in L^\infty(\mathbb{R}^n)$
with $\operatorname{supp} \kappa \subseteq B(0, \varepsilon)$ 
and any $f \in Y(\mathbb{R}^n, w)$ one has
\begin{equation}\label{eq:nontrivial-multipliers-2}
\|\kappa\ast f\|_{Y(\mathbb{R}^n, w)} 
\le 
C\|\kappa\|_{L^\infty(\mathbb{R}^n)} 
\|f\|_{Y(\mathbb{R}^n, w)}.
\end{equation}
\end{theorem}
\begin{proof}
Since $w \in Y_{\mathrm{loc}}(\mathbb{R}^n)$ and 
$1/w  \in Y'_{\mathrm{loc}}(\mathbb{R}^n)$, we see that
$Y(\mathbb{R}^n,w)$ is a Banach function space and $Y'(\mathbb{R}^n,w^{-1})$
is its associate space in view of \cite[Lemma~2.4]{KS14}.

If $|x| \le R + \varepsilon$, then
\begin{equation}\label{eq:nontrivial-multipliers-3}
\kappa\ast f(x) 
= 
\int_{\mathbb{R}^n} \kappa(x - y) f(y)\, dy 
= 
\int_{B(0, R + 2\varepsilon)} \kappa(x - y) f(y)\, dy ,
\end{equation}
since $\operatorname{supp}\kappa \subseteq B(0, \varepsilon)$.
 
If $|x| > R + \varepsilon$, then
\begin{align}
\kappa\ast f(x) 
&= 
\int_{\mathbb{R}^n} \kappa(y) f(x - y)\, dy 
= 
\int_{B(0, \varepsilon)} \kappa(y) f(x - y)\, dy 
\nonumber \\
&= 
\int_{B(0, \varepsilon)} 
\kappa(y) \chi_{\mathbb{R}^n\setminus B(0, R)}(x - y) f(x - y)\,dy .
\label{eq:nontrivial-multipliers-4}
\end{align}

Further, axiom (A5) implies the existence of a constant 
$C_{R,\varepsilon} > 0$ such that 
\begin{equation}\label{eq:nontrivial-multipliers-5}
\int_{B(0, R + 2\varepsilon)}  |f(y)|\, dy 
\le 
C_{R,\varepsilon} \|f\|_{Y(\mathbb{R}^n, w)} 
\quad\mbox{for all}\quad
f \in Y(\mathbb{R}^n, w). 
\end{equation}

It is clear that
\begin{equation}\label{eq:nontrivial-multipliers-6}
\|\kappa\ast f\|_{Y(\mathbb{R}^n, w)} 
\le 
\|\chi_{B(0, R + \varepsilon)}\kappa\ast f\|_{Y(\mathbb{R}^n, w)}  
+ 
\|\chi_{\mathbb{R}^n\setminus B(0, R + \varepsilon)}
\kappa\ast f\|_{Y(\mathbb{R}^n, w)} .
\end{equation}
It follows from \eqref{eq:nontrivial-multipliers-3}, 
\eqref{eq:nontrivial-multipliers-5}, and axiom (A4) that
\begin{align}
\|\chi_{B(0, R + \varepsilon)}\kappa\ast f\|_{Y(\mathbb{R}^n, w)} 
&\le
\left\|
\chi_{B(0, R + \varepsilon)}
\|\kappa\|_{L^\infty(\mathbb{R}^n)} 
\|f\|_{L^1(B(0, R + 2\varepsilon))}
\right\|_{Y(\mathbb{R}^n, w)} 
\nonumber \\
&\le 
C_{R,\varepsilon}
\|\kappa\|_{L^\infty(\mathbb{R}^n)} \|f\|_{Y(\mathbb{R}^n, w)} 
\left\|\chi_{B(0, R + \varepsilon)}\right\|_{Y(\mathbb{R}^n, w)}  
\nonumber \\
&=: 
C'_{R,\varepsilon}
\|\kappa\|_{L^\infty(\mathbb{R}^n)} \|f\|_{Y(\mathbb{R}^n, w)} .
\label{eq:nontrivial-multipliers-7}
\end{align}

Taking into account that $Y(\mathbb{R}^n)$ is translation-invariant and using 
\eqref{eq:nontrivial-multipliers-1}, one gets for all $y \in B(0, \varepsilon)$,
\begin{align}
&
\left\|
\tau_y \left(\chi_{\mathbb{R}^n\setminus B(0, R)} f\right)
\right\|_{Y(\mathbb{R}^n, w)}  
=
\left\|
w\tau_y \left(\chi_{\mathbb{R}^n\setminus B(0, R)} f\right)
\right\|_{Y(\mathbb{R}^n)} 
\nonumber\\
&\hspace{1cm}
= 
\left\|
\tau_y \left(
(\tau_{-y}w) \left(\chi_{\mathbb{R}^n\setminus B(0, R)} f\right)
\right)
\right\|_{Y(\mathbb{R}^n)} 
=
\left\|(
\tau_{-y}w) \left(\chi_{\mathbb{R}^n\setminus B(0, R)} f\right)
\right\|_{Y(\mathbb{R}^n)} 
\nonumber \\
&\hspace{1cm}
\le 
C_\varepsilon \left\|
w \left(\chi_{\mathbb{R}^n\setminus B(0, R)} f\right)
\right\|_{Y(\mathbb{R}^n)}  
= 
C_\varepsilon 
\left\|
\chi_{\mathbb{R}^n\setminus B(0, R)} f
\right\|_{Y(\mathbb{R}^n, w)}.
\label{eq:nontrivial-multipliers-8}
\end{align}

Using \eqref{eq:nontrivial-multipliers-4}, \eqref{eq:nontrivial-multipliers-8} 
and H\"older's inequality for Banach function spaces
(see \cite[Chap. 1, Theorem 2.4]{BS88}),
and taking into account that $Y(\mathbb{R}^n)$ is translation-invariant, 
one gets for all  $g \in Y'(\mathbb{R}^n,w^{-1})$,
\begin{align*}
& 
\left|
\int_{\mathbb{R}^n} 
(\chi_{\mathbb{R}^n\setminus B(0, R + \varepsilon)} \kappa\ast f)(x) g(x)\, dx
\right| 
\\
& \hspace{1cm}
\le 
\int_{\mathbb{R}^n} \chi_{\mathbb{R}^n\setminus B(0, R + \varepsilon)}(x)
\left(
\int_{B(0, \varepsilon)} |\kappa(y)| 
\left|\chi_{\mathbb{R}^n\setminus B(0, R)} f\right|(x - y)\, dy
\right) |g(x)|\, dx 
\\
& \hspace{1cm}
\le
\int_{B(0, \varepsilon)} |\kappa(y)| 
\left(
\int_{\mathbb{R}^n} 
\left|\chi_{\mathbb{R}^n\setminus B(0, R)} f\right|(x - y) |g(x)|\, dx
\right)\, dy 
\\
& \hspace{1cm}
\le 
\int_{B(0, \varepsilon)} |\kappa(y)|  
\left\|
\tau_y \left(\chi_{\mathbb{R}^n\setminus B(0, R)} f\right)
\right\|_{Y(\mathbb{R}^n, w)} 
\|g\|_{Y'(\mathbb{R}^n, w^{-1})}\, dy
\\
& \hspace{1cm}
\le 
C_\varepsilon \|\chi_{\mathbb{R}^n\setminus B(0, R)} f\|_{Y(\mathbb{R}^n, w)} 
\left\|g\right\|_{Y'(\mathbb{R}^n, w^{-1})}
\int_{B(0, \varepsilon)} |\kappa(y)| \, dy 
\\
& \hspace{1cm}
\le 
C_\varepsilon \|\kappa\|_{L^1(B(0, \varepsilon))} 
\|f\|_{Y(\mathbb{R}^n, w)} \|g\|_{Y'(\mathbb{R}^n, w^{-1})}.
\end{align*}
By \cite[Chap.~1, Theorem~2.7 and Lemma~2.8]{BS88},
the above inequality implies that
\begin{align}\label{eq:nontrivial-multipliers-9}
& 
\|
\chi_{\mathbb{R}^n\setminus B(0, R + \varepsilon)} \kappa\ast f
\|_{Y(\mathbb{R}^n,w)} 
\nonumber \\
&\hspace{1cm}
=
\sup\left\{\left|
\int_{\mathbb{R}^n} 
(\chi_{\mathbb{R}^n\setminus B(0, R + \varepsilon)} \kappa\ast f)(x) g(x)\, dx
\right|\ : \
g\in Y'(\mathbb{R}^n, w^{-1}), \
\|g\|_{Y'(\mathbb{R}^n, w^{-1})}\le 1
\right\} 
\nonumber \\
&\hspace{1cm}
\le 
C_\varepsilon \|\kappa\|_{L^1(B(0, \varepsilon))} 
\|f\|_{Y(\mathbb{R}^n, w)} 
\le 
C_\varepsilon |B(0, 1)| \varepsilon^n 
\|\kappa\|_{L^\infty(\mathbb{R}^n)} 
\|f\|_{Y(\mathbb{R}^n, w)} .
\end{align}
Combining \eqref{eq:nontrivial-multipliers-6}, 
\eqref{eq:nontrivial-multipliers-7}, and 
\eqref{eq:nontrivial-multipliers-9}, one gets 
\eqref{eq:nontrivial-multipliers-2} with 
$C = C'_{R,\varepsilon} +C_\varepsilon |B(0, 1)| \varepsilon^n$.
\end{proof}
\begin{corollary}
Let $c>0$, $0<\alpha\le 1$, and $w(x)=\exp(c|x|^\alpha)$ for 
$x\in\mathbb{R}^n$. If $Y(\mathbb{R}^n)$ is a translation-invariant 
Banach function space, then there exist non-trivial Fourier multipliers in 
$\mathcal{M}_{Y(\mathbb{R}^n,w)}$.
\end{corollary}
\begin{proof}
Fix $\varepsilon >0$. Then for all $x \in \mathbb{R}^n$ and 
$|y| \le \varepsilon$,
\[
\frac{w(x + y)}{w(x)} 
\le 
\frac{\exp\left(c (|x| + \varepsilon)^\alpha\right)}{\exp\left(c |x|^\alpha\right)}
 =
\exp\left(c \big((|x| + \varepsilon)^\alpha - |x|^\alpha\big)\right) 
\le \exp\left(c M(\alpha, \varepsilon)\right),
\]
where, by the mean value theorem,
\[
M(\alpha, \varepsilon) := 
\max_{0 \le \rho < \infty} 
\left((\rho + \varepsilon)^\alpha - \rho^\alpha\right) 
< +\infty .
\]
Then $w$ satisfies condition \eqref{eq:nontrivial-multipliers-1}.
There exists $j\in\mathbb{N}$ such that the function 
$\varrho_j\in C_0^\infty(\mathbb{R}^n)$ given by \eqref{eq:mollification} 
satisfies $\operatorname{supp}\varrho_j\subseteq B(0,\varepsilon)$.
Put $a:=F\varrho_j$. By Theorem~\ref{th:nontrivial-multipliers}, 
there exists a constant $C>0$ such that for all 
$u\in S(\mathbb{R}^n)\cap Y(\mathbb{R}^n,w)$, one has
\[
\|F^{-1}aFu\|_{Y(\mathbb{R}^n,w)}=\|\varrho_j\ast u\|_{Y(\mathbb{R}^n,w)}
\le 
C\|\varrho_j\|_{L^\infty(\mathbb{R}^n)}\|u\|_{Y(\mathbb{R}^n,w)}.
\]
Therefore, $a\in \mathcal{M}_{Y(\mathbb{R}^n,w)}$.
\end{proof}

\affiliationone{%
Alexei Karlovich\\
Centro de Matem\'atica e Aplica\c{c}\~oes,\\
Departamento de Matem\'atica,\\
Faculdade de Ci\^encias e Tecnologia,\\
Universidade Nova de Lisboa,\\
Quinta da Torre,\\
2829--516 Caparica,\\
Portugal
\email{oyk@fct.unl.pt}}
\affiliationtwo{%
Eugene Shargorodsky\\
Department of Mathematics\\
King's College London\\
Strand, London WC2R 2LS\\
United Kingdom
\email{eugene.shargorodsky@kcl.ac.uk}}
\end{document}